\documentclass[12pt]{amsart}
\usepackage{amsfonts}
\usepackage{amsmath}
\usepackage{amssymb}
\usepackage{mathrsfs}
\usepackage{geometry}
\usepackage{graphicx}
\usepackage{subfigure}
\usepackage{hyperref}
\usepackage{microtype}
\usepackage{geometry}
\usepackage{mathrsfs}
\usepackage{tikz}%% drawing
\usepackage{tikz-cd}%% drawing
\usepackage{color}
\numberwithin{equation}{section}

\newcommand*{\circled}[1]{\lower.7ex\hbox{\tikz\draw (0pt, 0pt)%
    circle (.5em) node {\makebox[1em][c]{\small #1}};}}

\newcommand{\la}{\lambda}
\newcommand{\al}{\alpha}
\newcommand{\be}{\beta}
\newcommand{\ga}{\gamma}
\newcommand{\Ga}{\Gamma}

\newcommand{\ve}{\varepsilon}

\newcommand{\R}{\mathbb{R}}
\newcommand{\N}{\mathbb{N}}
\newcommand{\C}{\mathbb{C}}

\newcommand{\Z}{\mathbb{Z}}
\newcommand{\T}{\mathbb{T}}

\newcommand{\ccc}{\cdot\cdot\cdot}

\newcommand{\n}[1]{\Vert #1\Vert }

\newcommand{\bbn}[1]{\Big\Vert #1 \Big \Vert }
\newcommand{\lr}[1]{\left\{ #1\right\} }
\newcommand{\lrc}[1]{\left[ #1\right] }
\newcommand{\lrs}[1]{\left( #1\right) }

\newcommand{\lra}[1]{\langle #1\rangle}

\newcommand{\babs}[1]{\big | #1 \big| }
\newcommand{\bbabs}[1]{\Big | #1 \Big| }
\newcommand{\wt}[1]{\widetilde{#1} }
\newcommand{\wq}{\infty}
\newcommand{\pa}{\partial}

\newcommand{\ch}{{\mathcal H}}

\newcommand{\cl}{{\mathcal L}}

\newcommand{\ol}{\overline}

\begin{document}

\newtheorem{theorem}{Theorem}[section]
\newtheorem{lemma}[theorem]{Lemma}

\theoremstyle{definition}
\newtheorem{definition}[theorem]{Definition}
\newtheorem{example}[theorem]{Example}
\newtheorem{remark}[theorem]{Remark}

\numberwithin{equation}{section}

\newtheorem{proposition}[theorem]{Proposition}
\newtheorem{corollary}[theorem]{Corollary}
\newtheorem{goal}[theorem]{Goal}
\newtheorem{algorithm}{Algorithm}

\renewcommand{\figurename}{Fig.}

\title[Unconditional uniqueness for the $H^{1}$-supercritical NLS]{The unconditional uniqueness for the energy-supercritical NLS}
\author[X. Chen]{Xuwen Chen}
\address{Department of Mathematics, University of Rochester, Rochester, NY 14627, USA}

\email{xuwenmath@gmail.com }
%chenxuwen@math.umd.edu

\author[S. Shen]{Shunlin Shen}
\address{School of Mathematical Sciences, Peking University, Beijing, 100871, China \& Department of Mathematics, University of Rochester, Rochester, NY 14627, USA}
\email{slshen100871@gmail.com; slshen@pku.edu.cn}

\author[Z. Zhang]{Zhifei Zhang}
\address{School of Mathematical Sciences, Peking University, Beijing, 100871, China}

\email{zfzhang@math.pku.edu.cn}

\subjclass[2010]{Primary 35Q55, 35A02, 81V70.}
\date{}

\dedicatory{}

%    Abstract is required.
\begin{abstract}

We consider the cubic and quintic nonlinear Schr\"{o}dinger equations (NLS) under the $\mathbb{R}^{d}$ and $\mathbb{T}^{d}$ energy-supercritical setting. Via a newly developed unified scheme, we prove the unconditional uniqueness for solutions to NLS at critical regularity for all dimensions. Thus, together with \cite{chen2019the,chen2020unconditional}, the unconditional uniqueness problems for $H^{1}$-critical and $H^{1}$-supercritical cubic and quintic NLS are completely and uniformly resolved at critical regularity for these domains.
One application of our theorem is to prove that defocusing blowup solutions of the type in \cite{merle2019blow} are the only possible $C([0,T);\dot{H}^{s_{c}})$ solutions if exist in these domains.

 \end{abstract}
\keywords{Energy-supercritical NLS, Gross-Pitaevskii Hierarchy, Klainerman-Machedon Board
Game, Multilinear Estimates}
\maketitle
\tableofcontents

\section{Introduction}
We consider the nonlinear Schr\"{o}dinger equation (NLS)
\begin{equation}\label{equ:NLS}
\begin{cases}
&i\pa_{t}u=-\Delta u\pm |u|^{p-1}u,\quad (t,x)\in [0,T]\times \Lambda^{d}\\
&u(0,x)=u_{0}(x)
\end{cases}
\end{equation}
where $\Lambda^{d}=\R^{d}$ or $\T^{d}$ and $\pm$ denotes defocusing/focusing.
In Euclidean spaces, the NLS $(\ref{equ:NLS})$ enjoys the scaling invariance
\begin{align}\label{equ:nls, scaling}
u_{\la}(t,x)=\la^{\frac{2}{p-1}}u(\la^{2}t,\la x),\quad \la>0.
\end{align}
which preserves the homogeneous Sobolev norm $\n{u_{0}}_{\dot{H}^{s_{c}}}$ where the critical scaling exponent is given by
\begin{align}\label{equ:scaling exponent}
s_{c}:=\frac{d}{2}-\frac{2}{p-1}.
\end{align}
 Accordingly, the initial value problem $(\ref{equ:NLS})$ for $u_{0}\in \dot{H}^{s_{c}}$ can be classified as energy subcritical, critical or
supercritical depending on whether the critical Sobolev exponent $s_{c}$ lies below, equal to
or above the energy exponent $s=1$.

In this paper, we focus on the cubic and quintic cases under the energy-supercritical setting $(s_{c}>1)$ where
\begin{equation}
s_{c}=\begin{cases}
\frac{d-2}{2}\quad &for\ d\geq 5,\ p=3,\\
\frac{d-1}{2}\quad &for\ d\geq 4,\ p=5.
\end{cases}
\end{equation}
In the energy-supercritical setting, the global well-posedness of $(\ref{equ:NLS})$ is fully open, away from the classical local well-posedness and $L^{2}$-supercritical blowup results \cite{cazenave1990the,glassey1977on}.
  But it has been, for a long time, believed that, even under the energy-supercritical setting, the defocusing version of $(\ref{equ:NLS})$ is globally well-posed and the solution scatters when $\Lambda=\R$, just like the energy-critical and subcritical cases, especially after the breakthrough
  \cite{bourgain1999global,colliander2008global,grillakis2000on,kenig2006global,ryckman2007global}\footnote{See \cite{dodson2019defocusing} for a more detailed survey.}
  on the $\R^{d}$ energy-critical cubic and quintic cases. (See, for example \cite{killip2010energy}.) Surprisingly, the recent work \cite{merle2019blow} unexpectedly constructed the first instance of finite time blowup solution for the defocusing energy-supercritical NLS.
Thus it is of interest to know if there could exist a scattering global solution in $\dot{H}^{s_{c}}$ but may not be in $C([0,T);\dot{H}^{s_{c}})\bigcap L_{t,x}^{p(d+2)/2}$ when blowups of this type exist.

 There are certainly multiple routes for such a problem. But one way is the classical unconditional uniqueness theorem in $\dot{H}^{s_{c}}$ which itself has remained open at least for $\T^{d}$. With an unconditional uniqueness result, we know that there could be at most one solution in $C([0,T);\dot{H}^{s_{c}})$ regardless of auxiliary spaces. One application is to prove that blowup solutions of the type in \cite{merle2019blow} is the only possible $C([0,T);\dot{H}^{s_{c}})$ solution if exist in these domains. In this paper, we prove the $\dot{H}^{s_{c}}$ unconditional uniqueness for $(\ref{equ:NLS})$ as follows and address this issue.

\begin{theorem}\footnote{One could extend the domain $\Lambda^{d}$ to more general manifolds, as long as the multilinear estimates which relies on Fourier analysis and Strichartz estimates in Section \ref{section:Multilinear Estimates} hold.}\label{thm:uniqueness for nls}
Let $s_{c}>1$ and $p=3$ or $5$. There is at most one $C([0,T_{0}];\dot{H}^{s_{c}}(\Lambda^{d}))$\footnote{We consider $H^{s_{c}}$ for the $\T^{d}$ case and $\dot{H}^{s_{c}}$ for the $\R^{d}$ case as $\dot{H}^{s_{c}}$ does not generate much differences for the $\mathbb{T}^{d}$ case.} solution to $(\ref{equ:NLS})$.
\end{theorem}

The fundamental concept of unconditional uniqueness was first raised by Kato in \cite{kato1995on,kato1996correction} when proving well-posedness in Strichartz type spaces had made vast progress.
In $\R^{d}$, these unconditional uniqueness problems at critical regularity are usually proved by showing any solution must agree with the Strichartz solution, if exists, using the inhomogeneous (retarded) Strichartz estimate.
Such a method has been proven to be successful even in the $\R^{3}$ quintic energy-critical case,
see for example \cite{colliander2008global}. (This is a very active field, see for example \cite{babin2011on,furioli2003unconditional,guo2013normal,kishimoto2019unconditional,
kwon2020normal,molinet2018unconditional,mosincat2020unconditional,
zhou1997uniqueness} and the reference within for work on other dispersive equations along this line.)

However, such arguments for the Euclidean setting are no longer effective if $(\ref{equ:NLS})$ is posed on $\T^{d}$, as the Strichartz estimate is rather weak in the
   periodic case. The $L_{x}^{2}$ Strichartz estimate does not hold in the periodic case and hence the dual Strichartz estimate also fails. On the other hand, the well-posedness on $\T^{d}$ is more intricate, such as using the $X_{s,b}$ space \cite{bourgain1993fourier} and the atomic $U^{p}$ and $V^{p}$ spaces \cite{herr2011global,ionescu2012the}. Thus the unconditional uniqueness
problems on $\T^{d}$ under the critical setting are much more difficult to handle. Nevertheless,
 a unified method has recently unexpectedly arisen from the study of the derivation of (\ref{equ:NLS}) on the $\T^{d}$ case in \cite{herr2019unconditional} and under the energy-critical setting in \cite{chen2019the,chen2020unconditional}.\footnote{We mention \cite{herr2019unconditional} 1st here and in the related places in the rest of the paper. Even though
 \cite{chen2019the} was posted on arXiv one month before
 \cite{herr2019unconditional}, X. Chen and Holmer were not
 aware of the unconditional uniqueness implication of \cite{chen2019the}
 until \cite{herr2019unconditional}.}
%under the energy-critical setting .

We find that one could use the scheme of \cite{chen2020unconditional} to perfectly solve the unconditional uniqueness problem under the energy-supercritical setting for both $\R^{d}$ and $\T^{d}$.
The proof comes from the Gross-Pitaevskii(GP) hierarchy, which seems to be weaker than the
 NLS analysis, as it originates from the derivation of NLS. However, we will see that such an argument is also powerful and worthy for further study. Here, we focus on the quintic GP hierarchy, also see \cite{chen2020unconditional} for the cubic case. The quintic GP hierarchy is a sequence $\lr{\ga^{(k)}(t)}_{k=1}^{\wq}$ which satisfies the infinitely coupled
hierarchy of equations:
\begin{equation}\label{equ:gp hierarchy,quintic}
i\pa_{t}\ga^{(k)}=\sum_{j=1}^{k}[-\Delta_{x_{j}},\ga^{(k)}]
\pm b_{0}\sum_{j=1}^{k}\operatorname{Tr}_{k+1,k+2}[\delta(x_{j}-x_{k+1})\delta(x_{j}-x_{k+2}),\ga^{(k+2)}]
\end{equation}
where $b_{0}$ is some coupling constant, $\pm$ denotes defocusing/focusing. Given any solution $u$
of $(\ref{equ:NLS})$, it generates a solution to $(\ref{equ:gp hierarchy,quintic})$ by letting
\begin{equation}\label{equ:factorized solution}
\ga^{(k)}=|u\rangle \langle u|^{\otimes k}
\end{equation}
in operator form or
$$\ga^{(k)}(t,\mathbf{x}_{k};\mathbf{x}_{k}')=\prod_{j=1}^{k}u(t,x_{j})\ol{u}(t,x_{j}')$$
in kernel form where $\mathbf{x}_{k}=(x_{1},...,x_{k})$.

The hierarchy approach was first suggested by Spohn \cite{spohn1980kinetic} for the derivation of NLS from quantum many-body dynamic. Around 2005, it was Erd{\"o}s, Schlein, and Yau who first rigorously derived the 3D cubic defocusing NLS from a 3D quantum many-body dynamic in their fundamental papers \cite{erdos2006derivation,erdos2007derivation,erdos2007rigorous,erdos2009rigorous,erdos2010derivation}. The proof for the uniqueness of the GP
hierarchy was the principal part and also surprisingly dedicate due to the fact that it is a system of infinitely
many coupled equations over an unbounded number of variables. With a sophisticated Feynman graph analysis in \cite{erdos2007derivation}, they proved the $H^{1}$-type unconditional uniqueness of the $\R^{3}$ cubic GP hierarchy. The first series of ground breaking papers have motivated a large amount of work.

Subsequently in 2007, with imposing an additional a-prior condition on space-time norm, Klainerman and Machedon \cite{klainerman2008on}, inspired by \cite{erdos2007derivation,klainerman1993space}, gave an another proof of the uniqueness of the GP hierarchy in a different space of density matrices defined by Strichartz type norms. They provided a different combinatorial argument, the now so-called Klainerman-Machedon (KM) board game argument, to combine the inhomogeneous terms effectively reducing their numbers and then derived a space-time estimate to control these terms.
At that time, it was open to prove that the limits coming from the $N$-body dynamics satisfy the space-time bound. Nonetheless, \cite{klainerman2008on} has made the delicate analysis of the GP hierarchy approachable from the perspective of PDE.
Later, Kirkpatrick, Schlein, and Staffilani \cite{kirkpatrick2011derivation} obtained the KM space-time bound via a simple trace theorem in both $\R^{2}$ and $\T^{2}$ and derived the 2D cubic defocusing NLS from the 2D quantum many-body dynamic. Such a scheme also motivated many works \cite{chen2011the,chen2016collapsing,chen2016focusing,chen2017focusing,gressman2014on,shen2021the,sohinger2016local,xie2015derivation} for the uniqueness of GP hierarchies.

Later in 2008, T. Chen and Pavlovi{\'c} \cite{chen2011the} initiated the study of the quintic GP hierarchy and provided a proof for the quintic KM board game argument, which laid the foundation for the further study of the quintic GP hierarchy. They also showed that the 2D quintic case, which is usually
considered the same as the 3D cubic case since they share the same scaling criticality, satisfied the
KM space-time bound while it was still open for the 3D cubic case at that time. To attack the problem, they also considered the well-posedness theory with more general data in \cite{chen2010on,chen2013a,chen2014higher}. (See also \cite{chen2010energy,mendelson2019poisson,mendelson2020a,mendelson2019an,sohinger2016local,sohinger2015randomization}).
 Then in 2011, they proved that the 3D cubic
KM space-time bound holds for the defocusing $\be<1/4$ case in \cite{chen2014derivation}.
The result was quickly improved to $\be<2/7$ by X. Chen in \cite{chen2013on} and then extended to the almost optimal case, $\be<1$, by X. Chen and Holmer in \cite{chen2016correlation,chen2016on}.
Around the same period of time, Gressman, Sohinger, and Staffilani \cite{gressman2014on} studied the
uniqueness of the GP hierarchy on $\T^{3}$ and proved that the sharp space-time estimate on $\T^{3}$ needed an additional $\ve$ derivatives than the $\R^{3}$ setting in which one derivative is needed. Later, Herr and
Sohinger generalized this fact to more general cases in \cite{herr2016the}.

In 2013, by introducing quantum de Finetti theorem from \cite{lewin2014derivation}, T. Chen, Hainzl, Pavlovi{\'c} and Seiringer \cite{chen2015unconditional} provided a simplified proof of the $L_{T}^{\infty}H_{x}^{1}$-type 3D cubic uniqueness theorem in \cite{erdos2007derivation}.
With the quantum de Finetti theorem, one can replace the space-time estimates by Sobolev multilinear estimates.
The scheme in \cite{chen2015unconditional}, which consists of the KM board game argument, the quantum de Finetti theorem and the Sobolev multilinear estimates, is
robust to deal with such uniqueness problems.
Following the scheme in \cite{chen2015unconditional}, Sohinger \cite{sohinger2015a} solved the aforementioned $\ve$-loss problem
for the defocusing $\T^{3}$ cubic case. In \cite{hong2015unconditional}, Hong, Taliaferro, and Xie used the scheme to obtain unconditional uniqueness theorems in $\R^{d}$, $d=1,2,3$, with regularities matching
the NLS analysis. Then in \cite{hong2016uniqueness}, they proved $H^{1}$ small solution uniqueness for the $\R^{3}$ quintic case. For other refined uniqueness theorems, see also \cite{chen2014on}.

The uniqueness analysis of GP hierarchy started to unexpectedly yield new NLS results with regularity lower than
the NLS analysis all of a sudden since \cite{herr2019unconditional} and \cite{chen2019the,chen2020unconditional}.
In \cite{herr2019unconditional}, with the scheme in \cite{chen2015unconditional}, Herr and Sohinger discovered new unconditional uniqueness results for the cubic NLS on $\T^{d}$, which covered the full scaling-subcritical regime for $d\geq 4$. (See also the later work \cite{kishimoto2021unconditional} using NLS analysis.)

On the other hand, the $\T^{3}$ quintic energy-critical
case at $H^{1}$ regularity was not known until recently
 \cite{chen2019the}. By discovering the new hierarchical uniform frequency localization (HUFL) property for the GP hierarchy, X. Chen and Holmer
established a new $H^{1}$-type uniqueness theorem for the $\T^{3}$ quintic energy-critical GP hierarchy. The new uniqueness theorem, though neither conditional nor unconditional for the GP
hierarchy implies the $H^{1}$ unconditional uniqueness result for the $\T^{3}$ quintic energy-critical NLS. Then in \cite{chen2020unconditional}, they proved the unconditional uniqueness for the $\T^{4}$ cubic energy-critical case by working out new combinatorics and extending the KM
board game argument. As the previously used Sobolev multilinear estimates fail on $\T^{4}$, they develop the new combinatorics which enable the application of $U$-$V$ multilinear estimates, which is indeed weaker than Sobolev multilinear estimates. The scheme in \cite{chen2020unconditional}, which effectively combines the quantum de Finetti theorem, the $U$-$V$ space techniques, the multilinear estimates proved by using the scale invariant Strichartz estimates / $l^{2}$-decoupling theorem and the HUFL properties, provides a unified proof of the large solution uniqueness.

\section{Proof of the Main Theorem}  \label{proof of the main theorem}

\subsection{Outline of the Proof}
Our proof will focus on the $\T^{d}$ case, as it works the same for $\R^{d}$.\footnote{By using the classical methods, we also give a more usual proof for the $\R^{d}$ case at the appendix.} Our argument follows the scheme of \cite{chen2020unconditional} where an extended version of KM board game argument which is compatible with $U$-$V$ estimates was discovered. We summarize our proof below, especially for the quintic case.

To conclude the uniqueness for NLS $(\ref{equ:NLS})$, one usually proves that
$$w(t)=u_{1}(t,x)-u_{2}(t,x)\equiv 0$$
where $u_{1}$ and $u_{2}$ are two solutions to $(\ref{equ:NLS})$ with the same initial datum.
Instead, we turn to prove that
\begin{align}\label{equ:difference of two solutions}
\ga^{(k)}(t):=\prod_{j=1}^{k}u_{1}(t,x_{j})\ol{u}_{1}(t,x_{j}')-\prod_{j=1}^{k}u_{2}(t,x_{j})\ol{u}_{2}(t,x_{j}'),
\end{align}
which is a solution to $(\ref{equ:gp hierarchy,quintic})$ with zero initial datum, vanishes identically on $[0,T_{0}]$.
The formulation $(\ref{equ:difference of two solutions})$ endows the NLS (\ref{equ:NLS}) with an extra linear structure via the GP hierarchy so that one could
iteratively use multilinear estimates to yield smallness, instead of constructing a closed inequality in some Strichartz space.

Hence, we first prove Theorem \ref{thm:uniqueness for gp hierarchy}, which is a uniqueness theorem for the GP hierarchy, and then Theorem \ref{thm:uniqueness for nls} comes as a corollary of Theorem \ref{thm:uniqueness for gp hierarchy} and Lemma \ref{lemma:UTFL}.
As Theorem \ref{thm:uniqueness for gp hierarchy} requires the uniform in time frequency localization (UTFL) condition, we prove that every $C([0,T_{0}];H^{s_{c}})$ solution to $(\ref{equ:NLS})$ satisfies UTFL condition by Lemma $\ref{lemma:UTFL}$. Thus we would have established Theorem $\ref{thm:uniqueness for nls}$ once we have proved Theorem $\ref{thm:uniqueness for gp hierarchy}$.

 The GP hierarchy argument does not require the dual Strichartz estimate or the existence of a Strichartz solution. However, we have to carefully combine and estimate the $(2k-1)!!2^{k}$ summands in iterated Duhamel expansions. More precisely,
\begin{align}\label{equ:duhamel expansion, summands, outline}
\ga^{(1)}(t_{1})=\sum_{(\mu,sgn) }\int_{t_{1}\geq t_{3}\geq \ccc\geq t_{2k+1}} J_{\mu,sgn}^{(2k+1)}(\ga^{(2k+1)})(t_{1},\underline{t}_{2k+1})d\underline{t}_{2k+1}
\end{align}
with $\underline{t}_{2k+1}=(t_{3},t_{5},...,t_{2k+1})$ and
\begin{align} \label{equ:duhamel expansion, section 2}
J_{\mu,sgn}^{(2k+1)}(\ga^{(2k+1)})(t_{1},\underline{t}_{2k+1})=
&U^{(1)}(t_{1}-t_{3})B_{\mu(2);2,3}^{sgn(2)}U^{(3)}(t_{3}-t_{5})
B_{\mu(4);4,5}^{sgn(4)}\\
&\ccc U^{(2k-1)}(t_{2k-1}-t_{2k+1})B_{\mu(2k);2k,2k+1}^{sgn(2k)}\ga^{(2k+1)}(t_{2k+1})\notag
\end{align}
where $U^{(2j+1)}(t)$ is the propagator and $B_{i;2j,2j+1}^{\pm}$ is the collapsing operator, and thus there are $(2k-1)!!2^{k}$ terms in $\ga^{(1)}(t_{1})$. Hence handling the $(2k-1)!!2^{k}$ terms in the critical setting is the main difficulty. Now, we divide the proof of Theorem \ref{thm:uniqueness for gp hierarchy} into two main parts.

 \textbf{Part 1: The Estimate Part.}
  Usually, one
employs the original KM board game argument to sort the $(2k-1)!!2^{k}$ summands of $\ga^{(1)}$ into a sum of KM upper echelon forms with a time integration domain, which is a union of a very large number of high dimensional simplexes. As Sobolev type multilinear estimates work regardless of the time integration domain, one can iteratively use them to yield smallness.
Nevertheless, if we have some combinatorics which is compatible with space-time type multilinear estimates, we could exploit the multilinear estimates in $U$-$V$ spaces. Indeed, based on the combinatorics part (Part 2),
we can write
 \begin{align}\label{equ:duhamel expansion, reference form, section 2}
\ga^{(1)}(t_{1})=\sum_{ reference\ (\hat{\mu},\hat{sgn})}
\int_{T_{C}(\hat{\mu},\hat{sgn})}J_{\hat{\mu},\hat{sgn}}^{(2k+1)}(\ga^{(2k+1)})
(t_{1},\underline{t}_{2k+1})d\underline{t}_{2k+1}
\end{align}
 where the number of reference pairs in Definition \ref{def:reference pair} can be controlled by $16^{k}$, which is substantially smaller than the original $(2k-1)!!2^{k}$. More importantly, the time integration domain $T_{C}(\hat{\mu},\hat{sgn})$ is compatible with space-time type multilinear estimates. Hence it comes down to how to estimate
 \begin{align} \label{equ:duhamel expansion with the compatible time integration domain,section 2}
\int_{T_{C}(\mu,sgn)}J_{\mu,sgn}^{(2k+1)}(\ga^{(2k+1)})
(t_{1},\underline{t}_{2k+1})d\underline{t}_{2k+1}.
\end{align}

  In Section \ref{subsection:Duhamel Expansion and Duhamel Tree}, we start with a Duhamel tree to represent the Duhamel expansion. Then in Section \ref{subsection:Compatible Time Integration Domain}, we introduce the time integration domain $T_{C}(\mu,sgn)$ which is compatible with space-time type multilinear estimates. Subsequently in Section \ref{subsection:Estimates using the $U$-$V$ multilinear estimates}, after giving a short introduction to $U$-$V$ spaces, we show how to apply the $U$-$V$ multilinear estimates to estimate $(\ref{equ:duhamel expansion with the compatible time integration domain,section 2})$.
We will use the following $U$-$V$ multilinear estimates.
\begin{align}
\bbn{\int_{a}^{t}e^{i(t-s)\Delta}(\wt{u}_{1}\wt{u}_{2}\wt{u}_{3}\wt{u}_{4}\wt{u}_{5})ds}_{X^{s}}
\leq &C \n{u_{1}}_{X^{s}}
\n{u_{2}}_{X^{s_{c}}}\n{u_{3}}_{X^{s_{c}}}
\n{u_{4}}_{X^{s_{c}}}\n{u_{5}}_{X^{s_{c}}},
\end{align}
\begin{align} \label{equ:multilinear estimate, low frequency, low regulariy,section 2}
&\bbn{\int_{a}^{t}e^{i(t-\tau)\Delta}(\wt{u}_{1}\wt{u}_{2}\wt{u}_{3}\wt{u}_{4}\wt{u}_{5})d\tau}_{X^{s}}\\
\leq& C\n{u_{1}}_{X^{s}}
\lrs{T^{\frac{1}{2(d+3)}}M_{0}^{\frac{2d+3}{3(d+3)}}\n{P_{\leq M_{0}}u_{2}}_{X^{s_{c}}}+\n{P_{>M_{0}}u_{2}}_{X^{s_{c}}}}\n{u_{3}}_{X^{s_{c}}}
\n{u_{4}}_{X^{s_{c}}}\n{u_{5}}_{X^{s_{c}}},\notag
\end{align}
where $\wt{u}\in \lr{u,\ol{u}}$ and $s\in\lr{s_{c},s_{c}-2}$. The proof highly relies on the scale invariant Strichartz estimates / $l^{2}$-decoupling theorem \cite{bourgain2015the,killip2016scale} and hence is postponed to Section \ref{section:Multilinear Estimates}. Compared with Sobolev multilinear estimates, the proof of $U$-$V$ multilinear estimates is simpler and less technical. (Although the representation $(\ref{equ:duhamel expansion, reference form, section 2})$ is also compatible with $X_{s,b}$ multilinear estimates, they need an additional $\ve$ derivatives in time and hence cannot be used to deal with the critical problem.) On the one hand, to prove Sobolev multilinear estimates, the $L_{T}^{\wq}H^{-s}$ space, which is usually used in the duality argument, is an endpoint case in Littlewood-Paley theory and does not carry any Strichartz regularity. On the other hand, $U$-$V$ techniques have been proven to be successful and adaptive for NLS in many different general domains.

Together with assumptions in Theorem \ref{thm:uniqueness for gp hierarchy}, we are able to prove the following key estimate.
\begin{align} \label{equ:key estimate, section 2}
\bbn{\lra{\nabla_{x_{1}}}^{s_{c}-2}\lra{\nabla_{x_{1}'}}^{s_{c}-2}
\int_{T_{C}(\hat{\mu},\hat{sgn})}J_{\hat{\mu},\hat{sgn}}^{(2k+1)}(\ga^{(2k+1)})
(t_{1},\underline{t}_{2k+1})d\underline{t}_{2k+1}}_{L_{T}^{\wq}L_{x_{1}}^{2}L_{x_{1}'}^{2}}
\leq \delta^{k}
\end{align}
where $\delta(T,\ve,C_{0},M_{0})$ can be sufficiently small by properly choosing these parameters. More specifically, the smallness comes from the UTFL property, that is,
\begin{equation*}
\n{\lra{\nabla}^{s_{c}}P_{> M(\ve)}u}_{L_{[0,T_{0}]}^{\wq}L_{x}^{2}}\leq \ve.
\end{equation*}
By iteratively using $(\ref{equ:multilinear estimate, low frequency, low regulariy,section 2})$ at least $\frac{4}{5}k$ times, we obtain the factor of smallness as follows
$$\lrs{T^{\frac{1}{2(d+3)}}M_{0}^{\frac{2d+3}{2(d+3)}}C_{0}+\ve}^{\frac{4}{5}k}.$$

Thus, we are left to prove the representation $(\ref{equ:duhamel expansion, reference form, section 2})$, especially, the compatibility part.

 \textbf{Part 2: The Combinatorics Part.}
In Section \ref{section:Existence of Compatible Time Integration Domain}, by working out
new combinatorics to reconstruct the quintic KM board game argument from the ground up, we could represent $\ga^{(1)}$ in the form of $(\ref{equ:duhamel expansion, reference form, section 2})$, which is compatible with the $U$-$V$ multilinear estimates. The combinatorics analysis is independent of the multilinear estimates or the regularity settings, so it could be applied for more general cases. Such a representation $(\ref{equ:duhamel expansion, reference form, section 2})$, which enables the application of $U$-$V$ multilinear estimates, would also be helpful for further study of GP hierarchy.

In Section \ref{subsection:Admissible Tree}, we first give a brief review of the
quintic KM board game argument as in \cite{chen2011the,klainerman2008on}. Then, we give an introduction to an admissible tree diagram representation used to represent collapsing map pairs.
For example\footnote{We will not use  this example again in the paper, as
we can generate as many as we want.}, given a collapsing map pair $(\mu_{1},sgn_{1})$ as follows,
$$
\begin{tabular}{c|ccccc}
$2j$&2&4&6&8&10\\
\hline
$\mu_{1}$ &1&1&1&3&6\\
$sgn_{1}$ &+&+&$-$&$-$&+
\end{tabular}
$$
we generate the following trees in turn by Algorithm \ref{algorithm:generate an admissible tree}.

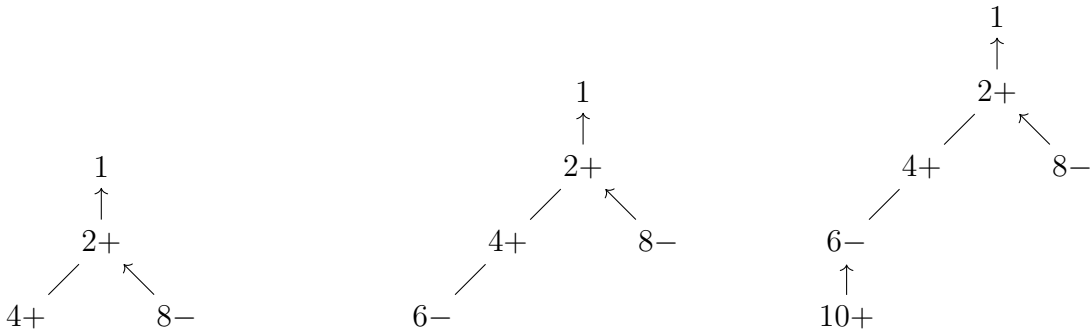
\begin{minipage}[t]{0.32\textwidth}
\begin{tikzpicture}
\node (1) at (0,0) {1};
\node (2) at (0,-1) {2+};
\node (4) at (-1,-2) {4+};
\node (8) at (1,-2) {8$-$};
\draw[<-] (1)--(2);
\draw[-] (2)--(4);
\draw[<-] (2)--(8);
\end{tikzpicture}
\end{minipage}
\begin{minipage}[t]{0.32\textwidth}
\begin{tikzpicture}
\node (1) at (0,0) {1};
\node (2) at (0,-1) {2+};
\node (4) at (-1,-2) {4+};
\node (8) at (1,-2) {8$-$};
\node (6) at (-2,-3) {6$-$};
\draw[<-] (1)--(2);
\draw[-] (2)--(4);
\draw[<-] (2)--(8);
\draw[-] (4)--(6);
\end{tikzpicture}
\end{minipage}
\begin{minipage}[t]{0.32\textwidth}
\begin{tikzpicture}
\node (1) at (0,0) {1};
\node (2) at (0,-1) {2+};
\node (4) at (-1,-2) {4+};
\node (8) at (1,-2) {8$-$};
\node (6) at (-2,-3) {6$-$};
\node (10) at (-2,-4) {10+};
\draw[<-] (1)--(2);
\draw[-] (2)--(4);
\draw[<-] (2)--(8);
\draw[-] (4)--(6);
\draw[<-] (6)--(10);
\end{tikzpicture}
\end{minipage}
~\\
Such a tree diagram reprentation could provide a more elaborated proof of the original quintic KM board game argument as well.

In Section $\ref{subsection:Signed KM Acceptable Moves}$, we give the signed Klainerman-Machedon acceptable moves in Chen-Pavlovi\'{c} format (signed acceptable moves), which sorts $(2k-1)!!2^{k}$ collapsing map pairs $(\mu,sgn)$ into various equivalence classes, the number of which can be controlled by $16^{k}$. Moreover, the signed acceptable moves preserve the signed tree structures. Here are all the collapsing map pairs and the corresponding trees equivalent to $(\mu_{1},sgn_{1})$.

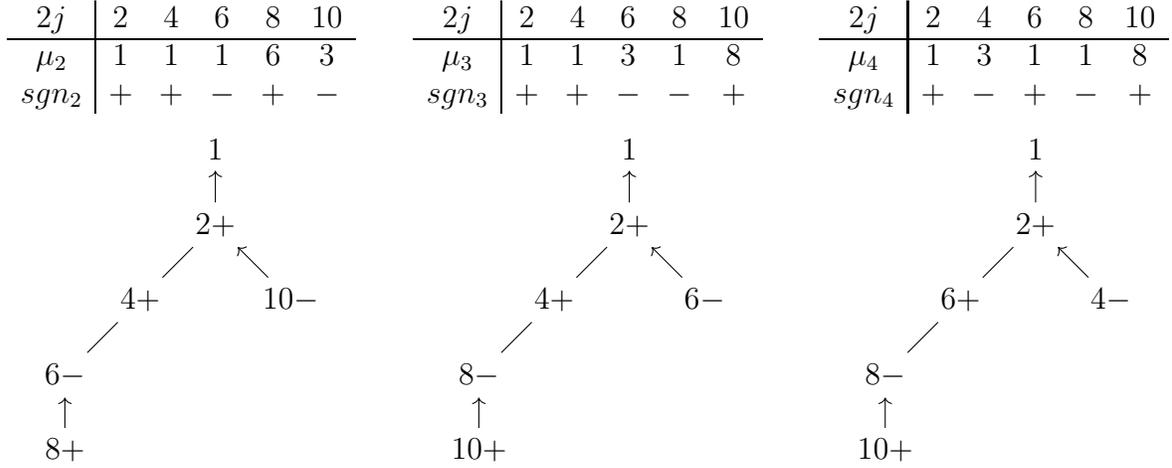
\begin{minipage}{0.32\textwidth}
\centering
$$
\begin{tabular}{c|ccccc}
$2j$&2&4&6&8&10\\
\hline
$\mu_{2}$ &1&1&1&6&3\\
$sgn_{2}$ &+&+&$-$&+&$-$
\end{tabular}
$$

\begin{tikzpicture}
\node (1) at (0,0) {1};
\node (2) at (0,-1) {2+};
\node (4) at (-1,-2) {4+};
\node (10) at (1,-2) {10$-$};
\node (6) at (-2,-3) {6$-$};
\node (8) at (-2,-4) {8+};
\draw[<-] (1)--(2);
\draw[-] (2)--(4);
\draw[<-] (2)--(10);
\draw[-] (4)--(6);
\draw[<-] (6)--(8);
\end{tikzpicture}
\end{minipage}
\begin{minipage}{0.32\textwidth}
\centering
$$
\begin{tabular}{c|ccccc}
$2j$&2&4&6&8&10\\
\hline
$\mu_{3}$ &1&1&3&1&8\\
$sgn_{3}$ &+&+&$-$&$-$&+
\end{tabular}
$$

\begin{tikzpicture}
\node (1) at (0,0) {1};
\node (2) at (0,-1) {2+};
\node (4) at (-1,-2) {4+};
\node (6) at (1,-2) {6$-$};
\node (8) at (-2,-3) {8$-$};
\node (10) at (-2,-4) {10+};
\draw[<-] (1)--(2);
\draw[-] (2)--(4);
\draw[<-] (2)--(6);
\draw[-] (4)--(8);
\draw[<-] (8)--(10);
\end{tikzpicture}
\end{minipage}
\begin{minipage}{0.32\textwidth}
\centering
$$
\begin{tabular}{c|ccccc}
$2j$&2&4&6&8&10\\
\hline
$\mu_{4}$ &1&3&1&1&8\\
$sgn_{4}$ &+&$-$&+&$-$&+
\end{tabular}
$$

\begin{tikzpicture}
\node (1) at (0,0) {1};
\node (2) at (0,-1) {2+};
\node (6) at (-1,-2) {6+};
\node (4) at (1,-2) {4$-$};
\node (8) at (-2,-3) {8$-$};
\node (10) at (-2,-4) {10+};
\draw[<-] (1)--(2);
\draw[-] (2)--(6);
\draw[<-] (2)--(4);
\draw[-] (6)--(8);
\draw[<-] (8)--(10);
\end{tikzpicture}

\end{minipage}
~\\
(Notice that the above trees have the same skeleton.)

However, extending to signed move is not sufficient for our proposes. To be compatible with the $U$-$V$ multilinear estimates, we have to further combine the integrals and enlarge the time integration domain. For this purpose, we need the wild moves defined by $(\ref{definiton:wild moves})$. The wild moves, unlike the signed acceptable moves, do change the tree structure. However, KM upper echelon forms are not invariant under the wild moves.

Thus in Section $\ref{subsection:Tamed Form}$, we prove that there exists a unique special form, for which we call the tamed form, in each equivalent class and hence arrive at
\begin{equation} \label{equ:integration,tamed form, time integration domain, section 2}
\ga^{(1)}(t_{1})=\sum_{(\mu_{*},sgn_{*})\ \text{tamed}}\int_{T_{D}(\mu_{*})}J_{\mu_{*},sgn_{*}}^{(2k+1)}(\ga^{(2k+1)})(t_{1},\underline{t}_{2k+1})
d\underline{t}_{2k+1}
\end{equation}
where $T_{D}(\mu_{*})$ can be read out from the corresponding tree. Here, $(\mu_{1},sgn_{1})$ is the unique tamed form in the equivalent class $\lr{(\mu_{i},sgn_{i})}_{i=1}^{4}$.

Subsequently in Section \ref{subsection:Wild Moves},we exploit wild moves for a further reduction of $(\ref{equ:integration,tamed form, time integration domain, section 2})$, which keeps the tamed form invariant, to sort the tamed forms into tamed classes. All the tamed forms and trees wildly relatable to $(\mu_{1},sgn_{1})$ are shown as follows.

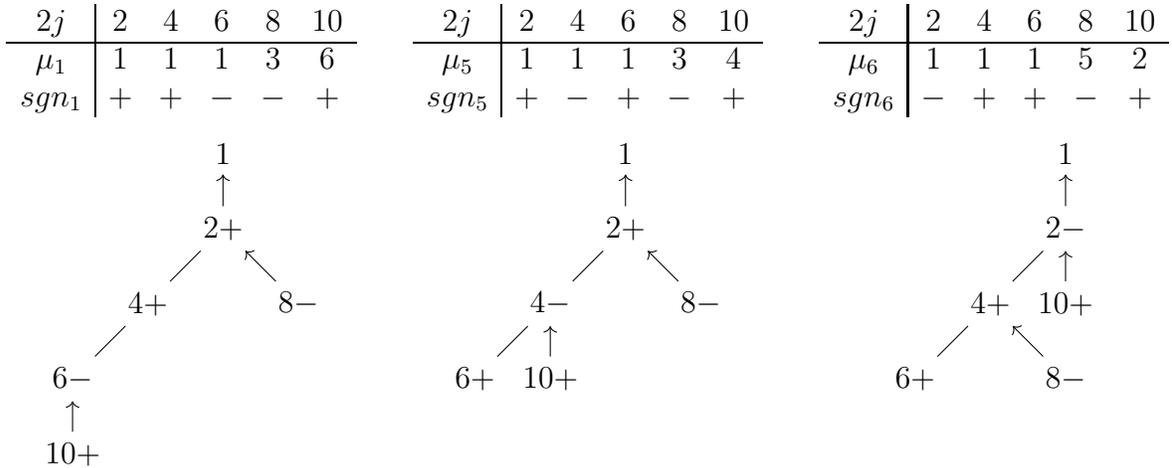
\begin{minipage}[t]{0.32\textwidth}
\centering
$$
\begin{tabular}{c|ccccc}
$2j$&2&4&6&8&10\\
\hline
$\mu_{1}$ &1&1&1&3&6\\
$sgn_{1}$ &+&+&$-$&$-$&+
\end{tabular}
$$

\begin{tikzpicture}
\node (1) at (0,0) {1};
\node (2) at (0,-1) {2+};
\node (4) at (-1,-2) {4+};
\node (8) at (1,-2) {8$-$};
\node (6) at (-2,-3) {6$-$};
\node (10) at (-2,-4) {10+};
\draw[<-] (1)--(2);
\draw[-] (2)--(4);
\draw[<-] (2)--(8);
\draw[-] (4)--(6);
\draw[<-] (6)--(10);
\end{tikzpicture}

\end{minipage}
\begin{minipage}[t]{0.32\textwidth}
\centering
$$
\begin{tabular}{c|ccccc}
$2j$&2&4&6&8&10\\
\hline
$\mu_{5}$ &1&1&1&3&4\\
$sgn_{5}$ &+&$-$&+&$-$&+
\end{tabular}
$$

\begin{tikzpicture}
\node (1) at (0,0) {1};
\node (2) at (0,-1) {2+};
\node (4) at (-1,-2) {4$-$};
\node (8) at (1,-2) {8$-$};
\node (6) at (-2,-3) {6+};
\node (10) at (-1,-3) {10+};
\draw[<-] (1)--(2);
\draw[-] (2)--(4);
\draw[<-] (2)--(8);
\draw[-] (4)--(6);
\draw[<-] (4)--(10);
\end{tikzpicture}
\end{minipage}
\begin{minipage}[t]{0.32\textwidth}
\centering
$$
\begin{tabular}{c|ccccc}
$2j$&2&4&6&8&10\\
\hline
$\mu_{6}$ &1&1&1&5&2\\
$sgn_{6}$ &$-$&+&+&$-$&+
\end{tabular}
$$

\begin{tikzpicture}
\node (1) at (0,0) {1};
\node (2) at (0,-1) {2$-$};
\node (4) at (-1,-2) {4+};
\node (8) at (0,-3) {8$-$};
\node (6) at (-2,-3) {6+};
\node (10) at (0,-2) {10+};
\draw[<-] (1)--(2);
\draw[-] (2)--(4);
\draw[<-] (2)--(10);
\draw[-] (4)--(6);
\draw[<-] (4)--(8);
\end{tikzpicture}

\end{minipage}
~\\
(Notice the changes on the tree structures under the wild move.)

Finally, in Section \ref{subsection:Reference Form and Proof of Compatibility}, we prove that, given a tamed class, there exists a unique reference form representing the
tamed class, and the time integration domain for the whole tamed class can be directly read
out from the reference form. Moreover, the time integration domain is compatible. For instance, as $(\mu_{1},sgn_{1})$ is also the unique reference form, we could directly read $T_{C}(\mu_{1},sgn_{1})$ out as follows
$$T_{C}(\mu_{1},sgn_{1})=\lr{t_{1}\geq t_{3},t_{3}\geq t_{5},t_{7}\geq t_{11},t_{3}\geq t_{9}},$$
which is indeed compatible with $U$-$V$ multilinear estimates. (See Example \ref{example:estimate part} on how to estimate.) Hence, we arrive at the representation $(\ref{equ:duhamel expansion, reference form, section 2})$.

 At the end, we have justified the representation $(\ref{equ:duhamel expansion, reference form, section 2})$ and the key estimate $(\ref{equ:key estimate, section 2})$ is now valid. Hence, we have
\begin{align*}
\bbn{\lra{\nabla_{x_{1}}}^{s_{c}-2}\lra{\nabla_{x_{1}'}}^{s_{c}-2}\ga^{(1)}(t_{1})}
_{L_{T}^{\wq}L_{x_{1}}^{2}L_{x_{1}'}^{2}}\leq (16\delta)^{k} \to 0\ \text{as $k\to \wq$},
\end{align*}
which implies that $\ga^{(1)}\equiv 0$ which finishes the proof of Theorem \ref{thm:uniqueness for gp hierarchy}.

To sum up, we prove unconditional uniqueness for solutions to the $\R^{d}$ and $\T^{d}$ energy-supercritical cubic and quintic NLS at critical regularity for all dimensions via a newly developed unified scheme. Thus, together with \cite{chen2019the,chen2020unconditional}, the unconditional uniqueness problems for $H^{1}$-critical and $H^{1}$-supercritical cubic and quintic NLS are completely and uniformly resolved at critical regularity for these domains. The novelty of this paper is that our procedure works uniformly in all dimensions regardless of the domain.

\subsection{The Uniqueness for GP Hierarchy}\label{section:The Uniqueness for GP Hierarchy}
In this section, we first prove Theorem \ref{thm:uniqueness for gp hierarchy}, which is a uniqueness theorem for the GP hierarchy.  Theorem \ref{thm:uniqueness for nls} then comes as a corollary of Theorem \ref{thm:uniqueness for gp hierarchy} and Lemma \ref{lemma:UTFL}. Here, we consider the $\T^{d}$ case with the inhomogeneous norm $H^{s_{c}}$, as the homogeneous norm $\dot{H}^{s_{c}}$ is special for the $\R^{d}$ case.
\begin{definition}[\cite{chen2015unconditional}]\label{definition:admissible sequence}
A nonnegative trace class symmetric operator sequence
$\Ga=\lr{\ga^{(k)}}_{k=1}^{\wq}$, is
called admissible if for all $k$, one has
\begin{align}
\operatorname{Tr}\ga^{(k)}=1,\ \ga^{(k)}=\operatorname{Tr}_{k+1}\ga^{(k+1)}.
\end{align}
Here, a trace class operator is called symmetric, if, written in kernel form
\begin{align*}
&\ga^{(k)}(x_{1},...,x_{k};x'_{1},...,x'_{k})=\ol{\ga^{(k)}(x'_{1},...,x'_{k};x_{1},...,x_{k})},\\
&\ga^{(k)}(x_{\pi(1)},...,x_{\pi(k)};x'_{\pi(1)},...,x'_{\pi(k)})=
\ga^{(k)}(x_{1},...,x_{k};x'_{1},...,x'_{k}),
\end{align*}
for all permutations $\pi$ on $\lr{1,2,...,k}$. Let $\ch^{s}\equiv\ch^{s}(\Lambda^{d})$ denote the set of all symmetric operator sequences
$\lr{\ga^{(k)}}_{k=1}^{\wq}$ of density matrices such that, for each
$k\in \N$
\begin{align*}
\lrs{\prod_{j=1}^{k}\lra{\nabla_{x_{j}}}^{s_{c}}}\ga^{(k)}\lrs{\prod_{j=1}^{k}
\lra{\nabla_{x_{j}}}^{s_{c}}}\in \cl_{k}^{1},
\end{align*}
where $\cl_{k}^{1}$ denotes the space of trace class on $L^{2}(\Lambda^{dk}\times \Lambda^{dk})$.
\end{definition}

\begin{theorem}\label{thm:uniqueness for gp hierarchy}
Let $\Ga=\lr{\ga^{(k)}}\in \bigoplus_{k\geq 1}C([0,T_{0}];\ch_{k}^{s_{c}})$ be a solution to $(\ref{equ:gp hierarchy,quintic})$ in $[0,T_{0}]$ in the sense that

$(1)$ $\Ga$ is admissible in the sense of Definition $\ref{definition:admissible sequence}$.

$(2)$ $\Ga$ satisfies the kinetic energy condition that $\exists$ $C_{0}>0$ such that
\begin{align}\label{equ:energy bound}
\sup_{t\in[0,T_{0}]} \operatorname{Tr}\lrs{\prod_{j=1}^{k}\lra{\nabla_{x_{j}}}^{s_{c}}}\ga^{(k)}(t)\lrs{\prod_{j=1}^{k}
\lra{\nabla_{x_{j}}}^{s_{c}}}\leq C_{0}^{2k}.
\end{align}

Then there is a threshold $\eta(C_{0})>0$ such that the solution is unique in $[0,T_{0}]$ provided that
\begin{align}\label{equ:HUFL}
\sup_{t\in[0,T_{0}]} \operatorname{Tr}\lrs{\prod_{j=1}^{k}P_{>M}^{j}\lra{\nabla_{x_{j}}}^{s_{c}}}\ga^{(k)}(t)\lrs{\prod_{j=1}^{k}
P_{>M}^{j}\lra{\nabla_{x_{j}}}^{s_{c}}}\leq \eta^{2k}
\end{align}
for some frequency $M$, which is allowed to depend on $\ga^{(k)}$ but must apply uniformly on $[0,T_{0}]$. Here, $P_{> M}^{j}$ is the projection
onto frequencies $>M$ acting on functions of $x_{j}$.
\end{theorem}

By letting $\ga^{(k)}=|u\rangle \langle u|^{\otimes k}$, we could obtain Corollary \ref{cor:uniqueness for nls}, which is a special case of Theorem \ref{thm:uniqueness for gp hierarchy} as follows.
\begin{corollary} \label{cor:uniqueness for nls}
Given an initial datum $u_{0}\in H^{s_{c}}$, there is at most one
$C([0,T_{0}];H^{s_{c}})$ solution to $(\ref{equ:NLS})$ satisfying the following conditions

$(1)$ There is a $C_{0}>0$ such that
\begin{align*}
\sup_{t\in [0,T_{0}]}\n{u(t)}_{H^{s_{c}}}\leq C_{0}
\end{align*}

$(2)$ There is some frequency $M$ such that
\begin{align*}
\sup_{t\in [0,T_{0}]}\n{P_{> M}u(t)}_{H^{s_{c}}}\leq \eta(C_{0})
\end{align*}
for the threshold $\eta(C_{0})>0$ concluded in Theorem $\ref{thm:uniqueness for gp hierarchy}$.
\end{corollary}

Before the proof, we set up some notations. We rewrite $\ga^{(k)}$ in Duhamel form
\begin{equation} \label{equ:gp hierarchy in duhamel form}
\ga^{(k)}(t)=U^{(k)}(t)\ga^{(k)}(0)\mp i\int_{0}^{t}U^{(k)}(t-s)B^{(k+2)}\ga^{(k+2)}(s)ds
\end{equation}
where $U^{(k)}(t)=\prod_{j=1}^{k}e^{it(\Delta_{x_{j}}-\Delta_{x_{j}'})}$ and
\begin{align*}
&B^{(k+2)}\ga^{(k+2)}\\
=&\sum_{j=1}^{k}B_{j;k+1,k+2}\ga^{(k+2)}\\
=&\sum_{j=1}^{k}\lrs{B_{j;k+1,k+2}^{+}-B_{j;k+1,k+2}^{-}}\ga^{(k+2)}\\
=&\sum_{j=1}^{k}\operatorname{Tr}_{k+1,k+2}\lrs{\delta(x_{j}-x_{k+1})\delta(x_{j}-x_{k+2})\ga^{(k+2)}-
\ga^{(k+2)}\delta(x_{j}-x_{k+1})\delta(x_{j}-x_{k+2})}.
\end{align*}
In the above, products are interpreted as the compositions of operators. For example, in kernels,
\begin{align*}
&\operatorname{Tr}_{k+1,k+2}\lrs{\ga^{(k+2)}\delta(x_{j}-x_{k+1})\delta(x_{j}-x_{k+2})}(\mathbf{x}_{k};
\mathbf{x}_{k}')\\
=&\int \ga^{(k+2)}(\mathbf{x}_{k},x_{k+1}',x_{k+2}';
\mathbf{x}_{k}',x_{k+1}',x_{k+2}')\delta (x_{j}'-x_{k+1}')\delta(x_{j}'-x_{k+2}')dx_{k+1}'dx_{k+2}'
\end{align*}
where $\mathbf{x}_{k}=(x_{1},...,x_{k})$.

We will prove that if $\Ga_{1}=\lr{\ga_{1}^{(k)}}_{k=1}^{\wq}$ and $\Ga_{2}=\lr{\ga_{2}^{(k)}}_{k=1}^{\wq}$ are two
solutions to $(\ref{equ:gp hierarchy in duhamel form})$, with the same initial datum and assumptions in Theorem $\ref{thm:uniqueness for gp hierarchy}$, then $\Ga=\lr{\ga^{(k)}=\ga_{1}^{(k)}-\ga_{2}^{(k)}}$ is identically zero. We start
using a representation of $\Ga$ given by the quantum de Finetti theorem.

\begin{lemma}[quantum de Finetti Theorem \cite{chen2015unconditional,lewin2014derivation}] \footnote{We in fact do not need Lemma $\ref{lemma:quantum de Finetti theorem}$ to prove Theorem $\ref{thm:uniqueness for nls}$, but it is the origin of the ideas of this proof.} \label{lemma:quantum de Finetti theorem}
  Let $\lr{\ga^{(k)}}_{k=1}^{\wq}$ be admissible. Then there exists a probability measure $d\mu_{t}(\phi)$ supported on the unit sphere of $L^{2}(\Lambda^{d})$ such that
\begin{align*}
\ga^{(k)}(t)=\int |\phi\rangle \langle \phi|^{\otimes k}d\mu_{t}(\phi).
\end{align*}
\end{lemma}

By Lemma $\ref{lemma:quantum de Finetti theorem}$, there exist $d\mu_{1,t}$ and $d\mu_{2,t}$
representing the two solutions $\Ga_{1}$ and $\Ga_{2}$. Then the Chebyshev argument as in \cite[Lemma 4.5]{chen2015unconditional} turns the assumptions in Theorem \ref{thm:uniqueness for gp hierarchy}
to the support property that $d\mu_{j,t}$ is supported in the
set
\begin{align} \label{equ:support set, quantum de Finetti theorem}
S=\lr{\phi\in \mathbb{S}(L^{2}(\Lambda^{d})):\n{P_{>M}\lra{\nabla}^{s_{c}}\phi}_{L^{2}}\leq \ve}
\bigcap \lr{\phi\in \mathbb{S}(L^{2}(\Lambda^{d})):\n{\phi}_{H^{s_{c}}}\leq C_{0}}.
\end{align}

Let the signed measure $d\mu_{t}=d\mu_{1,t}-d\mu_{2,t}$, we have
\begin{equation}
\ga^{(k)}(t)=\lrs{\ga_{1}^{(k)}-\ga_{2}^{(k)}}(t)=\int|\phi\rangle \langle \phi|^{\otimes k} d\mu_{t}(\phi)
\end{equation}
and $d\mu_{t}$ is supported in the set $S$ defined in $(\ref{equ:support set, quantum de Finetti theorem})$.

It suffices to
prove $\ga^{(1)}(t)=0$ as the proof is the same for the general $k$ case.
For notational convenience, we set the $\pm i$ in $(\ref{equ:gp hierarchy in duhamel form})$ to be $1$.
Since $(\ref{equ:gp hierarchy in duhamel form})$ is linear, $\Ga$ is a solution to
$(\ref{equ:gp hierarchy in duhamel form})$ with zero initial datum. Thus after iterating $(\ref{equ:gp hierarchy in duhamel form})$ $k$ times, we can write
\begin{align}
\ga^{(1)}(t_{1})=\int_{0}^{t_{1}}\int_{0}^{t_{3}}\ccc \int_{0}^{t_{2k-1}}J^{(2k+1)}(\ga^{(2k+1)})(t_{1},\underline{t}_{2k+1})d\underline{t}_{2k+1}
\end{align}
where
\begin{align}
&J^{(2k+1)}(\ga^{(2k+1)})(t_{1},\underline{t}_{2k+1})\\
=&U^{(1)}(t_{1}-t_{3})B^{(3)}U^{(3)}(t_{3}-t_{5})
B^{(5)}\ccc U^{(2k-1)}(t_{2k-1}-t_{2k+1})B^{(2k+1)}\ga^{(2k+1)}(t_{2k+1})\notag
\end{align}
with $\underline{t}_{2k+1}=(t_{3},t_{5},...,t_{2k+1})$.

We notice that there are $(2k-1)!!2^{k}$ summands inside $\ga^{(1)}(t_{1})$. Exactly,
\begin{equation}\label{equ:duhamel expansion, summands}
\ga^{(1)}(t_{1})=\sum_{(\mu,sgn) }\int_{t_{1}\geq t_{3}\geq \ccc\geq t_{2k+1}} J_{\mu,sgn}^{(2k+1)}(\ga^{(2k+1)})(t_{1},\underline{t}_{2k+1})d\underline{t}_{2k+1},
\end{equation}
where
\begin{align} \label{equ:duhamel expansion}
J_{\mu,sgn}^{(2k+1)}(\ga^{(2k+1)})(t_{1},\underline{t}_{2k+1})=
&U^{(1)}(t_{1}-t_{3})B_{\mu(2);2,3}^{sgn(2)}U^{(3)}(t_{3}-t_{5})
B_{\mu(4);4,5}^{sgn(4)}\\
&\ccc U^{(2k-1)}(t_{2k-1}-t_{2k+1})B_{\mu(2k);2k,2k+1}^{sgn(2k)}\ga^{(2k+1)}(t_{2k+1})\notag
\end{align}
with $sgn$ meaning the signature array $(sgn(2), sgn(4),..., sgn(2k))$ and
$B_{\mu(2j);2j,2j+1}^{sgn(2j)}$ stands for $B_{\mu(2j);2j,2j+1}^{+}$
or $B_{\mu(2j);2j,2j+1}^{-}$ depending on the sign of the $2j$-th signature element.
Here, $\lr{\mu}$ is a set of maps from $\lr{2,4,...,2k}$ to $\lr{1,2,...,2k-1}$ satisfying
that $\mu(2)=1$ and $\mu(2j)<2j$ for all $j$.

Now, we get into the proof, which spans Section \ref{section:Estimates for the Compatible Time Integration Domain}-Section \ref{section:Multilinear Estimates}, so we first state two propositions which are the main results in Section \ref{section:Estimates for the Compatible Time Integration Domain} and Section \ref{section:Existence of Compatible Time Integration Domain}.

In Section \ref{section:Estimates for the Compatible Time Integration Domain}, the main result is the following proposition.
\begin{proposition}\label{prop:estimate for the compatible time integration domain,section 2}
Let $\ga^{(k)}(t)=\int |\phi\rangle \langle \phi|^{\otimes k} d\mu_{t}(\phi)$. Then we have
\begin{align*}
&\bbn{\lra{\nabla_{x_{1}}}^{s_{c}-2}\lra{\nabla_{x_{1}'}}^{s_{c}-2}\int_{T_{C}(\mu,sgn)}J_{\mu,sgn}^{(2k+1)}
(\ga^{(2k+1)})
(t_{1},\underline{t}_{2k+1})d\underline{t}_{2k+1}}
_{L_{T}^{\wq}L_{x_{1}}^{2}L_{x_{1}'}^{2}}\\
\leq &C^{k}\int_{0}^{T}\int \n{\phi}_{H^{\frac{d-1}{2}}}^{\frac{16}{5}k+2}\lrs{T^{\frac{1}{2(d+3)}}M_{0}^{\frac{2d+3}{2(d+3)}}\n{P_{\leq M_{0}}\phi}_{H^{\frac{d-1}{2}}}+\n{P_{>M_{0}}\phi}_{H^{\frac{d-1}{2}}}}^{\frac{4k}{5}}
d|\mu_{t_{2k+1}}|(\phi)dt_{2k+1}
\end{align*}
where $T(\mu,sgn)$ is the compatible time integration domain defined by $(\ref{equ:compatible time integration domain})$.
\end{proposition}

In Section \ref{section:Existence of Compatible Time Integration Domain}, our goal is to represent $\ga^{(1)}$ as follows.
\begin{proposition} \label{prop:compatibility for time integration,section 2}
We can write
\begin{align}
\ga^{(1)}(t_{1})=\sum_{ reference\ (\hat{\mu},\hat{sgn})}
\int_{T_{C}(\hat{\mu},\hat{sgn})}J_{\hat{\mu},\hat{sgn}}^{(2k+1)}(\ga^{(2k+1)})
(t_{1},\underline{t}_{2k+1})d\underline{t}_{2k+1}
\end{align}
where the number of reference pairs in Definition \ref{def:reference pair} can be controlled by $16^{k}$, which is substantially smaller than the original $(2k-1)!!2^{k}$. More importantly, the summands are endowed with the time integration domain $T_{C}(\hat{\mu},\hat{sgn})$, which is compatible with the estimate part in Section \ref{section:Estimates for the Compatible Time Integration Domain}.
\end{proposition}

With the above two propositions, we could complete the proofs of Theorem \ref{thm:uniqueness for gp hierarchy} and Corollary \ref{cor:uniqueness for nls}.

\begin{proof}[\textbf{Proof of Theorem $\ref{thm:uniqueness for gp hierarchy}$}]
By Proposition $\ref{prop:compatibility for time integration,section 2}$, we can write
\begin{align*}
\ga^{(1)}(t_{1})=\sum_{ reference\ (\hat{\mu},\hat{sgn})}
\int_{T_{C}(\hat{\mu},\hat{sgn})}J_{\hat{\mu},\hat{sgn}}^{(2k+1)}(\ga^{(2k+1)})
(t_{1},\underline{t}_{2k+1})d\underline{t}_{2k+1}
\end{align*}
where  the number of reference pairs can be controlled by $16^{k}$. Then we have
\begin{align*}
&\bbn{\lra{\nabla_{x_{1}}}^{s_{c}-2}\lra{\nabla_{x_{1}'}}^{s_{c}-2}\ga^{(1)}(t_{1})}
_{L_{T}^{\wq}L_{x_{1}}^{2}L_{x_{1}'}^{2}}\\
\leq &\sum_{reference\  (\hat{\mu},\hat{sgn})}\bbn{\lra{\nabla_{x_{1}}}^{s_{c}-2}\lra{\nabla_{x_{1}'}}^{s_{c}-2}
\int_{T_{C}(\hat{\mu},\hat{sgn})}J_{\hat{\mu},\hat{sgn}}^{(2k+1)}(\ga^{(2k+1)})
(t_{1},\underline{t}_{2k+1})d\underline{t}_{2k+1}}_{L_{T}^{\wq}L_{x_{1}}^{2}L_{x_{1}'}^{2}}
\end{align*}
By Proposition \ref{prop:estimate for the compatible time integration domain,section 2},
\begin{align*}
\leq& (16C)^{k} \int_{0}^{T}\int \n{\phi}_{H^{s_{c}}}^{\frac{16}{5}k+2}\lrs{T^{\frac{1}{2(d+3)}}M_{0}^{\frac{2d+3}{2(d+3)}}\n{P_{\leq M_{0}}\phi}_{H^{s_{c}}}+\n{P_{>M_{0}}\phi}_{H^{s_{c}}}}^{\frac{4k}{5}}
d|\mu_{t_{2k+1}}|(\phi)dt_{2k+1}
\end{align*}
Put in the support property $(\ref{equ:support set, quantum de Finetti theorem})$,
\begin{align*}
\leq&2T(16C)^{k}C_{0}^{\frac{16}{5}k+2}
\lrs{T^{\frac{1}{2(d+3)}}M_{0}^{\frac{2d+3}{2(d+3)}}C_{0}+\ve}^{\frac{4}{5}k}\\
\leq& 2TC_{0}^{2}\lrs{T^{\frac{1}{2(d+3)}}M_{0}^{\frac{2d+3}{2(d+3)}}C^{2}C_{0}^{5}+C^{2}C_{0}^{4}\ve }^{\frac{4}{5}k}
\end{align*}
for all $k$. Select $\ve$ small enough such that $C^{2}C_{0}^{4}\ve< \frac{1}{4}$ and then select $T$ small enough such that $T^{\frac{1}{2(d+3)}}M_{0}^{\frac{2d+3}{2(d+3)}}C^{2}C_{0}^{5}< \frac{1}{4}$, we thus have for $T$ small enough,
\begin{align*}
\bbn{\lra{\nabla_{x_{1}}}^{s_{c}-2}\lra{\nabla_{x_{1}'}}^{s_{c}-2}\ga^{(1)}(t_{1})}
_{L_{T}^{\wq}L_{x_{1}}^{2}L_{x_{1}'}^{2}}\leq 2TC_{0}^{2}\lrs{\frac{1}{2}}^{k}\to 0\ \text{as $k\to \wq$.}
\end{align*}
We can then bootstrap to fill the whole $[0,T_{0}]$ interval.
\end{proof}

\begin{proof}[\textbf{Proof of Corollary $\ref{cor:uniqueness for nls}$}]
Given the solution $\mu$ of $(\ref{equ:NLS})$, it generates a solution to the GP hierarchy $(\ref{equ:gp hierarchy,quintic})$ taking the following form
\begin{align}
\lr{\prod_{j=1}^{k}u(t,x_{j})\ol{u}(t,x_{j}')}_{k=1}^{\wq}.
\end{align}
Thus, one could just apply the proof of Theorem $\ref{thm:uniqueness for gp hierarchy}$ to the special case
\begin{align*}
\ga^{(k)}(t)=&\int|\phi\rangle \langle \phi|^{\otimes k} d\mu_{t}(\phi)
=\prod_{j=1}^{k}u_{1}(t,x_{j})\ol{u}_{1}(t,x_{j}')-
\prod_{j=1}^{k}u_{2}(t,x_{j})\ol{u}_{2}(t,x_{j}'),
\end{align*}
where $u_{1}$ and $u_{2}$ are two solutions to $(\ref{equ:NLS})$ and $\mu_{t}$ is the signed measure
$\delta_{u_{1,t}}-\delta_{u_{2,t}}$. Especially, when $k=1$, we have proved that
\begin{align}\label{equ:uniqueness, conjugate}
u_{1}(t,x)\ol{u}_{1}(t,x')=u_{2}(t,x)\ol{u}_{2}(t,x'),
\end{align}
which directly implies the uniqueness for the trivial solution $u\equiv 0$.
Then by Lemma \ref{lemma:uniqueness from tensor form} which concludes that the phase difference is zero, we complete the proof.
\end{proof}
As Corollary \ref{cor:uniqueness for nls} requires condition (2), uniform in time frequency (UTFL) condition, we prove that every $C([0,T_{0}];H^{s_{c}})$ solution to $(\ref{equ:NLS})$ satisfies UTFL condition by Lemma $\ref{lemma:UTFL}$. Immediately, Theorem $\ref{thm:uniqueness for nls}$ follows from Corollary \ref{cor:uniqueness for nls} and Lemma \ref{lemma:UTFL}.
\begin{lemma}\label{lemma:UTFL}
Let $u$ be a $C([0,T_{0}];H^{s_{c}})$ solution,
then $u$ satisfies uniform in time frequency localization $($UTFL$)$
, that is, for each $\ve>0$ there exists $M(\ve)$ such that
\begin{equation} \label{equ:UTFL for NLS}
\n{\lra{\nabla}^{s_{c}}P_{> M(\ve)}u}_{L_{[0,T_{0}]}^{\wq}L_{x}^{2}}\leq \ve.
\end{equation}
\end{lemma}
\begin{proof}
We compute
\begin{align*}
\babs{\pa_{t}\n{\lra{\nabla}^{s_{c}} P_{\leq M}u}_{L_{x}^{2}}^{2}}=&2\bbabs{\operatorname{Im} \int P_{\leq M}\lra{\nabla}^{s_{c}} u\cdot P_{\leq M}\lra{\nabla}^{s_{c}}(|u|^{p-1}u)dx}\\
\leq& 2\n{P_{\leq M}\lra{\nabla}^{s_{c}} u}_{L^{2}}\n{P_{\leq M}\lra{\nabla}^{s_{c}}(|u|^{p-1}u)}_{L^{2}}.
\end{align*}
Noting that $\n{P_{\leq M}\lra{\nabla}^{s}f}_{L^{2}}\lesssim M^{s}\n{P_{\leq M}f}_{L^{2}}$, then by Sobolev embedding $\lrs{\ref{lemma:sobolev embedding critical case}}$ and $(\ref{equ:sobolev inequality, multilinear estimate, endpoint})$, we have
\begin{align*}
\babs{\pa_{t}\n{\lra{\nabla}^{s_{c}} P_{\leq M}u}_{L_{x}^{2}}^{2}}\lesssim & 2M^{2}\n{ P_{\leq M}\lra{\nabla}^{s_{c}}u}_{L^{2}}\n{P_{\leq M}\lra{\nabla}^{s_{c}-2}(|u|^{p-1}u)}_{L^{2}}\\
\lesssim& 2M^{2}\n{u}_{H^{s_{c}}}^{p+1}.
\end{align*}
Hence there exists $\delta'>0$ such that for any $t_{0}\in [0,T_{0}]$, it holds that for $t\in (t_{0}-\delta',t_{0}+\delta')\cap [0,T]$,
\begin{align}
\bbabs{\n{\lra{\nabla}^{s_{c}} P_{\leq M} u(t)}_{L_{x}^{2}}^{2}-\n{\lra{\nabla}^{s_{c}} P_{\leq M}u(t_{0})}_{L_{x}^{2}}^{2}}\leq \frac{1}{16}\ve^{2}.
\end{align}
On the other hand, since $u\in C([0,T_{0}];H^{s_{c}})$, for each $t_{0}$, there exists $\delta''>0$ such that for any $t\in (t_{0}-\delta'',t_{0}+\delta'')\cap [0,T_{0}]$,
\begin{align}
\bbabs{\n{\lra{\nabla}^{s_{c}}  u(t)}_{L_{x}^{2}}^{2}-\n{\lra{\nabla}^{s_{c}} u(t_{0})}_{L_{x}^{2}}^{2}}\leq \frac{1}{16}\ve^{2}.
\end{align}
Let $\delta=\min\lrs{\delta',\delta''}$. Then we have that for any $t\in(t_{0}-\delta,t_{0}+\delta)\cap [0,T_{0}]$,
\begin{align*}
\bbabs{\n{\lra{\nabla}^{s_{c}}P_{>M}u(t)}_{L_{x}^{2}}^{2}-\n{\lra{\nabla}^{s_{c}}P_{>M}u(t_{0})}_{L_{x}^{2}}^{2}}\leq\frac{1}{4}\ve^{2}.
\end{align*}
For each $t\in[0,T_{0}]$, there exists $M_{t}$ such that
\begin{align*}
\n{\lra{\nabla}^{s_{c}}P_{>M_{t}}u(t)}_{L_{x}^{2}}\leq \frac{1}{2}\ve.
\end{align*}
By the above, there exists $\delta_{t}>0$ such that on $(t-\delta_{t},t+\delta_{t})\cap [0,T_{0}]$, we have
\begin{align*}
\n{\lra{\nabla}^{s_{c}}P_{>M_{t}}u}_{L_{(t-\delta_{t},t+\delta_{t})\cap [0,T_{0}]}^{\wq}L_{x}^{2}}\leq \ve.
\end{align*}
Since the collection of interval $(t-\delta_{t},t+\delta_{t})\cap [0,T_{0}]$, as $t$ ranges over $[0,T_{0}]$, is an open cover of $[0,T_{0}]$. By compactness, we might as well assume that
\begin{align*}
(t_{1}-\delta_{t_{1}},t_{1}+\delta_{t_{1}})\cap [0,T_{0}],...,(t_{J}-\delta_{t_{J}},t_{J}+\delta_{t_{J}})\cap [0,T_{0}]
\end{align*}
be a finite open cover of $[0,T_{0}]$. Letting
$$M=\lrs{M_{t_{1}},...,M_{t_{J}}},$$
we have established $(\ref{equ:UTFL for NLS})$.

\end{proof}
%The proof is under $\T^{d}$ setting with inhomogeneous norm and also works the same for $\R^{d}$ case provided that we replace $H^{s_{c}}$ by $\dot{H}^{s_{c}}$ and $\lra{\nabla}^{s_{c}}$ by $|\nabla|^{s_{c}}$.
%We remark that the proof of UTFL is the same for $\R^{d}$ case, in which we replace $H^{s_{c}}$ by $\dot{H}^{s_{c}}$ and $\lra{\nabla}^{s_{c}}$ by $|\nabla|^{s_{c}}$.

%\begin{proof}[\textbf{Proof of Theorem $\ref{thm:uniqueness for nls}$}]
%For the $\T^{d}$ case, it follows from Corollary \ref{cor:uniqueness for nls} by using the UTFL property in Lemma $\ref{lemma:UTFL}$. For the $\R^{d}$ case with the homogeneous norm $\dot{H}^{s_{c}}$, by the same argument of treating the inhomogeneous norm $H^{s_{c}}$, we have that $P_{N}\ga^{(1)}P_{N}=0$ for every $N$, which implies $u_{1}(t,x)\ol{u}_{1}(t,x')=u_{2}(t,x)\ol{u}_{2}(t,x')$ and hence the uniqueness of NLS by Lemma \ref{lemma:uniqueness from tensor form}.
%\end{proof}

\section{Estimates for the Compatible Time Integration Domain} \label{section:Estimates for the Compatible Time Integration Domain}
\subsection{Duhamel Expansion and Duhamel Tree}\label{subsection:Duhamel Expansion and Duhamel Tree}
We start the analysis of the Duhamel expansions.
We will create a Duhamel tree (we write $D$-tree for short) and show how to obtain the Duhamel expansion $J_{\mu,sgn}^{(2k+1)}$ from the $D$-tree. At first, we present an algorithm to generate a Duhamel tree from a collapsing map pair $(\mu,sgn)$ and then show this by an example. Subsequently, we are able to calculate the Duhamel expansion by a general algorithm. Finally, we exhibit an example by employing the above algorithms.
\begin{algorithm}[Duhamel Tree] \label{algorithm:duhamel tree}
~\\
\hspace*{1em}$(1)$ Let $D^{(0)}$ be a starting node in the $D$-tree.  Find the pair of indices $l$ and $r$ so that
$l\geq 1$, $r\geq 1$ and
\begin{align*}
&\mu(2l)=1,\ sgn(2l)=+,\\
&\mu(2r)=1,\ sgn(2r)=-,
\end{align*}
and moreover $l$ and $r$ are the minimal indices for which the above equalities hold. Then place
$D^{(2l)}$ or $D^{(2r)}$ as the left or right child of $D^{(0)}$ in the $D$-tree. If there is no such $l$ or $r$, place $F_{1,+}$ or $F_{1,-}$ as the left or right child of $D^{(0)}$ in the $D$-tree.

$(2)$ Set counter $j=1$.

$(3)$ Given $j$, find the indices $\lr{k_{i}}_{i=1}^{5}$ so that $k_{i}>j$ and
\begin{align}\label{equ:duhamel tree, algorithm, step 3}
\begin{cases}
&\mu(2k_{1})=\mu(2j),\ sgn(2k_{1})=sgn(2j),\\
&\mu(2k_{2})=2j,\ sgn(2k_{2})=+,\\
&\mu(2k_{3})=2j,\ sgn(2k_{3})=-,\\
&\mu(2k_{4})=2j+1,\ sgn(2k_{4})=+,\\
&\mu(2k_{5})=2j+1,\ sgn(2k_{5})=-,
\end{cases}
\end{align}
and $k_{i}$ is the minimal index for which the corresponding equalities hold. Then place $D^{(2k_{i})}$ as the $i$-th child of $D^{(2j)}$ in the $D$-tree. If there is no such $k_{1}$/$k_{2}$/$k_{3}$/$k_{4}$/$k_{5}$, then place $F_{\mu(2j),sgn(2j)}$/$F_{2j,+}$/$F_{2j,-}$/$F_{2j+1,+}$/$F_{2j+1,-}$ as the $i$-th child of $D^{(2j)}$ in the $D$-tree.

$(4)$ If $j=k$, then stop, otherwise set $j=j+1$ and go to step $(3)$.

\end{algorithm}
%\begin{remark}
%Since it requires $\mu(2j)<2j$, it is not hard to check that every node $D^{(2j)}$ has a parent by induction argument.
%\end{remark}

\begin{example} \label{example:duhamel tree}
Let us work with the following example
$$
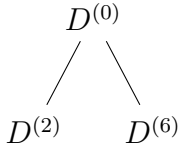
\begin{tabular}{c|ccccccc}
$2j$&2&4&6&8&10&12&14\\
\hline
$\mu(2j)$ &1&1&1&2&3&6&6\\
$sgn(2j)$ &+&+&$-$&$-$&+&+&$-$
\end{tabular}
$$
\begin{minipage}{0.3\textwidth}
\centering
\begin{tikzpicture}
\node{$D^{(0)}$}[sibling distance=45pt]
child{node{$D^{(2)}$}}
child{node{$D^{(6)}$}}
;
\end{tikzpicture}
\end{minipage}
\begin{minipage}{0.69\textwidth}
Let $D^{(0)}$ be a starting node in the $D$-tree, so we need to find the minimal $l\geq 1$, $r\geq 1$
such that $\mu(2l)=1$, $sgn(2l)=+$ and $\mu(2r)=1$, $sgn(2r)=-$. In the case, it is $2l=2$ and
$2r=6$ so we put $D^{(2)}$ and $D^{(6)}$ as left and right children of $D^{(0)}$, respectively, in
the $D$-tree as shown in the left.
\end{minipage}
\\[0.5cm]

\begin{minipage}{0.3\textwidth}
\centering
\begin{tikzpicture}[ box/.style={circle, minimum width =10pt, minimum height=10pt,
inner sep=1pt,draw=black}]
\node {$D^{(0)}$}[sibling distance=60pt]

child {node{$D^{(2)}$} [sibling distance=25pt]
child {node{$D^{(4)}$} [sibling distance=10pt]
}
child {node{$F_{2,+}$}}
child {node{$D^{(8)}$} [sibling distance=10pt,level distance=2cm]
}
child {node{$D^{(10)}$} [sibling distance=10pt]
}
child {node{$F_{3,-}$}}
}
child {node{$D^{(6)}$} [sibling distance=50pt]
}

;
\end{tikzpicture}
\end{minipage}
\begin{minipage}{0.67\textwidth}
Now we start with counter $j=1$ so we need to find the minimal $k_{i}>j$ such that
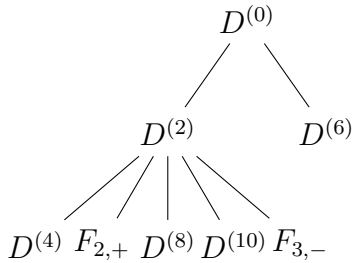
\begin{align*}
&\mu(2k_{1})=\mu(2),\ sgn(2k_{1})=sgn(2),\\
&\mu(2k_{2})=2,\ sgn(2k_{2})=+,\\
&\mu(2k_{3})=2,\ sgn(2k_{3})=-,\\
&\mu(2k_{4})=3,\ sgn(2k_{4})=+,\\
&\mu(2k_{5})=3,\ sgn(2k_{5})=-.
\end{align*}
We find $2k_{1}=4$, $2k_{3}=8$, $2k_{4}=10$ and there is no such $k_{2}$ and $k_{5}$. Thus, we put $D^{(4)}/F_{2,+}/D^{(8)}/D^{(10)}/F_{3,-}$ as the $i$-th child of $D^{(2)}$ $($shown at left$)$.
\end{minipage}

Next, the counter turns to $j=2$ and we find that there is no $k_{i}$ satisfying $(\ref{equ:duhamel tree, algorithm, step 3})$, so we put $F_{1,+}$/$F_{4,+}$/$F_{4,-}$/$F_{5,+}$/$F_{5,-}$ as the $i$-th child of $D^{(4)}$. Then we move to $j=3$ and find that $2k_{2}=12$ and $2k_{3}=14$ satisfy $(\ref{equ:duhamel tree, algorithm, step 3})$ so we put $F_{1,-}/D^{(12)}/D^{(14)}/F_{7,+}/F_{7,-}$ as the $i$-th child of $D^{(6)}$ shown as follows.

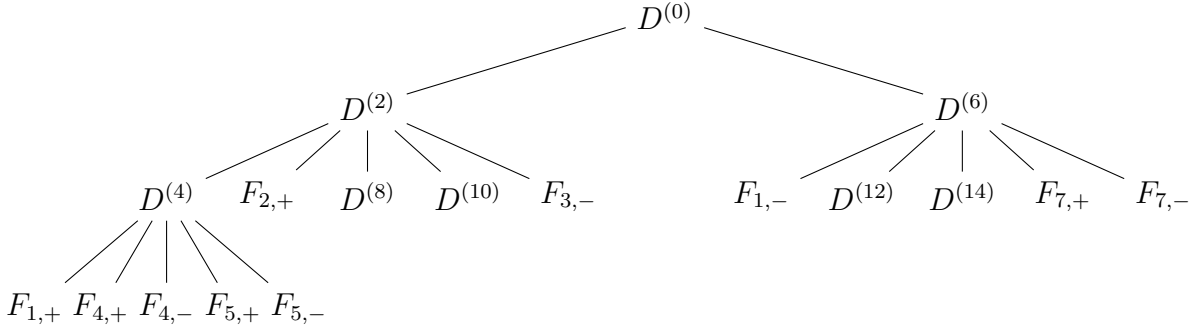
\begin{tikzpicture}
\node {$D^{(0)}$}[sibling distance=45pt,level distance=1.2cm]

child {node{$D^{(2)}$} [sibling distance=38pt]
child {node{$D^{(4)}$} [sibling distance=25pt,level distance=1.5cm]
child {node{$F_{1,+}$}}
child {node{$F_{4,+}$}}
child {node{$F_{4,-}$}}
child {node{$F_{5,+}$}}
child {node{$F_{5,-}$}}
}
child {node{$F_{2,+}$}}
child {node{$D^{(8)}$} [sibling distance=10pt,level distance=1.5cm]
}
child {node{$D^{(10)}$} [sibling distance=10pt]
}
child {node{$F_{3,-}$}}
}
child [missing]
child [missing]
child [missing]
child [missing]
child {node{$D^{(6)}$} [sibling distance=38pt]
child {node{$F_{1,-}$}}
child {node{$D^{(12)}$} [sibling distance=10pt]
}
child {node{$D^{(14)}$} [sibling distance=10pt,level distance=1.5cm]
}
child {node{$F_{7,+}$}}
child {node{$F_{7,-}$}}
}
;
\end{tikzpicture}

Finally, by repeating the above step, we jump to the full $D$-tree shown as follows where we use $F$ to replace
$F_{i,\pm}$ for short.
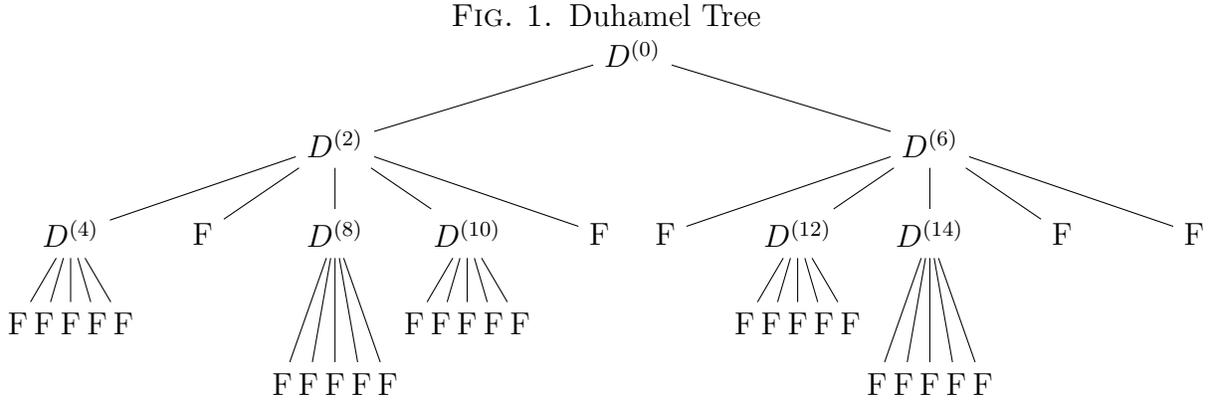
\begin{figure}[htb]
\caption{Duhamel Tree}
\label{figure:duhamel tree}
\begin{tikzpicture}
\node {$D^{(0)}$}[sibling distance=45pt,level distance=1.2cm]

child {node{$D^{(2)}$} [sibling distance=50pt]
child {node{$D^{(4)}$} [sibling distance=10pt]
child {node{F}}
child {node{F}}
child {node{F}}
child {node{F}}
child {node{F}}
}
child {node{F}}
child {node{$D^{(8)}$} [sibling distance=10pt,level distance=2cm]
child {node{F}}
child {node{F}}
child {node{F}}
child {node{F}}
child {node{F}}
}
child {node{$D^{(10)}$} [sibling distance=10pt]
child {node{F}}
child {node{F}}
child {node{F}}
child {node{F}}
child {node{F}}
}
child {node{F}}
}
child [missing]
child [missing]
child [missing]
child [missing]
child {node{$D^{(6)}$} [sibling distance=50pt]
child {node{F}}
child {node{$D^{(12)}$} [sibling distance=10pt]
child {node{F}}
child {node{F}}
child {node{F}}
child {node{F}}
child {node{F}}
}
child {node{$D^{(14)}$} [sibling distance=10pt,level distance=2cm]
child {node{F}}
child {node{F}}
child {node{F}}
child {node{F}}
child {node{F}}
}
child {node{F}}
child {node{F}}
}
;
\end{tikzpicture}
\end{figure}
\end{example}

Next, we present the following algorithm to obtain the Duhamel expansion from the $D$-tree. For convenience, let $U_{\pm j}:=e^{\pm it_{j}\Delta}$.
\begin{algorithm}[From $D$-tree to Duhamel expansion]\label{algorithm:from d-tree to duhamel expansion}
~\\
\hspace*{1em}$(1)$ Set $F_{i,+}=U_{-2k-1}\phi$, $F_{i,-}=\ol{U_{-2k-1}\phi}$ for $i=1,2,...,2k$.
If $sgn(2k)=+$, set
\begin{align*}
&D^{(2k)}(t_{2k+1})=U_{-2k-1}(|\phi|^{4}\phi),
\end{align*}
if $sgn(2k)=-$, set
\begin{align*}
&D^{(2k)}(t_{2k+1})=\ol{U_{-2k-1}(|\phi|^{4}\phi)}.
\end{align*}

$(2)$ Set counter $l=k-1$.

$(3)$ Given $l$, if $sgn(2l)=+$, set
\begin{align*}
&D^{(2l)}(t_{2l+1})=U_{-2l-1}\lrc{\lrs{U_{2l+1}C_{1}}\lrs{U_{2l+1}C_{2}}
\lrs{\ol{U_{2l+1}\ol{C_{3}}}}\lrs{U_{2l+1}C_{4}}\lrs{\ol{U_{2l+1}\ol{C_{5}}}}}
\end{align*}
if $sgn(2l)=-$, set
\begin{align*}
&D^{(2l)}(t_{2l+1})=\ol{U_{-2l-1}\lrc{\lrs{U_{2l+1}\ol{C_{1}}}\lrs{\ol{U_{2l+1}C_{2}}}
\lrs{U_{2l+1}\ol{C_{3}}}\lrs{\ol{U_{2l+1}C_{4}}}\lrs{U_{2l+1}\ol{C_{5}}}}}
\end{align*}
where $C_{i}$ is the $i$-th child of $D^{(2l)}$ in the $D$-tree.

$(4)$ Set $l=l-1$. If $l=0$, set
\begin{align*}
&D^{(0)}(t_{1},\underline{t}_{2k+1})=(U_{1}C_{l})(x_{1})(\ol{U_{1}\ol{C_{r}}})(x_{1}'),
\end{align*}
where $C_{l}$/$C_{r}$ is the left/right child of $D^{(0)}$ in the $D$-tree, and stop, otherwise go to step $(3)$.
\end{algorithm}

\begin{proposition}\label{prop:from d-tree to duhamel expansion}
By Algorithm $\ref{algorithm:from d-tree to duhamel expansion}$ (From $D$-tree to Duhamel expansion), we have
\begin{align*}
J_{\mu,sgn}^{(2k+1)}(|\phi\rangle\langle \phi|^{\otimes (2k+1)})(t_{1},\underline{t}_{2k+1})=D^{(0)}(t_{1},\underline{t}_{2k+1}).
\end{align*}
\end{proposition}

\begin{proof}
It follows from Algorithms $\ref{algorithm:duhamel tree}$ and $\ref{algorithm:from d-tree to duhamel expansion}$.
\end{proof}
\begin{example} \label{example:duhamel expansion}
We calculate the Duhamel expansion in Example $\ref{example:duhamel tree}$. By Algorithm $\ref{algorithm:duhamel tree}$ and $\ref{algorithm:from d-tree to duhamel expansion}$, we obtain
\begin{align*}
\begin{cases}
&D^{(14)}(t_{15})=\ol{U_{-15}(|\phi|^{4}\phi)},\\
&D^{(12)}(t_{13},t_{15})=U_{-13}(|U_{13,15}\phi|^{4}U_{13,15}\phi),\\
&D^{(10)}(t_{11},t_{15})=U_{-11}(|U_{11,15}\phi|^{4}U_{11,15}\phi),\\
&D^{(8)}(t_{9},t_{15})=\ol{U_{-9}(|U_{9,15}\phi|^{4}U_{9,15}\phi)},\\
&D^{(6)}(t_{7},t_{13},t_{15})=\ol{U_{-7}\lrc{(U_{7,15}\phi)(\ol{U_{7}D^{(12)}})
(U_{7}\ol{D^{(14)}})(\ol{U_{7,15}\phi})(U_{7,15}\phi)}},\\
&D^{(4)}(t_{5},t_{15})=U_{-5}(|U_{5,15}\phi|^{4}U_{5,15}\phi),\\
&D^{(2)}(t_{3},t_{5},t_{9},t_{11},t_{15})=U_{-3}\lrc{(U_{3}D^{(4)})(U_{3,15}\phi)
(\ol{U_{3}\ol{D^{(8)}}})
(U_{3}D^{(10)})(\ol{U_{3,15}\phi})},\\
&D^{(0)}(t_{1},\underline{t}_{15})=(U_{1}D^{(2)})(\ol{U_{1}\ol{D^{(6)}}}).
\end{cases}
\end{align*}
where $U_{i,j}:=U_{i}U_{-j}$.

On the one hand, expanding $D^{(0)}(t_{1},\underline{t}_{15})$ gives the Duhamel expansion
\begin{align*}
&D^{(0)}(t_{1},\underline{t}_{15})\\
=&(U_{1}D^{(2)})(\ol{U_{1}\ol{D^{(6)}}})\\
=&\lrs{U_{1,3}\lrc{(U_{3}D^{(4)})(U_{3,15}\phi)
(\ol{U_{3}\ol{D^{(8)}}})(U_{3}D^{(10)})(\ol{U_{3,15}\phi})}}\\
&\cdot \lrs{\ol{U_{1,7}\lrc{(U_{7,15}\phi)(\ol{U_{7}D^{(12)}})
(U_{7}\ol{D^{(14)}})(\ol{U_{7,15}\phi})(U_{7,15}\phi)}}}\\
=&U_{1,3}\lrc{(U_{3,5}(|U_{5,15}\phi|^{4}U_{5,15}\phi))(|U_{3,15}\phi|^{2})
(\ol{U_{3,9}(|U_{9,15}\phi|^{4}U_{9,15}\phi)})
(U_{3,11}(|U_{11,15}\phi|^{4}U_{11,15}\phi))}\\
&\cdot \ol{U_{1,7}\lrc{(|U_{7,15}\phi|^{2}U_{7,15}\phi)(\ol{U_{7,13}(|U_{13,15}\phi|^{4}U_{13,15}\phi)})
(U_{7,15}(|\phi|^{4}\phi))}}
\end{align*}

On the other hand, we calculate $J_{\mu,sgn}^{(15)}(|\phi\rangle\langle \phi|^{\otimes 15})(t_{1},\underline{t}_{2k+1})$ by step. Note that
\begin{align*}
&J_{\mu,sgn}^{(15)}(|\phi\rangle\langle \phi|^{\otimes 15})(t_{1},\underline{t}_{15})\\
=&U_{1,3}^{(1)}B_{1;2,3}^{+}U_{3,5}^{(3)}B_{1;4,5}^{+}U_{5,7}^{(5)}B_{1;6,7}^{-}
U_{7,9}^{(7)}B_{2;8,9}^{-}U_{9,11}^{(9)}B_{3;10,11}^{+}U_{11,13}^{(11)}B_{6;12,13}^{+}
U_{13,15}^{(13)}B_{6;14,15}^{-}(|\phi\rangle\langle \phi|^{\otimes 15}).
\end{align*}
At first,
\begin{align*}
U_{13,15}^{(13)}B_{6;14,15}^{-}(|\phi\rangle\langle \phi|^{\otimes 15})
=&(U_{13,15}\phi)(x_{6})(\ol{U_{13,15}(|\phi|^{4}\phi)})(x_{6}')|U_{13,15}\phi\rangle\langle
U_{13,15}\phi|^{\otimes 12}.
\end{align*}
Adding $U^{(11)}_{11,13}B_{6;12,13}^{+}$ gives
\begin{align*}
(U_{11,13}(|U_{13,15}\phi|^{4}U_{13,15}\phi))(x_{6})(\ol{U_{11,15}(|\phi|^{4}\phi)})(x_{6}')
|U_{11,15}\phi\rangle\langle
U_{11,15}\phi|^{\otimes 10}.
\end{align*}
Then adding $U_{7,9}^{(7)}B_{2;8,9}^{-}U_{9,11}^{(9)}B_{3;10,11}^{+}$ gives
%\begin{align*}
%&(U_{9,11}(|U_{11,15}\phi|^{4}U_{11,15}\phi))(x_{3})(\ol{U_{9,15}\phi})(x_{3}')\\
%&(U_{9,13}(|U_{13,15}\phi|^{4}U_{13,15}\phi))(x_{6})(\ol{U_{9,15}(|\phi|^{4}\phi)})(x_{6}')
%|U_{9,15}\phi\rangle\langle
%U_{9,15}\phi|^{\otimes 7}
%\end{align*}
%Adding gives
\begin{align*}
&(U_{7,15}\phi)(x_{2})(\ol{U_{7,9}(|U_{9,15}\phi|^{4}U_{9,15}\phi)})(x_{2}')\\
&(U_{7,11}(|U_{11,15}\phi|^{4}U_{11,15}\phi))(x_{3})(\ol{U_{7,15}\phi})(x_{3}')\\
&(U_{7,13}(|U_{13,15}\phi|^{4}U_{13,15}\phi))(x_{6})(\ol{U_{7,15}(|\phi|^{4}\phi)})(x_{6}')
|U_{7,15}\phi\rangle\langle
U_{7,15}\phi|^{\otimes 4}
\end{align*}
Finally adding $U^{(1)}_{1,3}B_{1;2,3}^{+}U_{3,5}^{(3)}B_{1;4,5}^{+}U_{5,7}^{(5)}B_{1;6,7}^{-}$ gives
%\begin{align*}
%=&(U_{5,15}\phi)(x_{1})\lrs{\ol{U_{5,7}\lrc{(|U_{7,15}\phi|^{2}U_{7,15}\phi)(\ol{U_{7,13}(|U_{13,15}\phi|^{4}U_{13,15}\phi)})
%(U_{7,15}(|\phi|^{4}\phi))}}}(x_{1}')\times\\
%&(U_{5,15}\phi)(x_{2})(\ol{U_{5,9}(|U_{9,15}\phi|^{4}U_{9,15}\phi)})(x_{2}')\times\\
%&(U_{5,11}(|U_{11,15}\phi|^{4}U_{11,15}\phi))(x_{3})(\ol{U_{5,15}\phi})(x_{3}')
%|U_{5,15}\phi\rangle\langle
%U_{5,15}\phi|^{\otimes 2}
%\end{align*}
%
%\begin{align*}
%=&(U_{3,5}(|U_{5,15}\phi|^{4}U_{5,15}\phi))(x_{1})
%\lrs{\ol{U_{3,7}\lrc{(|U_{7,15}\phi|^{2}U_{7,15}\phi)(\ol{U_{7,13}(|U_{13,15}\phi|^{4}U_{13,15}\phi)})
%(U_{7,15}(|\phi|^{4}\phi))}}}(x_{1}')\times\\
%&(U_{3,15}\phi)(x_{2})(\ol{U_{3,9}(|U_{9,15}\phi|^{4}U_{9,15}\phi)})(x_{2}')\times\\
%&(U_{3,11}(|U_{11,15}\phi|^{4}U_{11,15}\phi))(x_{3})(\ol{U_{3,15}\phi})(x_{3}')
%\end{align*}

\begin{align*}
&U_{1,3}\lrc{(U_{3,5}(|U_{5,15}\phi|^{4}U_{5,15}\phi))(|U_{3,15}\phi|^{2})(\ol{U_{3,9}(|U_{9,15}\phi|^{4}U_{9,15}\phi)})
(U_{3,11}(|U_{11,15}\phi|^{4}U_{11,15}\phi))}(x_{1})\\
&\cdot\lrs{\ol{U_{1,7}\lrc{(|U_{7,15}\phi|^{2}U_{7,15}\phi)(\ol{U_{7,13}(|U_{13,15}\phi|^{4}U_{13,15}\phi)})
(U_{7,15}(|\phi|^{4}\phi))}}}(x_{1}')
\end{align*}
which shows that $J_{\mu,sgn}^{(15)}(|\phi\rangle\langle \phi|^{\otimes 15})(t_{1},\underline{t}_{15})=D^{(0)}(t_{1},\underline{t}_{15})$.
\end{example}

\subsection{Compatible Time Integration Domain}\label{subsection:Compatible Time Integration Domain}
To enable the application of $U$-$V$ multilinear estimates, we have to take into account the compatible time integration domain. Combining with the Duhamel tree, we present a general algorithm to compute the Duhamel expansion with the compatible time integration domain.
\begin{definition}\label{definition:compatible time integration domain}
Define the compatible time integration domain as follows
\begin{align}\label{equ:compatible time integration domain}
T_{C}(\mu,sgn)=\lr{t_{2j+1}\geq t_{2l+1}:D^{(2l)}\to D^{(2j)}}.
\end{align}
where $D^{(2l)}\to D^{(2j)}$ denotes that $D^{(2l)}$ is the child of $D^{(2j)}$. Moreover, we say that $D^{(2l)}$ is the offspring of $D^{(2j)}$ if there exist $2l_{1}$,...,$2l_{r}$ such that $D^{(2l)}\to D^{(2l_{1})}\to \ccc \to D^{(2l_{r})}\to D^{(2j)}$.
\end{definition}

\begin{example} \label{example:duhamel integral with the compatible time domain}
Continuing the Example $\ref{example:duhamel expansion}$, we will expand
\begin{align*}
\int_{T_{C}}J_{\mu,sgn}^{(15)}(|\phi\rangle \langle \phi|^{\otimes 15})(t_{1},\underline{t}_{15})d\underline{t}_{15}.
\end{align*}
From the $D$-tree (Fig. $\ref{figure:duhamel tree}$), the compatible time integration domain is as follows
\begin{align*}
\int_{t_{3}=0}^{t_{1}}\int_{t_{7}=0}^{t_{1}}
\int_{t_{5}=0}^{t_{3}}\int_{t_{9}=0}^{t_{3}}\int_{t_{11}=0}^{t_{3}}
\int_{t_{13}=0}^{t_{7}}\int_{t_{15}=0}^{t_{7}}.
\end{align*}
Write $\int_{t_{15}=0}^{t_{1}}$ on the outside and hence $\int_{t_{7}=0}^{t_{1}}$
changes into $\int_{t_{7}=t_{15}}^{t_{1}}$. Then it turns to
\begin{align*}
\int_{t_{15}=0}^{t_{1}}
\int_{t_{3}=0}^{t_{1}}\int_{t_{7}=t_{15}}^{t_{1}}
\int_{t_{5}=0}^{t_{3}}\int_{t_{9}=0}^{t_{3}}\int_{t_{11}=0}^{t_{3}}
\int_{t_{13}=0}^{t_{7}}.
\end{align*}
So we can rewrite
\begin{align*}
&\int_{T_{C}}J_{\mu,sgn}^{(15)}(|\phi\rangle \langle \phi|^{\otimes 15})(t_{1},\underline{t}_{15})d\underline{t}_{15}\\
=&\int_{t_{15}=0}^{t_{1}}
\lrc{
\int_{t_{3}=0}^{t_{1}}
\int_{t_{5}=0}^{t_{3}}\int_{t_{9}=0}^{t_{3}}\int_{t_{11}=0}^{t_{3}} U_{1}D^{(2)}}
\lrc{
\int_{t_{7}=t_{15}}^{t_{1}} \int_{t_{13}=0}^{t_{7}} \ol{U_{1}\ol{D^{(6)}}}}\\
=&\int_{t_{15}=0}^{t_{1}} \int_{t_{3}=0}^{t_{1}}
U_{1,3}\lrc{
\lrs{U_{3}\int_{t_{5}=0}^{t_{3}}D^{(4)}}
U_{3,15}\phi
\lrs{\ol{U_{3}\int_{t_{9}=0}^{t_{3}}\ol{D^{(8)}}}}
\lrs{U_{3}\int_{t_{11}=0}^{t_{3}}D^{(10)}}
\ol{U_{3,15}\phi}}\\
&\cdot
\int_{t_{7}=t_{15}}^{t_{1}}
\ol{U_{1,7}
\lrc{(U_{7,15}\phi)\lrs{\ol{U_{7}\int_{t_{13}=0}^{t_{7}}D^{(12)}}}
(U_{7}\ol{D^{(14)}})(\ol{U_{7,15}\phi})(U_{7,15}\phi)
}
}
\end{align*}
where $D^{(2j)}$ is shown in Example \ref{example:duhamel expansion}. We can see that all the Duhamel structures are fully compatible with $U$-$V$ multilinear estimates, which we will show in Section \ref{subsection:Estimates using the $U$-$V$ multilinear estimates}.
\end{example}

Next, we give a general form of the algorithm.

\begin{algorithm}[From $D$-tree to Duhamel integration]\label{algorithm:from d-tree to duhamel integration}
~\\
\hspace*{1em}$(1)$ Let
$$Q^{(2k)}(t_{2k+1})=D^{(2k)}(t_{2k+1}).$$
and replace $D^{(2k)}$ by $Q^{(2k)}(t_{2k+1})$ in the $D$-tree.

$(2)$ Set counter $l=k-1$.

$(3)$ Given $l$, there exists only one $j$ such that $D^{(2l)}\to D^{(2j)}$. Then there will be four cases as follows.

Case 1. $D^{(2k)}$ is the offspring of $D^{(2l)}$ and $sgn(2l)=+$. Then set
\begin{align*}
&Q^{(2l)}(t_{2j+1},t_{2k+1})\\
=&\int_{t_{2l+1}=t_{2k+1}}^{t_{2j+1}}U_{-2l-1}\lrc{\lrs{U_{2l+1}C_{1}}\lrs{U_{2l+1}C_{2}}
\lrs{\ol{U_{2l+1}\ol{C_{3}}}}\lrs{U_{2l+1}C_{4}}\lrs{\ol{U_{2l+1}\ol{C_{5}}}}}dt_{2l+1}.
\end{align*}

Case 2. $D^{(2k)}$ is the offspring of $D^{(2l)}$ and $sgn(2l)=-$. Then set
\begin{align*}
&Q^{(2l)}(t_{2j+1},t_{2k+1})\\
=&\int_{t_{2l+1}=t_{2k+1}}^{t_{2j+1}}\ol{U_{-2l-1}
\lrc{\lrs{U_{2l+1}\ol{C_{1}}}\lrs{\ol{U_{2l+1}C_{2}}}
\lrs{U_{2l+1}\ol{C_{3}}}\lrs{\ol{U_{2l+1}C_{4}}}\lrs{U_{2l+1}\ol{C_{5}}}}}dt_{2l+1}.
\end{align*}

Case 3. $D^{(2k)}$ is not the offspring of $D^{(2l)}$ and $sgn(2l)=+$. Then set
\begin{align*}
&Q^{(2l)}(t_{2j+1},t_{2k+1})\\
=&\int_{t_{2l+1}=0}^{t_{2j+1}}
U_{-2l-1}\lrc{\lrs{U_{2l+1}C_{1}}\lrs{U_{2l+1}C_{2}}
\lrs{\ol{U_{2l+1}\ol{C_{3}}}}\lrs{U_{2l+1}C_{4}}\lrs{\ol{U_{2l+1}\ol{C_{5}}}}}dt_{2l+1}.
\end{align*}

Case 4. $D^{(2k)}$ is not the offspring of $D^{(2l)}$ and $sgn(2l)=-$. Then set
\begin{align*}
&Q^{(2l)}(t_{2j+1},t_{2k+1})\\
=&\int_{t_{2l+1}=0}^{t_{2j+1}}\ol{U_{-2l-1}
\lrc{\lrs{U_{2l+1}\ol{C_{1}}}\lrs{\ol{U_{2l+1}C_{2}}}
\lrs{U_{2l+1}\ol{C_{3}}}\lrs{\ol{U_{2l+1}C_{4}}}\lrs{U_{2l+1}\ol{C_{5}}}}}dt_{2l+1}.
\end{align*}
In the above, $C_{i}$ is the $i$-th child of $D^{(2l)}$ in the $D$-tree.

$(4)$ Update the $D$-tree by using $Q^{(2l)}(t_{2j+1},t_{2k+1})$ to replace $D^{(2l)}$.

$(5)$ Set $l=l-1$. If $l=0$, set
$$Q^{(0)}(t_{1})=\int_{t_{2k+1}=0}^{t_{1}}(U_{1}C_{l})(\ol{U_{1}\ol{C_{r}}})dt_{2k+1}$$
where $C_{l}$ or $C_{r}$ is the left or right child of $D^{(0)}$ in the updated $D$-tree, and stop, otherwise go to step $(3)$.
\end{algorithm}
Hence, we arrive at a representation as follows.
\begin{proposition}\label{prop:from d-tree to duhamel integration}
\begin{align}
\int_{T_{C}}J_{\mu,sgn}^{(2k+1)}(|\phi\rangle \langle \phi|^{(2k+1)})
(t_{1},\underline{t}_{2k+1})d\underline{t}_{2k+1}=Q^{(0)}(t_{1}).
\end{align}
\end{proposition}
\begin{proof}
It follows from Algorithm $\ref{algorithm:from d-tree to duhamel integration}$.
\end{proof}

\subsection{Estimates using the $U$-$V$ multilinear estimates}\label{subsection:Estimates using the $U$-$V$ multilinear estimates}

Referring to the standard text \cite{koch2014dispersive} for the definition of $U_{t}^{p}$ and $V_{t}^{p}$, we define $X^{s}([0,T))$ and $Y^{s}([0,T))$ to be the spaces of all functions $u:[0,T)\mapsto H^{s}(\T^{d})$ such that for every $\xi\in \Z^{d}$ the map $t\mapsto \widehat{e^{-it\Delta}u(t)}(\xi)$ is in $U^{2}([0,T);\C)$ and $V_{rc}^{2}([0,T);\C)$, respectively, with norms given by
$$\n{u}_{X^{s}([0,T))}:=\lrs{\sum_{\xi\in \Z^{d}}\lra{\xi}^{2s}
\n{\widehat{e^{-it\Delta}u(t)}(\xi)}_{U^{2}}^{2}}^{1/2},$$
$$\n{u}_{Y^{s}([0,T))}:=\lrs{\sum_{\xi\in \Z^{d}}\lra{\xi}^{2s}
\n{\widehat{e^{-it\Delta}u(t)}(\xi)}_{V^{2}}^{2}}^{1/2}$$
as in \cite{herr2011global,herr2014strichartz,ionescu2012the,killip2016scale}. In particular, we have the usual properties,
\begin{align}
&\n{u}_{L_{t}^{\wq}H_{x}^{s}}\lesssim \n{u}_{Y^{s}}\lesssim \n{u}_{X^{s}},\label{equ:u-v and sobolev}\\
&\n{e^{it\Delta}f}_{Y^{s}}\leq  \n{e^{it\Delta}f}_{X^{s}}\leq\n{f}_{H^{s}}, \label{equ:u-v and sobolev, special case}
\end{align}
which were proved in \cite[Propositions 2.8-2.10]{herr2011global}.

By quintilinear estimates in Lemma $\ref{lemma:multilinear estimate d>=5}$ and the trivial estimate $\n{u}_{Y^{s}}\lesssim \n{u}_{X^{s}}$, we have that
\begin{align}\label{equ:multilinear estimate, high frequency, low regulariy}
&\bbn{\int_{a}^{t}e^{i(t-s)\Delta}(\wt{u}_{1}\wt{u}_{2}\wt{u}_{3}\wt{u}_{4}\wt{u}_{5})ds}_{X^{\frac{d-5}{2}}}\\
\leq &C \n{u_{1}}_{X^{\frac{d-5}{2}}}
\n{u_{2}}_{X^{\frac{d-1}{2}}}\n{u_{3}}_{X^{\frac{d-1}{2}}}
\n{u_{4}}_{X^{\frac{d-1}{2}}}\n{u_{5}}_{X^{\frac{d-1}{2}}}\notag
\end{align}

\begin{align}\label{equ:multilinear estimate, high frequency, high regulariy}
&\bbn{\int_{a}^{t}e^{i(t-s)\Delta}(\wt{u}_{1}\wt{u}_{2}\wt{u}_{3}\wt{u}_{4}\wt{u}_{5})ds}_{X^{\frac{d-1}{2}}}\\
\leq& C\n{u_{1}}_{X^{\frac{d-1}{2}}}
\n{u_{2}}_{X^{\frac{d-1}{2}}}\n{u_{3}}_{X^{\frac{d-1}{2}}}
\n{u_{4}}_{X^{\frac{d-1}{2}}}\n{u_{5}}_{X^{\frac{d-1}{2}}}\notag
\end{align}

\begin{align}\label{equ:multilinear estimate, low frequency, low regulariy}
&\bbn{\int_{a}^{t}e^{i(t-s)\Delta}(\wt{u}_{1}\wt{u}_{2}\wt{u}_{3}\wt{u}_{4}\wt{u}_{5})ds}_{X^{\frac{d-5}{2}}}\\
\leq &C \n{u_{1}}_{X^{\frac{d-5}{2}}}
\lrs{T^{\frac{1}{2(d+3)}}M_{0}^{\frac{2d+3}{2(d+3)}}\n{P_{\leq M_{0}}u_{2}}_{X^{\frac{d-1}{2}}}+\n{P_{>M_{0}}u_{2}}_{X^{\frac{d-1}{2}}}}\n{u_{3}}_{X^{\frac{d-1}{2}}}
\n{u_{4}}_{X^{\frac{d-1}{2}}}\n{u_{5}}_{X^{\frac{d-1}{2}}}\notag
\end{align}

\begin{align}\label{equ:multilinear estimate, low frequency, high regulariy}
&\bbn{\int_{a}^{t}e^{i(t-s)\Delta}(\wt{u}_{1}\wt{u}_{2}\wt{u}_{3}\wt{u}_{4}\wt{u}_{5})ds}_{X^{\frac{d-1}{2}}}\\
\leq& C\n{u_{1}}_{X^{\frac{d-1}{2}}}
\lrs{T^{\frac{1}{2(d+3)}}M_{0}^{\frac{2d+3}{2(d+3)}}\n{P_{\leq M_{0}}u_{2}}_{X^{\frac{d-1}{2}}}+\n{P_{>M_{0}}u_{2}}_{X^{\frac{d-1}{2}}}}\n{u_{3}}_{X^{\frac{d-1}{2}}}
\n{u_{4}}_{X^{\frac{d-1}{2}}}\n{u_{5}}_{X^{\frac{d-1}{2}}}\notag
\end{align}
where $\wt{u}\in \lr{u,\ol{u}}$.

Before moving into the estimate part, we first mark the $D$-tree as a preparation, as we will use the above $U$-$V$ multilinear estimates according to the marked $D$-tree. Here, we give a general algorithm to mark a $D$-Tree.
\begin{algorithm}[Marked $D$-Tree] \label{algorithm:marked d-tree}
~\\
\hspace*{1em}$(1)$ We put a subscript $R$ at $D^{(2k)}$, that is, $D_{R}^{(2k)}$. Here, we use the subscript $R$ to denote the roughest term.

$(2)$ Set counter $l=k-1$. If $D^{(2k)}$ is the offspring (see Definition \ref{definition:compatible time integration domain}) of $D^{(2l)}$, put a subscript $R$ at $D^{(2l)}$, that is, $D_{R}^{(2l)}$. Moreover, if one of the children of $D^{(2l)}$ is $F$, then put a subscript $\phi$ at $D^{(2l)}$, that is, $D_{\phi}^{(2l)}$ or $D^{(2l)}_{\phi,R}$.

$(3)$ Set $l=l-1$. If $l=0$, then stop, otherwise go to step $(2)$.
\end{algorithm}

\begin{example} \label{example:estimate part}
 We estimate the Duhamel expansion in Example $\ref{example:duhamel expansion}$ with the corresponding time integration domain to show how to apply the $U$-$V$ multilinear estimates. First, Applying Algorithm \ref{algorithm:marked d-tree} to the $D$-tree in Fig. \ref{figure:duhamel tree}, we obtain a marked $D$-tree as in Fig. \ref{figure:marked duhamel tree}.
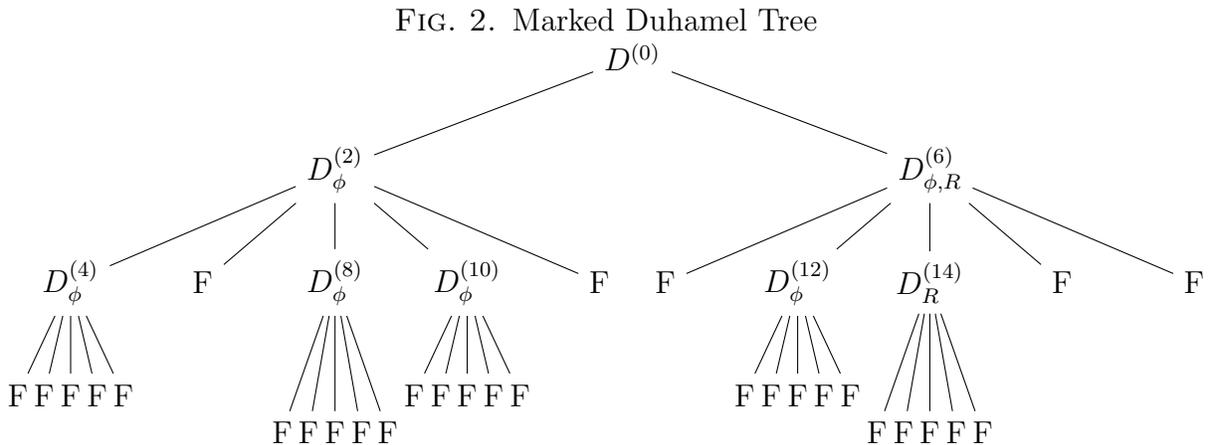
\begin{figure}[htpb]
\caption{Marked Duhamel Tree}
\label{figure:marked duhamel tree}
\begin{tikzpicture}[ box/.style={circle, minimum width =10pt, minimum height=10pt,
inner sep=1pt,draw=black}]
\node {$D^{(0)}$}[sibling distance=45pt]

child {node{$D_{\phi}^{(2)}$} [sibling distance=50pt]
child {node{$D_{\phi}^{(4)}$} [sibling distance=10pt]
child {node{F}}
child {node{F}}
child {node{F}}
child {node{F}}
child {node{F}}
}
child {node{F}}
child {node{$D_{\phi}^{(8)}$} [sibling distance=10pt,level distance=2cm]
child {node{F}}
child {node{F}}
child {node{F}}
child {node{F}}
child {node{F}}
}
child {node{$D_{\phi}^{(10)}$} [sibling distance=10pt]
child {node{F}}
child {node{F}}
child {node{F}}
child {node{F}}
child {node{F}}
}
child {node{F}}
}
child [missing]
child [missing]
child [missing]
child [missing]
child {node{$D_{\phi,R}^{(6)}$} [sibling distance=50pt]
child {node{F}}
child {node{$D_{\phi}^{(12)}$} [sibling distance=10pt]
child {node{F}}
child {node{F}}
child {node{F}}
child {node{F}}
child {node{F}}
}
child {node{$D_{R}^{(14)}$} [sibling distance=10pt,level distance=2cm]
child {node{F}}
child {node{F}}
child {node{F}}
child {node{F}}
child {node{F}}
}
child {node{F}}
child {node{F}}
}
;
\end{tikzpicture}
\end{figure}

Next, we get into the estimate part. By Proposition
$\ref{prop:from d-tree to duhamel integration}$, it suffices to estimate $Q^{(0)}(t_{1})$.
Combining the $D$-tree (Fig. $\ref{figure:duhamel tree}$) and Algorithm $\ref{algorithm:from d-tree to duhamel integration}$, we obtain
\begin{align*}
&Q^{(0)}(t_{1})=\int_{t_{15}=0}^{t_{1}}(U_{1}Q^{(2)}(t_{1},t_{15}))(\ol{U_{1}\ol{Q^{(6)}(t_{1},t_{15})}})dt_{15},\\
&Q^{(2)}(t_{1},t_{15})=\int^{t_{1}}_{t_{3}=0}U_{-3}\lrc{\lrs{U_{3}Q^{(4)}}\lrs{U_{3,15}\phi}
\lrs{\ol{U_{3}\ol{Q^{(8)}}}}\lrs{U_{3}Q^{(10)}}\lrs{\ol{U_{3,15}\phi}}}dt_{3},\\
&Q^{(6)}(t_{1},t_{15})=\int_{t_{7}=t_{15}}^{t_{1}}\ol{U_{-7}
\lrc{\lrs{U_{7,15}\phi}\lrs{\ol{U_{7}Q^{(12)}}}
\lrs{U_{7}\ol{Q^{(14)}}}\lrs{\ol{U_{7,15}\phi}}\lrs{U_{7,15}\phi}}}dt_{7}.
\end{align*}

At first, we use Minkowski to obtain
\begin{align*}
&\n{\lra{\nabla_{x_{1}}}^{
\frac{d-5}{2}}\lra{\nabla_{x_{1}'}}^{\frac{d-5}{2}}Q^{(0)}(t_{1})}
_{L_{T}^{\wq}L_{x_{1}}^{2}L_{x_{1}'}^{2}}\\
\leq& \int_{0}^{T}\n{U_{1}Q^{(2)}(t_{1},t_{15})}_{L_{t_{1}}^{\wq}H_{x_{1}}^{\frac{d-5}{2}}}
\n{\ol{U_{1}\ol{Q^{(6)}(t_{1},t_{15})}}}_{L_{t_{1}}^{\wq}H_{x_{1}'}^{\frac{d-5}{2}}}dt_{15}
\end{align*}
Note that $D_{\phi}^{(2)}$ carries no $R$ subscript, so we can bump it to $H^{\frac{d-1}{2}}$ and then use estimate $(\ref{equ:u-v and sobolev})$ to get
\begin{align*}
\leq&\int_{0}^{T}\n{U_{1}Q^{(2)}(t_{1},t_{15})}_{X^{\frac{d-1}{2}}}
\n{U_{1}\ol{Q^{(6)}(t_{1},t_{15})}}_{X^{\frac{d-5}{2}}}dt_{15}
\end{align*}

By multilinear estimate $(\ref{equ:multilinear estimate, low frequency, high regulariy})$,
\begin{align*}
\n{U_{1}Q^{(2)}}_{X^{\frac{d-1}{2}}}\leq  &C \n{U_{3}Q^{(4)}}_{X^{\frac{d-1}{2}}}
\lrs{T^{\frac{1}{2(d+3)}}M_{0}^{\frac{2d+3}{2(d+3)}}\n{P_{\leq M_{0}}U_{3,15}\phi}_{X^{\frac{d-1}{2}}}+\n{P_{>M_{0}}U_{3,15}\phi}_{X^{\frac{d-1}{2}}}}\\
&\n{U_{3}\ol{Q^{(8)}}}_{X^{\frac{d-1}{2}}}
\n{U_{3}Q^{(10)}}_{X^{\frac{d-1}{2}}}\n{U_{3,15}\phi}_{X^{\frac{d-1}{2}}}.
\end{align*}

As $D_{\phi,R}^{(6)}$ carries subscript $\phi$ and $R$, we use multilinear estimate $(\ref{equ:multilinear estimate, low frequency, low regulariy})$ to get
\begin{align*}
\n{U_{1}\ol{Q^{(6)}}}_{X^{\frac{d-5}{2}}}\leq &C\n{U_{7}\ol{Q^{(14)}}}_{X^{\frac{d-5}{2}}}
\lrs{T^{\frac{1}{2(d+3)}}M_{0}^{\frac{2d+3}{2(d+3)}}\n{P_{\leq M_{0}}U_{7,15}\phi}_{X^{\frac{d-1}{2}}}+\n{P_{>M_{0}}U_{7,15}\phi}_{X^{\frac{d-1}{2}}}}\\
&\n{U_{7}Q^{(12)}}_{X^{\frac{d-1}{2}}}
\n{U_{7,15}\phi}_{X^{\frac{d-1}{2}}}\n{U_{7,15}\phi}_{X^{\frac{d-1}{2}}}.
\end{align*}
From Algorithm $\ref{algorithm:from d-tree to duhamel integration}$, we have
\begin{align*}
&Q^{(4)}=\int_{t_{5}=0}^{t_{3}}U_{-5}\lrc{\lrs{U_{5,15}\phi}\lrs{U_{5,15}\phi}
\lrs{\ol{U_{5,15}\phi}}\lrs{U_{5,15}\phi}\lrs{\ol{U_{5,15}\phi}}}dt_{5},\\
&Q^{(8)}=\int_{t_{9}=0}^{t_{3}}\ol{U_{-9}
\lrc{\lrs{U_{9,15}\phi}\lrs{\ol{U_{9,15}\phi}}
\lrs{U_{9,15}\phi}\lrs{\ol{U_{9,15}\phi}}\lrs{U_{9,15}\phi}}}dt_{9},\\
&Q^{(10)}=\int_{t_{11}=0}^{t_{3}}U_{-11}\lrc{\lrs{U_{11,15}\phi}\lrs{U_{11,15}\phi}
\lrs{\ol{U_{11,15}\phi}}\lrs{U_{11,15}\phi}\lrs{\ol{U_{11,15}\phi}}}dt_{11},\\
&Q^{(12)}=\int_{t_{13}=0}^{t_{7}}U_{-13}\lrc{(U_{13,15}\phi)(U_{13,15}\phi)(\ol{U_{13,15}\phi})
(U_{13,15}\phi)(\ol{U_{13,15}\phi})}dt_{13},\\
&Q^{(14)}=\ol{U_{-15}(|\phi|^{4}\phi)}.
\end{align*}

Notice that $D_{\phi}^{(4)}$, $D_{\phi}^{(6)}$, $D_{\phi}^{(8)}$ and $D_{\phi}^{(10)}$ only carry subscript $\phi$, so we use multilinear estimate $(\ref{equ:multilinear estimate, low frequency, high regulariy})$ to obtain
\begin{align*}
\n{U_{3}Q^{(4)}}_{X^{\frac{d-1}{2}}}\leq& C\n{U_{5,15}\phi}_{X^{\frac{d-1}{2}}}
\lrs{T^{\frac{1}{2(d+3)}}M_{0}^{\frac{2d+3}{2(d+3)}}\n{P_{\leq M_{0}}U_{5,15}\phi}_{X^{\frac{d-1}{2}}}+\n{P_{>M_{0}}U_{5,15}\phi}_{X^{\frac{d-1}{2}}}}\\
&\n{U_{5,15}\phi}_{X^{\frac{d-1}{2}}}
\n{U_{5,15}\phi}_{X^{\frac{d-1}{2}}}\n{U_{5,15}\phi}_{X^{\frac{d-1}{2}}}\\
\leq &C\n{\phi}_{H^{\frac{d-1}{2}}}^{4}
\lrs{T^{\frac{1}{2(d+3)}}M_{0}^{\frac{2d+3}{2(d+3)}}\n{P_{\leq M_{0}}\phi}_{H^{\frac{d-1}{2}}}+\n{P_{>M_{0}}\phi}_{H^{\frac{d-1}{2}}}}
\end{align*}
\begin{align*}
\n{U_{3}\ol{Q^{(8)}}}_{X^{\frac{d-1}{2}}}\leq&
 C\n{U_{9,15}\phi}_{X^{\frac{d-1}{2}}}
\lrs{T^{\frac{1}{2(d+3)}}M_{0}^{\frac{2d+3}{2(d+3)}}\n{P_{\leq M_{0}}U_{9,15}\phi}_{X^{\frac{d-1}{2}}}+\n{P_{>M_{0}}U_{9,15}\phi}_{X^{\frac{d-1}{2}}}}\\
&\n{U_{9,15}\phi}_{X^{\frac{d-1}{2}}}
\n{U_{9,15}\phi}_{X^{\frac{d-1}{2}}}\n{U_{9,15}\phi}_{X^{\frac{d-1}{2}}}\\
\leq &C\n{\phi}_{H^{\frac{d-1}{2}}}^{4}
\lrs{T^{\frac{1}{2(d+3)}}M_{0}^{\frac{2d+3}{2(d+3)}}\n{P_{\leq M_{0}}\phi}_{H^{\frac{d-1}{2}}}+\n{P_{>M_{0}}\phi}_{H^{\frac{d-1}{2}}}}
\end{align*}
\begin{align*}
\n{U_{3}Q^{(10)}}_{X^{\frac{d-1}{2}}}\leq& C\n{U_{11,15}\phi}_{X^{\frac{d-1}{2}}}
\lrs{T^{\frac{1}{2(d+3)}}M_{0}^{\frac{2d+3}{2(d+3)}}\n{P_{\leq M_{0}}U_{11,15}\phi}_{X^{\frac{d-1}{2}}}+\n{P_{>M_{0}}U_{11,15}\phi}_{X^{\frac{d-1}{2}}}}\\
&\n{U_{11,15}\phi}_{X^{\frac{d-1}{2}}}
\n{U_{11,15}\phi}_{X^{\frac{d-1}{2}}}\n{U_{11,15}\phi}_{X^{\frac{d-1}{2}}}\\
\leq &C\n{\phi}_{H^{\frac{d-1}{2}}}^{4}
\lrs{T^{\frac{1}{2(d+3)}}M_{0}^{\frac{2d+3}{2(d+3)}}\n{P_{\leq M_{0}}\phi}_{H^{\frac{d-1}{2}}}+\n{P_{>M_{0}}\phi}_{H^{\frac{d-1}{2}}}}
\end{align*}
\begin{align*}
\n{U_{7}Q^{(12)}}_{X^{\frac{d-1}{2}}}\leq& C\n{U_{13,15}\phi}_{X^{\frac{d-1}{2}}}
\lrs{T^{\frac{1}{2(d+3)}}M_{0}^{\frac{2d+3}{2(d+3)}}\n{P_{\leq M_{0}}U_{13,15}\phi}_{X^{\frac{d-1}{2}}}+\n{P_{>M_{0}}U_{13,15}\phi}_{X^{\frac{d-1}{2}}}}\\
&\n{U_{13,15}\phi}_{X^{\frac{d-1}{2}}}
\n{U_{13,15}\phi}_{X^{\frac{d-1}{2}}}\n{U_{13,15}\phi}_{X^{\frac{d-1}{2}}}\\
\leq &C\n{\phi}_{H^{\frac{d-1}{2}}}^{4}
\lrs{T^{\frac{1}{2(d+3)}}M_{0}^{\frac{2d+3}{2(d+3)}}\n{P_{\leq M_{0}}\phi}_{H^{\frac{d-1}{2}}}+\n{P_{>M_{0}}\phi}_{H^{\frac{d-1}{2}}}}
\end{align*}

Finally, to deal with the roughest term $Q^{(14)}$, we use Sobolev inequality $(\ref{equ:sobolev inequality, multilinear estimate, endpoint})$,
\begin{align*}
\n{U_{7}\ol{Q^{(14)}}}_{X^{\frac{d-5}{2}}}=\n{U_{7,15}(|\phi|^{4}\phi)}_{X^{\frac{d-5}{2}}}
\leq \n{|\phi|^{4}\phi}_{H^{\frac{d-5}{2}}}\leq C\n{\phi}_{H^{\frac{d-1}{2}}}^{5}.
\end{align*}
Together with the above estimates, we arrive at
\begin{align*}
&\n{\lra{\nabla_{x_{1}}}^{
\frac{d-5}{2}}\lra{\nabla_{x_{1}'}}^{\frac{d-5}{2}}Q^{(0)}(t_{1})}
_{L_{T}^{\wq}L_{x_{1}}^{2}L_{x_{1}'}^{2}}\\
\leq & C^{7}\int_{0}^{T}\n{\phi}_{H^{\frac{d-1}{2}}}^{24}\lrs{T^{\frac{1}{2(d+3)}}M_{0}^{\frac{2d+3}{2(d+3)}}\n{P_{\leq M_{0}}\phi}_{H^{\frac{d-1}{2}}}+\n{P_{>M_{0}}\phi}_{H^{\frac{d-1}{2}}}}^{6}dt_{15}.
\end{align*}
\end{example}

From Example \ref{example:estimate part}, one can immediately tell that a decay power comes from estimates $(\ref{equ:multilinear estimate, low frequency, low regulariy})$ and $(\ref{equ:multilinear estimate, low frequency, high regulariy})$. Actually, such a decay power is at least proportional to $k$.
\begin{definition}[\cite{chen2019the,chen2020unconditional}]
For $l<k$, we say the $l$-th coupling is an unclogged coupling, if one of the children of $D^{(2l)}$ is $F$. If the $l$-th coupling is not unclogged, we will call it a congested coupling.
\end{definition}

\begin{lemma}\cite[Lemma 5.14]{chen2019the}
For large $k$, there are at least $\frac{4}{5}k$ unclogged couplings in $k$ couplings.
\end{lemma}

The main result of this section is the following proposition.
\begin{proposition}\label{prop:estimate for the compatible time integration domain}
Let $\ga^{(k)}(t)=\int |\phi\rangle \langle \phi|^{\otimes k} d\mu_{t}(\phi)$. Then we have
\begin{align*}
&\bbn{\lra{\nabla_{x_{1}}}^{\frac{d-5}{2}}\lra{\nabla_{x_{1}'}}^{\frac{d-5}{2}}\int_{T_{C}}J_{\mu,sgn}^{(2k+1)}
(\ga^{(2k+1)})
(t_{1},\underline{t}_{2k+1})d\underline{t}_{2k+1}}
_{L_{T}^{\wq}L_{x_{1}}^{2}L_{x_{1}'}^{2}}\\
\leq &C^{k}\int_{0}^{T}\int \n{\phi}_{H^{\frac{d-1}{2}}}^{\frac{16}{5}k+2}\lrs{T^{\frac{1}{2(d+3)}}M_{0}^{\frac{2d+3}{2(d+3)}}\n{P_{\leq M_{0}}\phi}_{H^{\frac{d-1}{2}}}+\n{P_{>M_{0}}\phi}_{H^{\frac{d-1}{2}}}}^{\frac{4k}{5}}
d|\mu_{t_{2k+1}}|(\phi)dt_{2k+1}.
\end{align*}
\end{proposition}
\begin{proof}
We rewrite
\begin{align*}
&\n{\lra{\nabla_{x_{1}}}^{\frac{d-5}{2}}\lra{\nabla_{x_{1}'}}^{\frac{d-5}{2}}\int_{T_{C}}J_{\mu,sgn}^{(2k+1)}
(\ga^{(2k+1)})
(t_{1},\underline{t}_{2k+1})d\underline{t}_{2k+1}}
_{L_{T}^{\wq}L_{x_{1}}^{2}L_{x_{1}'}^{2}}\\
=&\n{\lra{\nabla_{x_{1}}}^{\frac{d-5}{2}}\lra{\nabla_{x_{1}'}}^{\frac{d-5}{2}}\int_{T_{C}}\int J_{\mu,sgn}^{(2k+1)}
(|\phi\rangle \langle \phi|^{\otimes (2k+1)})
(t_{1},\underline{t}_{2k+1})d\mu_{t_{2k+1}}(\phi)d\underline{t}_{2k+1}}
_{L_{T}^{\wq}L_{x_{1}}^{2}L_{x_{1}'}^{2}}
\end{align*}
By Proposition $\ref{prop:from d-tree to duhamel integration}$, we then use Minkowski and estimate $(\ref{equ:u-v and sobolev})$
\begin{align*}
\leq& \int_{0}^{T}\int \n{U_{1}C_{l}(t_{1},t_{2k+1})}_{L_{t_{1}}^{\wq}H_{x_{1}}^{\frac{d-5}{2}}}
\n{\ol{U_{1}\ol{C_{r}(t_{1},t_{2k+1})}}}_{L_{t_{1}}^{\wq}H_{x_{1}'}^{\frac{d-5}{2}}}d|\mu_{t_{2k+1}}| (\phi)dt_{2k+1}\\
\leq&\int_{0}^{T}\int\n{U_{1}C_{l}(t_{1},t_{2k+1})}_{X^{\frac{d-5}{2}}}
\n{U_{1}\ol{C_{r}(t_{1},t_{2k+1})}}_{X^{\frac{d-5}{2}}}d|\mu_{t_{2k+1}}|(\phi) dt_{2k+1}.
\end{align*}
where $C_{l}$ or $C_{r}$ is the left or right child of $Q^{(0)}$ by Algorithm $\ref{algorithm:from d-tree to duhamel integration}$. Only one of $C_{l}$ and $C_{r}$ carries the subscript $R$, so bump the other one into $X^{\frac{d-1}{2}}$.

We can now present the algorithm which proves the general case.

\begin{algorithm}[Estimate]
~\\
\hspace*{1em}$(1)$ Set counter $l=1$.

$(2)$ Given $l$, there exists only one $j$ such that $D^{(2l)}\to D^{(2j)}$. There will be four cases as follows.

Case $1$. $D^{(2l)}=D_{\phi,R}^{(2l)}$. Then apply estimate $(\ref{equ:multilinear estimate, low frequency, low regulariy})$, put the factor carrying the subscript $R$ in $X^{\frac{d-5}{2}}$ and replace all the $X^{\frac{d-1}{2}}$ norm of $U\phi$ by $H^{\frac{d-1}{2}}$ norm of $\phi$.

Case $2$. $D^{(2l)}=D_{\phi}^{(2l)}$. Then apply estimate $(\ref{equ:multilinear estimate, low frequency, high regulariy})$ and replace all the $X^{\frac{d-1}{2}}$ norm of $U\phi$ by $H^{\frac{d-1}{2}}$ norm of $\phi$.

Case $3$. $D^{(2l)}=D_{R}^{(2l)}$. Then apply estimate $(\ref{equ:multilinear estimate, high frequency, low regulariy})$, put the factor carrying the subscript $R$ in $X^{\frac{d-5}{2}}$ and replace all the $X^{\frac{d-1}{2}}$ norm of $U\phi$ by $H^{\frac{d-1}{2}}$ norm of $\phi$.

Case $4$. $D^{(2l)}=D^{(2l)}$. Then apply estimate $(\ref{equ:multilinear estimate, high frequency, high regulariy})$ and replace all the $X^{\frac{d-1}{2}}$ norm of $U\phi$ by $H^{\frac{d-1}{2}}$ norm of $\phi$.

$(3)$ Set counter $l=l+1$. If $l<k$, go to step $(2)$, otherwise go to step $(4)$.

$(4)$ We are now at the $k$-th coupling and would have applied $(\ref{equ:multilinear estimate, low frequency, low regulariy})$ and $(\ref{equ:multilinear estimate, low frequency, high regulariy})$ at least $\frac{4}{5}k$ times, so we arrive at
\begin{align*}
&\bbn{\lra{\nabla_{x_{1}}}^{\frac{d-5}{2}}\lra{\nabla_{x_{1}'}}^{\frac{d-5}{2}}\int_{T_{C}}
J_{\mu,sgn}^{(2k+1)}(\ga^{(2k+1)})
(t_{1},\underline{t}_{2k+1})d\underline{t}_{2k+1}}
_{L_{T}^{\wq}L_{x_{1}}^{2}L_{x_{1}'}^{2}}\\
\leq &C^{k-1}\int_{0}^{T}\int\n{\phi}_{H^{\frac{d-1}{2}}}^{\frac{16}{5}k-3}\lrs{T^{\frac{1}{2(d+3)}}M_{0}^{\frac{2d+3}{2(d+3)}}\n{P_{\leq M_{0}}\phi}_{H^{\frac{d-1}{2}}}+\n{P_{>M_{0}}\phi}_{H^{\frac{d-1}{2}}}}^{\frac{4k}{5}}\\
&\quad \quad \quad \quad\quad \n{|\phi|^{4}\phi}_{H^{\frac{d-5}{2}}}d|\mu_{t_{2k+1}}|(\phi)dt_{2k+1}
\end{align*}
Apply Sobolev inequality $(\ref{equ:sobolev inequality, multilinear estimate, endpoint})$ to $\n{|\phi|^{4}\phi}_{H^{\frac{d-5}{2}}}$,
\begin{align*}
\leq C^{k}\int_{0}^{T}\int \n{\phi}_{H^{\frac{d-1}{2}}}^{\frac{16}{5}k+2}\lrs{T^{\frac{1}{2(d+3)}}M_{0}^{\frac{2d+3}{2(d+3)}}\n{P_{\leq M_{0}}\phi}_{H^{\frac{d-1}{2}}}+\n{P_{>M_{0}}\phi}_{H^{\frac{d-1}{2}}}}^{\frac{4k}{5}}
d|\mu_{t_{2k+1}}|(\phi)dt_{2k+1}.
\end{align*}
\end{algorithm}
\end{proof}

\section{Existence of Compatible Time Integration Domain}\label{section:Existence of Compatible Time Integration Domain}
In this section, our main goal is to prove that
\begin{equation}
\ga^{(1)}(t_{1})=\sum_{reference\ (\hat{\mu},\hat{sgn})}
\int_{T_{C}(\hat{\mu},\hat{sgn})}J_{\hat{\mu},\hat{sgn}}^{(2k+1)}(\ga^{(2k+1)})
(t_{1},\underline{t}_{2k+1})d\underline{t}_{2k+1},
\end{equation}
where the number of reference pairs can be controlled by $16^{k}$, and the summands are endowed with the compatible time integration domain $T_{C}(\hat{\mu},\hat{sgn})$ that we introduce in Section $\ref{subsection:Compatible Time Integration Domain}$.
 We divide this section into two main parts. In Section \ref{subsection:Admissible Tree}, we first recall the quintic KM board game argument and then give an introduction to an admissible tree diagram representation as a preparation for the subsequent sections. Such type of tree also gives an elaborated proof of the quintic KM board game argument.
Then in Section \ref{subsection:Signed KM Acceptable Moves}-\ref{subsection:Reference Form and Proof of Compatibility}, we prove the extended quintic KM board game argument, which allows to sort the summands in the initial
Duhamel-Born expansion $\ga^{(1)}$ into a sum of reference forms with the compatible time integration domain.

\subsection{Admissible Tree}\label{subsection:Admissible Tree}
We first give a brief review of the
quintic KM board game argument as in \cite{chen2011the,klainerman2008on}. In short, one could sort $(2k-1)!!2^{k}$ summands into a sum of upper echelon forms with the time integration domain, denoted by $D_{m}$, which is a union of
a very large number of high dimensional simplexes. The number of upper echelon forms can be controlled by $8^{k}$. Then, we give an introduction to an admissible tree diagram representation which could provide an elaborated proof of the quintic KM board game argument. Besides, one could use it to calculate $D_{m}$ explicitly, which was unknown.

 Recall that $\lr{\mu}$ is a set of maps from $\lr{2,4,...,2k}$ to $\lr{1,2,3,...,2k-1}$ satisfying $\mu(2)=1$ and $\mu(2l)<2l$ for all $2l$. For convenience, we extend the domain to $\lr{2,3,4,...,2k}$ and define
\begin{equation}\label{equ:extended definition of mu}
\mu(2l+1):=\mu(2l)\quad l\in\lr{1,2,...,k-1}.
\end{equation}
Moreover, if $\mu$ satisfies $\mu(2j)\leq \mu(2j+2)$ for $1\leq j\leq k-1$, then it is in upper echelon form as they are called in \cite{chen2011the,klainerman2008on}.

Let $P=\lr{\rho}$ be a set of permutations of $\lr{2,4,...,2k}$. To be compatible with the definition $(\ref{equ:extended definition of mu})$, we also extend the domain to $\lr{2,3,4,...,2k+1}$ and define
\begin{equation}
\rho(2l+1):=\rho(2l)+1,\quad l\in\lr{1,2,...,k}.
\end{equation}
We note that $P$ is closed under the composition and inverse operations.

Associated to each $\mu$ and $\sigma\in P$, we define the Duhamel integrals
\begin{equation}
I(\mu,\sigma,f^{(2k+1)})
=\int_{t_{1}\geq t_{\sigma(3)}\geq \ccc \geq t_{\sigma(2k+1)}}
J_{\mu}^{(2k+1)}(f^{(2k+1)})(t_{1},\underline{t}_{2k+1})d\underline{t}_{2k+1}.
\end{equation}
where
\begin{align*}
J_{\mu}^{(2k+1)}(f^{(2k+1)})(t_{1},\underline{t}_{2k+1})=&U^{(1)}(t_{1}-t_{3})B_{\mu(2);2,3}
U^{(3)}(t_{3}-t_{5})B_{\mu(4);4,5}\\
&\ccc U^{(2k-1)}(t_{2k-1}-t_{2k+1})B_{\mu(2k);2k,2k+1}f^{(2k+1)}(t_{2k+1})
\end{align*}
and $f^{(2k+1)}$ is a symmetric density.

\begin{definition}
For fixed $j\in \lr{2,3,..,k-1}$ and a permutation $\rho=(2j,2j+2)\circ (2j+1,2j+3)\in P$, if $\mu(2j)\neq \mu(2j+2)$ and $\mu(2j+2)<2j$,
we define the action as follows:
\begin{align*}
&\mu'=(2j,2j+2)\circ (2j+1,2j+3) \circ \mu \circ (2j,2j+2)\circ (2j+1,2j+3),\\
&\sigma'=(2j,2j+2)\circ (2j+1,2j+3) \circ \sigma.
\end{align*}

We call the action induced by $\rho$, which we simply denote $KM(\rho)$, a Klainerman-Machedon acceptable move in Chen-Pavlovi\'{c} format of $\mu$, or an acceptable move of $\mu$ for simplicity.

For general case, we also call a permutation $\rho$ a Klainerman-Machedon acceptable move in Chen-Pavlovi\'{c} format of $\mu$, if $\rho=\rho_{r}\circ \rho_{r-1}\circ \ccc \circ \rho_{1}$ where $\rho_{1}$ is an acceptable move of $\mu$ and $\rho_{i}=(2j_{i},2j_{i}+2)\circ (2j_{i}+1,2j_{i}+3)$ is an acceptable move of $\mu_{i}=KM(\rho_{i-1})\circ \ccc \circ KM(\rho_{1})(\mu)$ for $2\leq i\leq r$. Moreover, we define the action $(\mu',\sigma')=KM(\rho)(\mu,\sigma)$:
\begin{align*}
&\mu'=\rho \circ \mu \circ \rho^{-1},\\
&\sigma'=\rho \circ \sigma.
\end{align*}
\end{definition}

If $\mu$ and $\mu'$ are such that there exists $\rho$ as above for which $(\mu',\sigma')=KM(\rho)(\mu,\sigma)$ then we say that $\mu'$ and $\mu$ are
KM-relatable. This is an equivalence relation that partitions the set
of collapsing maps into equivalence classes.

Now, we could describe the quintic KM board game argument in \cite{chen2011the,klainerman2008on}. Namely,
for every $\mu$, there is exactly one $\mu_{m}$ in upper echelon form, which is KM-related to $\mu$ and the number of upper echelon forms can be controlled by $8^{k}$. Moreover, it follows from \cite[Lemma 7.1]{chen2011the} that
\begin{align}\label{equ:the equality under the km move}
I(\mu,\sigma,f^{(2k+1)})=I(\mu',\sigma',f^{(2k+1)}).
\end{align}
With the equality $(\ref{equ:the equality under the km move})$, one has
\begin{equation}\label{equ:the key equality for km move}
\sum_{\mu\sim \mu_{m}}I(u,id,\ga^{(2k+1)})=\int_{D_{m}}J_{\mu}^{(2k+1)}(\ga^{(2k+1)})(t_{1},\underline{t}_{2k+1})
d\underline{t}_{2k+1}
\end{equation}
where the time integration domain $D_{m}$ is a union of the simplexes $\lr{t_{1}\geq t_{\sigma(3)}\geq \ccc \geq t_{\sigma(2k+1)}}$.

The time integration
domain $D_{m}$ is obviously very complicated for large $k$, as it is a union of a very large
number of simplexes in high dimension. To calculate $D_{m}$, we construct a ternary tree with the following algorithm.

\begin{algorithm}\label{algorithm:generate an admissible tree}
~\\
\hspace*{1em}$(1)$ Set counter $j=1$.

$(2)$ Given $j$, find the indices $l$, $m$, $r$ so that $l>j$, $m>r$, $r>j$ and
\begin{align*}
&\mu(2l)=\mu(2j),\\
&\mu(2m)=2j,\\
&\mu(2r)=2j+1,
\end{align*}
and $l$, $m$ and $r$ are the minimal indices for which the above equalities hold. Then place $2l$/$2m$/$2r$ as the left/middle/right child of node $2j$ in the tree. If there is no such $l$/$m$/$r$, the node $2j$ will be missing a left/middle/right child.

$(3)$ If $j=k$, then stop, otherwise set $j=j+1$ and go to step $(2)$.
\end{algorithm}

Since it requires $\mu(2j)<2j$, one can check that every node $2j$ has a parent by induction argument. Hence, the generated tree by $\mu$, which we denote by $T(\mu)$, is a connected ternary tree with child node's label strictly larger than its parent node's label.

\begin{example}\label{example:admissible tree}
Let us work with the following example
$$
\begin{tabular}{c|ccccc}
$2j$&2&4&6&8&10\\
\hline
$\mu(2j)$ &1&1&1&2&3
\end{tabular}
$$
By Algorithm $2$, we start with $j=1$ and note that $\mu(2)=1$. For the left, middle and right child of node $2$, we need to respectively find the minimal $a>1$, $b>1$ and $c>1$ such that $\mu(2a)=1$, $\mu(2b)=2$ and $\mu(2c)=3$. In the case, it is $a=2$, $b=4$ and $c=5$, so we put $2$, $8$, and $10$ as left, middle and right children of node $2$, respectively, in the tree\footnote{We use a line to link the left child and an arrow to link the middle/right child, as we would like to emphasize the differences between the left child and the middle/right child. Besides, by this way, it is convenient to calculate the tier value which we introduce in Section \ref{subsection:Tamed Form}.}.

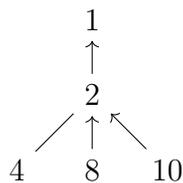
\begin{center}
\begin{tikzpicture}
\node (1) at (0,0) {1};
\node (2) at (0,-1) {2};
\node (4) at (-1,-2) {4};
\node (8) at (0,-2) {8};
\node (10) at (1,-2) {10};
\draw[<-] (1)--(2);
\draw[-] (2)--(4);
\draw[<-] (2)--(8);
\draw[<-] (2)--(10);
\end{tikzpicture}
\end{center}

Next we turn to $j=2$. Since $\mu(4)=1$, we find the minimal $a>2$, $b>2$ and $c>2$ such that $\mu(2a)=\mu(4)=1$, $\mu(2b)=4$ and $\mu(2c)=5$. We find $a=3$ and there is no such $b$ or $c$ satisfying the above condition, so we only put $6$ as the left child of node $4$ in the tree.

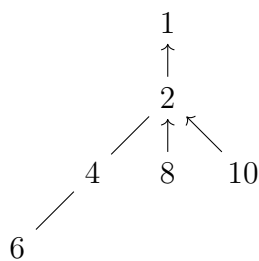
\begin{center}
\begin{tikzpicture}
\node (1) at (0,0) {1};
\node (2) at (0,-1) {2};
\node (4) at (-1,-2) {4};
\node (8) at (0,-2) {8};
\node (10) at (1,-2) {10};
\node (6) at (-2,-3) {6};
\draw[<-] (1)--(2);
\draw[-] (2)--(4);
\draw[<-] (2)--(8);
\draw[<-] (2)--(10);
\draw[-] (4)--(6);
\end{tikzpicture}
\end{center}

Since all indices appear in the tree, it is complete.

\end{example}

\begin{definition}
A ternary tree is called an admissible tree if every child node's label is strictly larger than its parent node's label. For an admissible tree, we call, the graph of the tree without any labels in its nodes, the skeleton of the tree.
\end{definition}

For example, the skeleton of the tree in Example $\ref{example:admissible tree}$ is shown as follows.
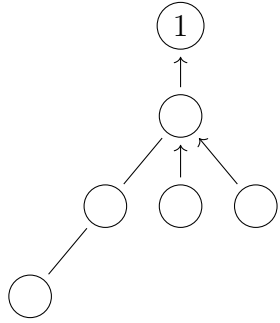
\begin{center}
\begin{tikzpicture}
\node[draw,circle,inner sep=0.1cm,outer sep=0.1cm] (1) at (0,0) {1};
\node[draw,circle,inner sep=0.2cm,outer sep=0.1cm] (2) at (0,-1.2) {};
\node[draw,circle,inner sep=0.2cm,outer sep=0.1cm] (4) at (-1,-2.4) {};
\node[draw,circle,inner sep=0.2cm,outer sep=0.1cm] (8) at (0,-2.4) {};
\node[draw,circle,inner sep=0.2cm,outer sep=0.1cm] (10) at (1,-2.4) {};
\node[draw,circle,inner sep=0.2cm,outer sep=0.1cm] (6) at (-2,-3.6) {};
\draw[<-] (1)--(2);
\draw[-] (2)--(4);
\draw[<-] (2)--(8);
\draw[<-] (2)--(10);
\draw[-] (4)--(6);
\end{tikzpicture}
\end{center}

Given an admissible ternary tree, we can uniquely reconstruct a collapsing map $\mu$ that generates it. For notational convenience, we take the following notations.
\begin{align*}
&2l\stackrel{L}{\to}2j:\text{node $2l$ is the left child of node $2j$},\\
&2l\stackrel{M}{\to}2j:\text{node $2l$ is the middle child of node $2j$},\\
&2l\stackrel{R}{\to}2j:\text{node $2l$ is the right child of node $2j$},\\
&2l\to 2j:\text{node $2l$ is a child of node $2j$}.
\end{align*}

\begin{algorithm}[From admissible tree to collapsing map]\label{algorithm:from admissible tree to collapsing map}
~\\
\hspace*{1em}$(1)$ Set counter $j=1$ and $\mu(2)=1$.

$(2)$ Given $j$, in the admissible tree $\al$,
\begin{align*}
&\text{if there exists $2k_{1}$ such that  $2k_{1}\stackrel{L}{\to} 2j$, then $\mu(2k_{1}):=\mu(2j)$;}\\
 &\text{if there exists $2k_{2}$ such that $2k_{2}\stackrel{M}{\to} 2j$, then $\mu(2k_{2}):=2j$;}\\
 &\text{if there exists $2k_{3}$ such that $2k_{3}\stackrel{R}{\to} 2j$, then $\mu(2k_{3}):=2j+1$.}
\end{align*}
Otherwise, go to step $(3)$.

$(3)$ Set $j=j+1$. If $j=k$, then stop, otherwise go to step $(2)$.
\end{algorithm}
Since it is an admissible tree $\al$, one can see that, if $j=l$, we have defined $\mu(2i)$ for $1 \leq i\leq l$ and $\mu(2i)<2i$ by the step $(2)$. Especially, when $j=k$, we generate a collapsing map $\mu$. Moreover, one has that $T(\mu)$ equals to tree $\al$.

\begin{example}
Suppose we are given the tree as follows.

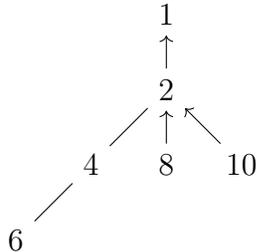
\begin{minipage}{0.3\textwidth}
\begin{tikzpicture}
\node (1) at (0,0) {1};
\node (2) at (0,-1) {2};
\node (4) at (-1,-2) {4};
\node (8) at (0,-2) {8};
\node (10) at (1,-2) {10};
\node (6) at (-2,-3) {6};
\draw[<-] (1)--(2);
\draw[-] (2)--(4);
\draw[<-] (2)--(8);
\draw[<-] (2)--(10);
\draw[-] (4)--(6);
\end{tikzpicture}
\end{minipage}
\begin{minipage}{0.65\textwidth}
At first, let $\mu(2)=1$. As there are left, middle and right children of node $2$, we define $\mu(4)=1$, $\mu(8)=2$, and $\mu(10)=3$. Next, turn to node $4$. There only exists the left child and hence we define $\mu(6)=\mu(4)=1$. Finally we arrive at
$$
\begin{tabular}{c|ccccc}
$2j$&2&4&6&8&10\\
\hline
$\mu(2j)$ &1&1&1&2&3
\end{tabular}
$$
\end{minipage}

\end{example}

Note that the upper echelon form $\mu_{m}$ is unique in every equivalent class. Given a skeleton tree, there also exists a unique upper echelon tree. We give an algorithm
to uniquely produce an upper echelon tree.

\begin{algorithm}[Generate an upper echelon tree]\label{algorithm:generate an upper echelon tree}
~\\
\hspace*{1em}$(1)$ Given a skeleton tree with $k+1$ nodes, label the top node with $1$ and set counter $j=1$.

$(2)$ If the node labeled $2j$ has a left child, then label that left child node with $2(j+1)$, set counter $j=j+1$ and go to step $(4)$. If not, go to step $(3)$.

$(3)$ In the already labeled nodes which has an unlabeled middle or right child, search for the node with the smallest label. If such a node has an unlabeled middle child, label the middle child with $2(j+1)$, set counter $j=j+1$, and go to step $(4)$. If such a node has no unlabeled middle child but an unlabeled right child, label the right child with $2(j+1)$, set counter $j=j+1$, and go to step
$(4)$. If none of the labeled nodes has an unlabeled middle or right child, then stop.

$(4)$ If $j=k$, then stop, otherwise go to step $(2)$.
\end{algorithm}

Next, we are able to show that acceptable moves preserve the tree structures but
permute the labeling under the admissibility requirement.

\begin{proposition}\label{prop:admissible tree, the same skeleton under km move}
Two collapsing maps $\mu$ and $\mu'$ are KM-relatable if and only if the trees corresponding to $\mu$ and $\mu'$ have the same skeleton. Moreover, if $\mu'=KM(\rho)(\mu)$, then $T(\mu')$ has the same skeleton to $T(\mu)$ with node $2j$ replaced by $\rho(2j)$.
\end{proposition}
\begin{proof}
Without loss, we might as well assume that $\rho=(2j_{0},2j_{0}+2)\circ (2j_{0}+1,2j_{0}+3)$ and $\mu'=KM(\rho)(\mu)$.
With node $2j$ in the tree $T(\mu)$ replaced by $\rho(2j)$, it generates a tree $\al''$ with the same skeleton as $T(\mu)$. Since $\rho\in P$ is an acceptable move with respect to $\mu$, we have $\mu(2j_{0}+2)<2j_{0}$ and $\mu(2j_{0})\neq \mu(2j_{0}+2)$, which implies that $\al''$ is also an admissible tree.
By Algorithm $\ref{algorithm:from admissible tree to collapsing map}$, it generates a collapsing map $\mu''$. Thus it suffices to prove $\mu'=\mu''$, or equivalently, $\mu''=KM(\rho)(\mu)$. Note that
 $$2l\stackrel{L/M/R}{\longrightarrow}2j\ \text{in the tree}\ T(\mu) \Longleftrightarrow \rho(2l)\stackrel{L/M/R}{\longrightarrow}\rho(2j)\ \text{in the tree}\ \al''.$$
By Algorithm $\ref{algorithm:from admissible tree to collapsing map}$, it implies that
\begin{align}
\begin{cases}
\mu(2l)=\mu(2j)&\Longleftrightarrow \mu''(\rho(2l))=\mu''(\rho(2j)),\\
\mu(2l)=2j&\Longleftrightarrow \mu''(\rho(2l))=\rho(2j),\\
\mu(2l)=2j+1&\Longleftrightarrow \mu''(\rho(2l))=\rho(2j)+1.
\end{cases}
\end{align}
With $\mu(2)=\mu''(2)=1$, by induction argument we obtain
$$\rho\circ \mu(2l)=\mu''(\rho(2l)),$$
that is, $\mu''=KM(\rho)(\mu)$.

Conversely, we suppose that $T(\mu)$ has the same skeleton as $T(\mu')$. By Algorithm $\ref{algorithm:generate an upper echelon tree}$ and Algorithm $\ref{algorithm:from admissible tree to collapsing map}$, it generates a unique collapsing map $\mu_{s}$ which is in upper echelon form for the skeleton of $T(\mu)$.
On the other hand, there exist an acceptable move $\sigma$ with respect to $\mu$ as well as  $\mu_{m}$, which is in an upper echelon form, such that $\mu_{m}=KM(\sigma)(\mu)$. Since $T(\mu_{m})$ also has the same skeleton as $T(\mu)$, it gives that $\mu_{m}=\mu_{s}$. In the same way, we also have
$\mu_{m}'=\mu_{s}$, which implies that $\mu$ and $\mu'$ are KM-relatable.
\end{proof}

Given $k$, we would like to have the number of different ternary tree structures of $k$ nodes, which equals to the number of equivalent classes.
This number is exactly defined as the generalized Catalan number (see \cite{hilton1991catalan}), that is,
\begin{align}\label{equ:catalan number}
\frac{1}{k}{3k \choose k-1}
\end{align}
which can be controlled by $8^{k}$ by Stirling's approximation to $k!$. Hence, we
just provide a proof of the quintic KM board game argument.

Now, let us get to the main part, namely, how to compute $D_{m}$ for a given upper echelon class.
We define a map $T_{D}$ which maps an admissible tree $T(\mu)$ to a time integration domain
\begin{equation} \label{equ:time integration domain generated by mu}
T_{D}(\mu)=\lr{t_{2j+1}\geq t_{2l+1}:2l\to 2j\ \text{in the tree}\ T(\mu)}\bigcap \lr{t_{1}\geq t_{3}}.
\end{equation}
where $2l\to 2j$ denotes that node $2l$ is a child of node $2j$.

\begin{proposition}\label{prop:upper echelon form, time integration domain}
Given a $\mu_{m}$ in upper echelon form, we have
\begin{align*}
\sum_{\mu\sim \mu_{m}}\int_{t_{1}\geq t_{3}\geq...\geq t_{2k+1}}
J_{\mu}^{(2k+1)}(\ga^{(2k+1)})(t_{1},\underline{t}_{2k+1})
d\underline{t}_{2k+1}
=\int_{T_{D}(\mu_{m})}J_{\mu_{m}}^{(2k+1)}(\ga^{(2k+1)})(t_{1},\underline{t}_{2k+1})
d\underline{t}_{2k+1}.
\end{align*}
and hence
\begin{align} \label{equ:integration,upper echelon form, time integration domain}
\ga^{(1)}(t_{1})=\sum_{\mu_{m}:\text{upper echelon form}}\int_{T_{D}(\mu_{m})}J_{\mu_{m}}^{(2k+1)}(\ga^{(2k+1)})(t_{1},\underline{t}_{2k+1})
d\underline{t}_{2k+1}.
\end{align}
\end{proposition}
\begin{proof}
Let $\Sigma(\mu_{m})$ be the set of all acceptable moves with respect to $\mu_{m}$. Then by the equality $(\ref{equ:the equality under the km move})$, we have
\begin{align*}
\sum_{\mu\sim \mu_{m}}I(\mu,id,\ga^{(2k+1)})=&\sum_{\rho \in \Sigma(\mu_{m})}I(\mu_{m},\rho^{-1},\ga^{(2k+1)}).
\end{align*}

By Proposition $\ref{prop:admissible tree, the same skeleton under km move}$, we see that
\begin{align}
\Sigma(\mu_{m})=\lr{\rho\in P:\rho(2j)<\rho(2l),\ \text{if $2l\to 2j$ in the tree $T(\mu_{m})$}}
\end{align}
and hence
\begin{align}
\bigcup_{\rho\in \Sigma(\mu_{m})}\lr{t_{1}\geq t_{\rho^{-1}(3)}\geq \ccc \geq t_{\rho^{-1}(2k+1)}}=T_{D}(\mu_{m}).
\end{align}
\end{proof}
\begin{example}Let us demonstrate Proposition $\ref{prop:upper echelon form, time integration domain}$ by an example.
Recall the upper echelon tree $T(\mu)$ in Example $\ref{example:admissible tree}$.
$$
\begin{tabular}{c|ccccc}
$2j$&2&4&6&8&10\\
\hline
$\mu_{1}(2j)$ &1&1&1&2&3
\end{tabular}
$$

There are $12$ acceptable moves with respect to $\mu_{1}$ such that
\begin{align*}
u_{i}=KM(\rho_{i})(\mu_{1}).
\end{align*}

\begin{table}[htbp]
\centering
\caption{Acceptable moves and Time integration domain}
\label{table:Acceptable moves and Time integration domain}
\begin{tabular}{c|ccccc||c|ccccc||c}
$2j$&2&4&6&8&10& $2j$&2&4&6&8&10&\text{Time integration domain} \\
\hline
$\rho_{1}(2j)$ &2&$4$&$6$&$8$&10& $\rho_{1}^{-1}(2j)$&2&$4$&$6$&$8$&10&$\lr{t_{1}\geq t_{3}\geq t_{5}\geq t_{7}\geq t_{9}\geq t_{11}}$\\
$\rho_{2}(2j)$ &2&$4$&$6$&$10$&8& $\rho_{2}^{-1}(2j)$&2&$4$&$6$&$10$&8&$\lr{t_{1}\geq t_{3}\geq t_{5}\geq t_{7}\geq t_{11}\geq t_{9}}$\\
$\rho_{3}(2j)$ &2&$4$&$8$&$6$&10& $\rho_{3}^{-1}(2j)$&2&$4$&$8$&$6$&10&$\lr{t_{1}\geq t_{3}\geq t_{5}\geq t_{9}\geq t_{7}\geq t_{11}}$\\
$\rho_{4}(2j)$ &2&$4$&$8$&$10$&6& $\rho_{4}^{-1}(2j)$&2&$4$&$10$&$6$&8&$\lr{t_{1}\geq t_{3}\geq t_{5}\geq t_{11}\geq t_{7}\geq t_{9}}$\\
$\rho_{5}(2j)$ &2&$4$&$10$&$6$&8& $\rho_{5}^{-1}(2j)$&2&$4$&$8$&$10$&6&$\lr{t_{1}\geq t_{3}\geq t_{5}\geq t_{9}\geq t_{11}\geq t_{7}}$\\
$\rho_{6}(2j)$ &2&$4$&$10$&$8$&6& $\rho_{6}^{-1}(2j)$&2&$4$&$10$&$8$&6&$\lr{t_{1}\geq t_{3}\geq t_{5}\geq t_{11}\geq t_{9}\geq t_{7}}$\\
$\rho_{7}(2j)$ &2&$6$&$8$&$4$&10& $\rho_{7}^{-1}(2j)$&2&$8$&$4$&$6$&10&$\lr{t_{1}\geq t_{3}\geq t_{9}\geq t_{5}\geq t_{7}\geq t_{11}}$\\
$\rho_{8}(2j)$ &2&$6$&$8$&$10$&4& $\rho_{8}^{-1}(2j)$&2&$10$&$4$&$6$&8&$\lr{t_{1}\geq t_{3}\geq t_{11}\geq t_{5}\geq t_{7}\geq t_{9}}$\\
$\rho_{9}(2j)$ &2&$6$&$10$&$4$&8& $\rho_{9}^{-1}(2j)$&2&$8$&$4$&$10$&6&$\lr{t_{1}\geq t_{3}\geq t_{9}\geq t_{5}\geq t_{11}\geq t_{7}}$\\
$\rho_{10}(2j)$ &2&$6$&$10$&$8$&4& $\rho_{10}^{-1}(2j)$&2&$10$&$4$&$8$&6&$\lr{t_{1}\geq t_{3}\geq t_{11}\geq t_{5}\geq t_{9}\geq t_{7}}$\\
$\rho_{11}(2j)$ &2&$8$&$10$&$4$&6& $\rho_{11}^{-1}(2j)$&2&$8$&$10$&$4$&6&$\lr{t_{1}\geq t_{3}\geq t_{9}\geq t_{11}\geq t_{5}\geq t_{7}}$\\
$\rho_{12}(2j)$ &2&$8$&$10$&$6$&4& $\rho_{12}^{-1}(2j)$&2&$10$&$8$&$4$&6&$\lr{t_{1}\geq t_{3}\geq t_{11}\geq t_{9}\geq t_{5}\geq t_{7}}$
\end{tabular}

\end{table}

Here are all the admissible trees equivalent to $T(\mu_{1})$.

\begin{minipage}{0.3\textwidth}
\centering
$$
\begin{tabular}{c|ccccc}
$2j$&2&4&6&8&10\\
\hline
$\mu_{1}(2j)$ &1&1&1&2&3
\end{tabular}
$$

\begin{tikzpicture}
\node (1) at (0,0) {1};
\node (2) at (0,-1) {2};
\node (4) at (-1,-2) {4};
\node (8) at (0,-2) {8};
\node (10) at (1,-2) {10};
\node (6) at (-2,-3) {6};
\draw[<-] (1)--(2);
\draw[-] (2)--(4);
\draw[<-] (2)--(8);
\draw[<-] (2)--(10);
\draw[-] (4)--(6);
\end{tikzpicture}

\end{minipage}
\begin{minipage}{0.3\textwidth}
\centering

$$
\begin{tabular}{c|ccccc}
$2j$&2&4&6&8&10\\
\hline
$\mu_{2}(2j)$ &1&1&1&3&2
\end{tabular}
$$

\begin{tikzpicture}
\node (1) at (0,0) {1};
\node (2) at (0,-1) {2};
\node (4) at (-1,-2) {4};
\node (10) at (0,-2) {10};
\node (8) at (1,-2) {8};
\node (6) at (-2,-3) {6};
\draw[<-] (1)--(2);
\draw[-] (2)--(4);
\draw[<-] (2)--(10);
\draw[<-] (2)--(8);
\draw[-] (4)--(6);
\end{tikzpicture}

\end{minipage}
\begin{minipage}{0.3\textwidth}
\centering

$$
\begin{tabular}{c|ccccc}
$2j$&2&4&6&8&10\\
\hline
$\mu_{3}(2j)$ &1&1&3&1&2
\end{tabular}
$$

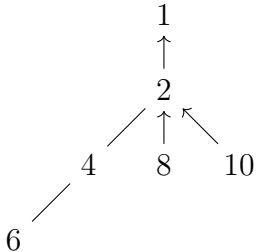
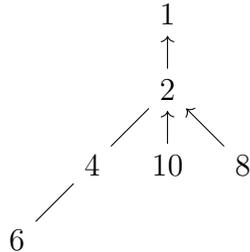
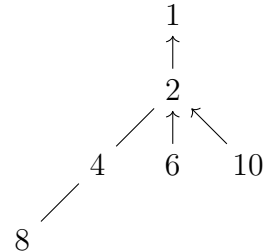
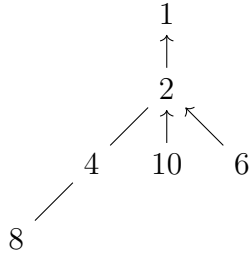
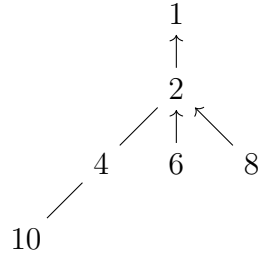
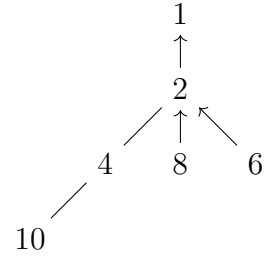
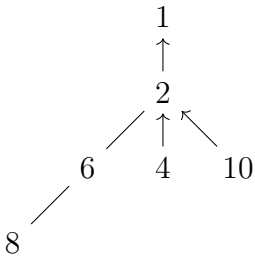
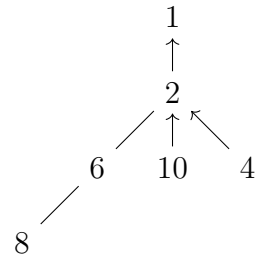
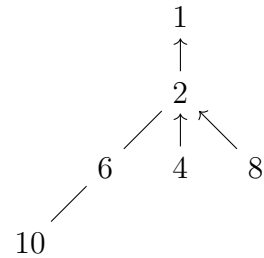
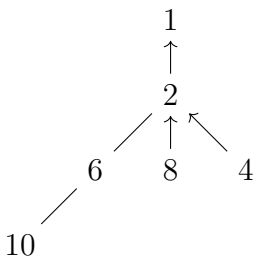
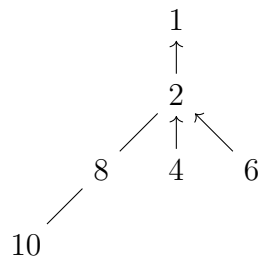
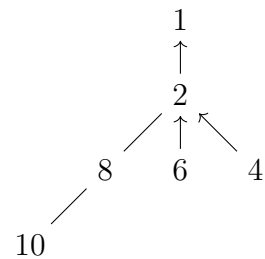
\begin{tikzpicture}
\node (1) at (0,0) {1};
\node (2) at (0,-1) {2};
\node (4) at (-1,-2) {4};
\node (6) at (0,-2) {6};
\node (10) at (1,-2) {10};
\node (8) at (-2,-3) {8};
\draw[<-] (1)--(2);
\draw[-] (2)--(4);
\draw[<-] (2)--(6);
\draw[<-] (2)--(10);
\draw[-] (4)--(8);
\end{tikzpicture}

\end{minipage}

\begin{minipage}{0.3\textwidth}
\centering
$$
\begin{tabular}{c|ccccc}
$2j$&2&4&6&8&10\\
\hline
$\mu_{4}(2j)$ &1&1&3&1&2
\end{tabular}
$$

\begin{tikzpicture}
\node (1) at (0,0) {1};
\node (2) at (0,-1) {2};
\node (4) at (-1,-2) {4};
\node (8) at (0,-2) {10};
\node (10) at (1,-2) {6};
\node (6) at (-2,-3) {8};
\draw[<-] (1)--(2);
\draw[-] (2)--(4);
\draw[<-] (2)--(8);
\draw[<-] (2)--(10);
\draw[-] (4)--(6);
\end{tikzpicture}

\end{minipage}
\begin{minipage}{0.3\textwidth}
\centering
$$
\begin{tabular}{c|ccccc}
$2j$&2&4&6&8&10\\
\hline
$\mu_{5}(2j)$ &1&1&2&3&1
\end{tabular}
$$

\begin{tikzpicture}
\node (1) at (0,0) {1};
\node (2) at (0,-1) {2};
\node (4) at (-1,-2) {4};
\node (8) at (0,-2) {6};
\node (10) at (1,-2) {8};
\node (6) at (-2,-3) {10};
\draw[<-] (1)--(2);
\draw[-] (2)--(4);
\draw[<-] (2)--(8);
\draw[<-] (2)--(10);
\draw[-] (4)--(6);
\end{tikzpicture}

\end{minipage}
\begin{minipage}{0.3\textwidth}
\centering
$$
\begin{tabular}{c|ccccc}
$2j$&2&4&6&8&10\\
\hline
$\mu_{6}(2j)$ &1&1&3&2&1
\end{tabular}
$$

\begin{tikzpicture}
\node (1) at (0,0) {1};
\node (2) at (0,-1) {2};
\node (4) at (-1,-2) {4};
\node (8) at (0,-2) {8};
\node (10) at (1,-2) {6};
\node (6) at (-2,-3) {10};
\draw[<-] (1)--(2);
\draw[-] (2)--(4);
\draw[<-] (2)--(8);
\draw[<-] (2)--(10);
\draw[-] (4)--(6);
\end{tikzpicture}

\end{minipage}

\begin{minipage}{0.3\textwidth}
\centering
$$
\begin{tabular}{c|ccccc}
$2j$&2&4&6&8&10\\
\hline
$\mu_{7}(2j)$ &1&2&1&1&3
\end{tabular}
$$

\begin{tikzpicture}
\node (1) at (0,0) {1};
\node (2) at (0,-1) {2};
\node (4) at (-1,-2) {6};
\node (8) at (0,-2) {4};
\node (10) at (1,-2) {10};
\node (6) at (-2,-3) {8};
\draw[<-] (1)--(2);
\draw[-] (2)--(4);
\draw[<-] (2)--(8);
\draw[<-] (2)--(10);
\draw[-] (4)--(6);
\end{tikzpicture}
\end{minipage}
\begin{minipage}{0.3\textwidth}
\centering
$$
\begin{tabular}{c|ccccc}
$2j$&2&4&6&8&10\\
\hline
$\mu_{8}(2j)$ &1&3&1&1&2
\end{tabular}
$$

\begin{tikzpicture}
\node (1) at (0,0) {1};
\node (2) at (0,-1) {2};
\node (4) at (-1,-2) {6};
\node (8) at (0,-2) {10};
\node (10) at (1,-2) {4};
\node (6) at (-2,-3) {8};
\draw[<-] (1)--(2);
\draw[-] (2)--(4);
\draw[<-] (2)--(8);
\draw[<-] (2)--(10);
\draw[-] (4)--(6);
\end{tikzpicture}
\end{minipage}
\begin{minipage}{0.3\textwidth}
\centering
$$
\begin{tabular}{c|ccccc}
$2j$&2&4&6&8&10\\
\hline
$\mu_{9}(2j)$ &1&2&1&3&1
\end{tabular}
$$

\begin{tikzpicture}
\node (1) at (0,0) {1};
\node (2) at (0,-1) {2};
\node (4) at (-1,-2) {6};
\node (8) at (0,-2) {4};
\node (10) at (1,-2) {8};
\node (6) at (-2,-3) {10};
\draw[<-] (1)--(2);
\draw[-] (2)--(4);
\draw[<-] (2)--(8);
\draw[<-] (2)--(10);
\draw[-] (4)--(6);
\end{tikzpicture}
\end{minipage}

\begin{minipage}{0.3\textwidth}
\centering
$$
\begin{tabular}{c|ccccc}
$2j$&2&4&6&8&10\\
\hline
$\mu_{10}(2j)$ &1&3&1&2&1
\end{tabular}
$$

\begin{tikzpicture}
\node (1) at (0,0) {1};
\node (2) at (0,-1) {2};
\node (4) at (-1,-2) {6};
\node (8) at (0,-2) {8};
\node (10) at (1,-2) {4};
\node (6) at (-2,-3) {10};
\draw[<-] (1)--(2);
\draw[-] (2)--(4);
\draw[<-] (2)--(8);
\draw[<-] (2)--(10);
\draw[-] (4)--(6);
\end{tikzpicture}
\end{minipage}
\begin{minipage}{0.3\textwidth}
\centering
$$
\begin{tabular}{c|ccccc}
$2j$&2&4&6&8&10\\
\hline
$\mu_{11}(2j)$ &1&2&3&1&1
\end{tabular}
$$

\begin{tikzpicture}
\node (1) at (0,0) {1};
\node (2) at (0,-1) {2};
\node (4) at (-1,-2) {8};
\node (8) at (0,-2) {4};
\node (10) at (1,-2) {6};
\node (6) at (-2,-3) {10};
\draw[<-] (1)--(2);
\draw[-] (2)--(4);
\draw[<-] (2)--(8);
\draw[<-] (2)--(10);
\draw[-] (4)--(6);
\end{tikzpicture}
\end{minipage}
\begin{minipage}{0.3\textwidth}
\centering
$$
\begin{tabular}{c|ccccc}
$2j$&2&4&6&8&10\\
\hline
$\mu_{12}(2j)$ &1&3&2&1&1
\end{tabular}
$$

\begin{tikzpicture}
\node (1) at (0,0) {1};
\node (2) at (0,-1) {2};
\node (4) at (-1,-2) {8};
\node (8) at (0,-2) {6};
\node (10) at (1,-2) {4};
\node (6) at (-2,-3) {10};
\draw[<-] (1)--(2);
\draw[-] (2)--(4);
\draw[<-] (2)--(8);
\draw[<-] (2)--(10);
\draw[-] (4)--(6);
\end{tikzpicture}

\end{minipage}	

From the above time integration domains in Table $\ref{table:Acceptable moves and Time integration domain}$, we have
\begin{align*}
&\bigcup_{\rho_{i}\in \Sigma(\mu_{1})}\lr{t_{1}\geq t_{\rho_{i}^{-1}(3)}\geq t_{\rho_{i}^{-1}(5)}\geq t_{\rho_{i}^{-1}(7)}\geq t_{\rho_{i}^{-1}(9)}\geq t_{\rho_{i}^{-1}(11)}}\\
=&\lr{t_{1}\geq t_{3},t_{3}\geq t_{5}\geq t_{7},t_{3}\geq t_{9},t_{3}\geq t_{11}}.
\end{align*}
By the definition,
\begin{align*}
T_{D}(\mu_{1})=\lr{t_{1}\geq t_{3},t_{3}\geq t_{5}\geq t_{7},t_{3}\geq t_{9},t_{3}\geq t_{11}}.
\end{align*}
Hence,
\begin{align*}
\bigcup_{\rho_{i}\in \Sigma(\mu_{1})}\lr{t_{1}\geq t_{\rho_{i}^{-1}(3)}\geq t_{\rho_{i}^{-1}(5)}\geq t_{\rho_{i}^{-1}(7)}\geq t_{\rho_{i}^{-1}(9)}\geq t_{\rho_{i}^{-1}(11)}}
=T_{D}(\mu_{1}).
\end{align*}
\end{example}

\subsection{Signed KM Acceptable Moves}\label{subsection:Signed KM Acceptable Moves}
In Section \ref{subsection:Admissible Tree}, we provide a method to compute $T_{D}(\mu)$. A nontrivial application is that, the original KM board game argument is not compatible with space-time multilinear estimates. Indeed, depending on the sign combination in the Duhamel expansion $J_{\mu_{m}}^{(2k+1)}(\ga^{(2k+1)})(t_{1},\underline{t}_{2k+1})$, one could run into the problem
that one needs to estimate the $x$ part and the $x'$ part using the same time integral. (See \cite[Example 4]{chen2020unconditional}.) To be compatible with the estimate part in Section \ref{section:Estimates for the Compatible Time Integration Domain},
we restart with the signed collapsing pair $(\mu,sgn)$.

First, we rewrite
\begin{equation}
\ga^{(1)}=\sum_{(\mu,sgn)}I(\mu,id,sgn,\ga^{(2k+1)}),
\end{equation}
where
\begin{equation}
I(\mu,\sigma,sgn,\ga^{(2k+1)})=\int_{t_{1}\geq t_{\sigma(3)}\geq \ccc \geq t_{\sigma(2k+1)}}
J_{\mu,sgn}^{(2k+1)}(\ga^{(2k+1)})(t_{1},\underline{t}_{2k+1})d\underline{t}_{2k+1}
\end{equation}
and
\begin{align}
J_{\mu,sgn}^{(2k+1)}(\ga^{(2k+1)})(t_{1},\underline{t}_{2k+1})=
&U^{(1)}(t_{1}-t_{3})B_{\mu(2);2,3}^{sgn(2)}U^{(3)}(t_{3}-t_{5})
B_{\mu(4);4,5}^{sgn(4)}\\
&\ccc U^{(2k-1)}(t_{2k-1}-t_{2k+1})B_{\mu(2k);2k,2k+1}^{sgn(2k)}\ga^{(2k+1)}(t_{2k+1})\notag
\end{align}
where the notations have been introduced in $(\ref{equ:duhamel expansion})$.

For convenience, we define
$$ sgn(2l+1):=sgn(2l)$$
for $l\in\lr{1,2,...,k}$.
\begin{definition}
Let $\rho$ be an acceptable move of $\mu$. We define a signed version of the KM acceptable move in Chen-Pavlovi\'{c} format, still denoted by $KM(\rho)$, as follows:
\begin{align*}
(\mu',\sigma',sgn')=KM(\rho)(\mu,\sigma,sgn)
\end{align*}
where
\begin{align*}
&\mu'=\rho \circ \mu \circ \rho^{-1},\\
&\sigma'=\rho \circ \sigma,\\
&sgn'=sgn\circ \rho^{-1}.
\end{align*}
\end{definition}
If $(\mu,\sigma,sgn)$ and $(\mu',\sigma',sgn')$ are such that there exists $\rho$ as above for which
$$(\mu',\sigma',sgn')=KM(\rho)(\mu,\sigma,sgn)$$
then we say that $(\mu,\sigma,sgn)$ and $(\mu',\sigma',sgn')$ are KM-relatable.
With a slight modification of the argument in \cite{chen2011the,klainerman2008on}, we also have
\begin{equation}
I(\mu',\sigma',sgn',\ga^{(2k+1)})=I(\mu,\sigma,sgn,\ga^{(2k+1)}).
\end{equation}

%Note that the Algorithm $\ref{algorithm:generate an admissible tree}$ is independent of the sign of a collapsing map $\mu$, so we can denote the collapsing map pair $(\mu,sgn)$ by the signed admissible tree generated by $(\mu,sgn)$.
\begin{example}\label{example:signed admissible tree}
We consider the following pair $(\mu,sgn)$
$$
\begin{tabular}{c|ccccc}
$2j$&2&4&6&8&10\\
\hline
$\mu(2j)$ &1&1&1&2&3\\
$sgn(2j)$ &$-$&$+$&$-$&$-$&+
\end{tabular}
$$
By Algorithm $\ref{algorithm:generate an admissible tree}$, with adding the sign, it generates a signed admissible tree as follows
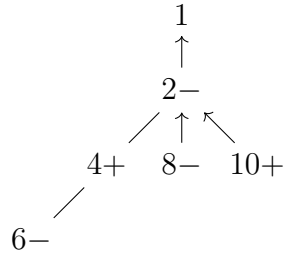
\begin{center}
\begin{tikzpicture}
\node (1) at (0,0) {1};
\node (2) at (0,-1) {2$-$};
\node (4) at (-1,-2) {4+};
\node (8) at (0,-2) {8$-$};
\node (10) at (1,-2) {10+};
\node (6) at (-2,-3) {6$-$};
\draw[<-] (1)--(2);
\draw[-] (2)--(4);
\draw[<-] (2)--(8);
\draw[<-] (2)--(10);
\draw[-] (4)--(6);
\end{tikzpicture}
\end{center}
\end{example}

\begin{definition}
For a skeleton tree, we call it the signed skeleton tree if we add the sign. For example, the signed skeleton of the tree in Example $\ref{example:signed admissible tree}$ is shown as follows.
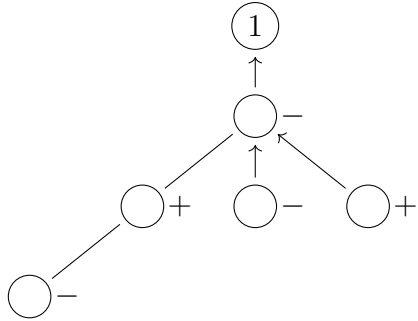
\begin{center}
\begin{tikzpicture}
\node[draw,circle,inner sep=0.1cm,outer sep=0.1cm] (1) at (0,0) {1};
\node[draw,circle,inner sep=0.2cm,outer sep=0.1cm] (2) at (0,-1.2) {};
\node (02) at(0.5,-1.2){$-$};
\node[draw,circle,inner sep=0.2cm,outer sep=0.1cm] (4) at (-1.5,-2.4) {};
\node (04) at (-1,-2.4) {+};
\node[draw,circle,inner sep=0.2cm,outer sep=0.1cm] (8) at (0,-2.4) {};
\node (08) at (0.5,-2.4) {$-$};
\node[draw,circle,inner sep=0.2cm,outer sep=0.1cm] (10) at (1.5,-2.4) {};
\node (010) at (2,-2.4) {+};
\node[draw,circle,inner sep=0.2cm,outer sep=0.1cm] (6) at (-3,-3.6) {};
\node (06) at (-2.5,-3.6) {$-$};
\draw[<-] (1)--(2);
\draw[-] (2)--(4);
\draw[<-] (2)--(8);
\draw[<-] (2)--(10);
\draw[-] (4)--(6);
\end{tikzpicture}
\end{center}
\end{definition}
Similarly, the signed acceptable moves also preserve the signed tree structures.
\begin{proposition}
Two collapsing map pairs are KM-relatable if and only if they have the same signed skeleton tree.
\end{proposition}
\begin{proof}
 By Proposition $\ref{prop:admissible tree, the same skeleton under km move}$, it suffices to prove that $\rho$ keeps the sign invariant. Indeed, we note that node $2j$ in the tree $T(\mu)$ is corresponding to node $\rho(2j)$ in the tree $T(\mu')$ and hence $sgn'(\rho(2j))=sgn(2j)$.
\end{proof}
\subsection{Tamed Form}\label{subsection:Tamed Form}
 We will prove that there exists a unique special form, which we call the tamed form, in every equivalent class. First, through an example, we present an algorithm for producing the tamed enumeration of a signed skeleton. Then we exhibit
how to reduce a signed tree with same skeleton but different enumeration into the tamed
form using signed KM acceptable moves. In the end, we arrive at
\begin{equation}
\ga^{(1)}(t_{1})=\sum_{(\mu_{*},sgn_{*})\ \text{tamed}}\int_{T_{D}(\mu_{*})}J_{\mu_{*},sgn_{*}}^{(2k+1)}(\ga^{(2k+1)})(t_{1},\underline{t}_{2k+1})
d\underline{t}_{2k+1}.
\end{equation}
which is an adaptation of representation $(\ref{equ:integration,upper echelon form, time integration domain})$.
\begin{definition}
We call $2j\geq 4$ is tier of $q$ if
$$\mu^{q}(2j)=1\quad \text{but}\quad \mu^{q-1}(2j)>1$$
 where $\mu^{q}=\mu\circ \ccc\circ u$, the composition taken $q$ times. We write $t(2j)$ for the tier value of $2j$\footnote{The tier value of $2j$ equals to the number of the arrows from node $2j$ to node $1$. That is why we use arrows to link the middle/right child.}.
\end{definition}
\begin{definition}
A pair $(\mu,sgn)$ is tamed if it meets the following four requirements:

$(1)$ If $t(2l)<t(2r)$, then $2l<2r$.

$(2)$ If $t(2l)=t(2r)$, $\mu^{2}(2l)=\mu^{2}(2r)$, $sgn(\mu(2l))=sgn(\mu(2r))$ and $\mu(2l)<\mu(2r)$,
then $2l<2r$.

$(3)$ If $t(2l)=t(2r)$, $\mu^{2}(2l)=\mu^{2}(2r)$, $sgn(\mu(2l))=+$, $sgn(\mu(2r))=-$, then $2l<2r$.

$(4)$ If $t(2l)=t(2r)$, $\mu^{2}(2l)\neq \mu^{2}(2r)$, $\mu(2l)<\mu(2r)$, then $2l<2r$.
\end{definition}
%Note that the statement $\mu^{2}(2l)=\mu^{2}(2r)$ means graphically that the parents of $2l$ and $2r$
%belong to the same left branch.
Conditions (2), (3), and (4) specify the ordering for $2l$ and $2r$
belonging to the same tier. More precisely, rule (2) says that the ordering of middle child is prior to the one of right child and the ordering follows the parental ordering if two different parents belong to the same left branch with the same sign. Rule (3) says that if the parents belong to the same left branch, a positive parent dominates
over a negative parent. Finally, if the parents do not belong to the same left
branch, rule (4) says that the ordering follows the parental ordering regardless of the signs
of the parents.

\begin{example}\label{example:from skeleton to tamed tree}
The pair $(\mu_{*},sgn_{*})$ in the following chart
$$
\begin{tabular}{c|ccccccccccccc}
$2j$&2&4&6&8&10&12&14&16&18&20&22&24&26\\
\hline
$\mu_{*}(2j)$ &1&1&1&1&1&6&6&7&2&3&10&13&18\\
$sgn_{*}(2j)$&$-$&$-$&+&+&$-$&$-$&+&+&$-$&+&+&$-$&+\\
$t(2j)$ &1&1&1&1&1&2&2&2&2&2&2&3&3
\end{tabular}
$$
is tamed. We illustrate an algorithm for determining the unique tamed enumberation of a signed skeleton tree.

\begin{minipage}[c]{0.47\textwidth}
\hspace*{1em}We start with the skeleton (on the right) of the tree generated by $(\mu_{*},sgn_{*})$ with only the signs indicated.
\end{minipage}
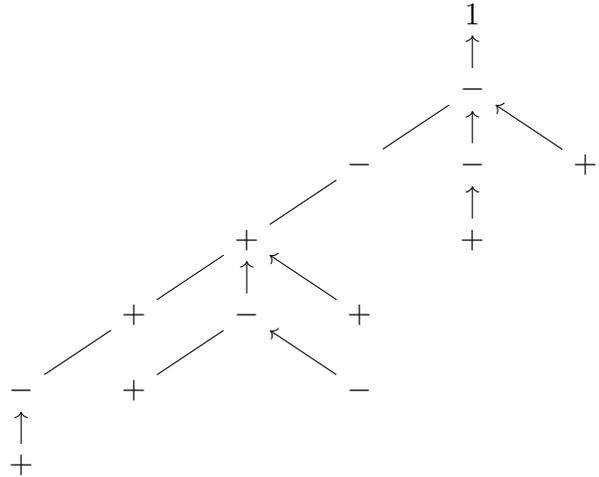
\begin{minipage}[c]{0.5\textwidth}
\begin{tikzpicture}
\node (1) at (0,0) {1};
\node (2) at (0,-1) {$-$};
\node (4) at (-1.5,-2) {$-$};
\node (6) at (-3,-3) {+};
\node (8) at (-4.5,-4) {+};
\node (10) at (-6,-5) {$-$};
\node (12) at (-3,-4) {$-$};
\node (14) at (-4.5,-5) {+};
\node (16) at (-1.5,-4) {+};
\node (18) at (0,-2) {$-$};
\node (20) at (1.5,-2) {+};
\node (22) at (-6,-6) {+};
\node (24) at (-1.5,-5) {$-$};
\node (26) at (0,-3) {+};
\draw[<-] (1)--(2);
\draw[-] (2)--(4);
\draw[<-] (2)--(18);
\draw[<-] (2)--(20);
\draw[-] (4)--(6);
\draw[-] (6)--(8);
\draw[<-] (6)--(12);
\draw[<-] (6)--(16);
\draw[-] (8)--(10);
\draw[<-] (10)--(22);
\draw[-] (12)--(14);
\draw[<-] (12)--(24);
\draw[<-] (18)--(26);
\end{tikzpicture}
\end{minipage}

\begin{minipage}[c]{0.47\textwidth}
\hspace*{1em}We consider the left branch attached to node $1$ where there are five nodes.
Then we label the left branch in order with $2$, $4$, $6$, $8$, $10$.
\end{minipage}
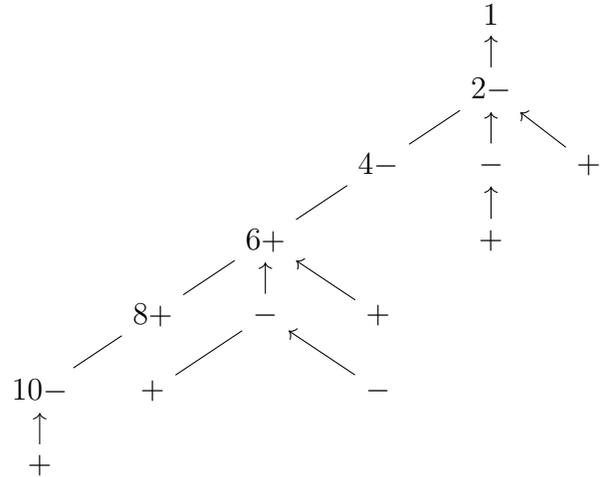
\begin{minipage}[c]{0.5\textwidth}
\begin{tikzpicture}
\node (1) at (0,0) {1};
\node (2) at (0,-1) {2$-$};
\node (4) at (-1.5,-2) {4$-$};
\node (6) at (-3,-3) {6+};
\node (8) at (-4.5,-4) {8+};
\node (10) at (-6,-5) {10$-$};
\node (12) at (-3,-4) {$-$};
\node (14) at (-4.5,-5) {+};
\node (16) at (-1.5,-4) {+};
\node (18) at (0,-2) {$-$};
\node (20) at (1.3,-2) {+};
\node (22) at (-6,-6) {+};
\node (24) at (-1.5,-5) {$-$};
\node (26) at (0,-3) {+};
\draw[<-] (1)--(2);
\draw[-] (2)--(4);
\draw[<-] (2)--(18);
\draw[<-] (2)--(20);
\draw[-] (4)--(6);
\draw[-] (6)--(8);
\draw[<-] (6)--(12);
\draw[<-] (6)--(16);
\draw[-] (8)--(10);
\draw[<-] (10)--(22);
\draw[-] (12)--(14);
\draw[<-] (12)--(24);
\draw[<-] (18)--(26);
\end{tikzpicture}
\end{minipage}

\begin{minipage}{0.48\textwidth}
\hspace*{1em}Now, we set a queue where we list the nodes $+$ first and then the $-$ nodes
$$Queue:6+,8+,2-,4-,10-$$
Then we work along the queue from left to right. Since $6+$ has both a middle and right child, we first label the middle child and its left branch with the next available number $12$ and $14$ and add these numbers to the queue putting the $+$ nodes before the $-$ nodes
$$Queue:6+,8+,2-,4-,10-,14+,12-$$
Next we label the right child with number $16$ and add it to the queue
$$Queue:6+,8+,2-,4-,10-,14+,12-,16+$$
\end{minipage}
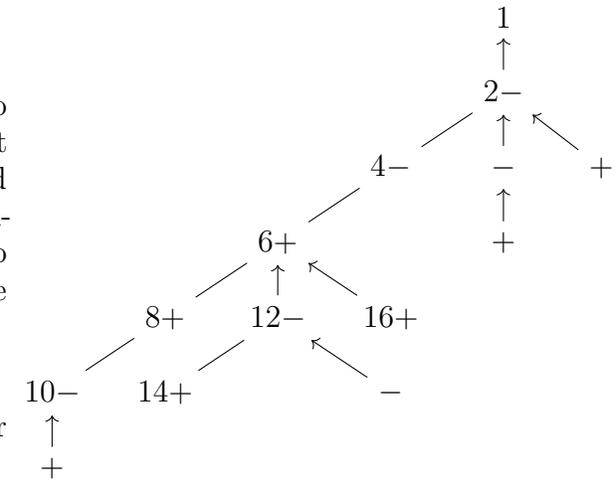
\begin{minipage}{0.49\textwidth}
\begin{tikzpicture}
\node (1) at (0,0) {1};
\node (2) at (0,-1) {2$-$};
\node (4) at (-1.5,-2) {4$-$};
\node (6) at (-3,-3) {6+};
\node (8) at (-4.5,-4) {8+};
\node (10) at (-6,-5) {10$-$};
\node (12) at (-3,-4) {12$-$};
\node (14) at (-4.5,-5) {14+};
\node (16) at (-1.5,-4) {16+};
\node (18) at (0,-2) {$-$};
\node (20) at (1.3,-2) {+};
\node (22) at (-6,-6) {+};
\node (24) at (-1.5,-5) {$-$};
\node (26) at (0,-3) {+};
\draw[<-] (1)--(2);
\draw[-] (2)--(4);
\draw[<-] (2)--(18);
\draw[<-] (2)--(20);
\draw[-] (4)--(6);
\draw[-] (6)--(8);
\draw[<-] (6)--(12);
\draw[<-] (6)--(16);
\draw[-] (8)--(10);
\draw[<-] (10)--(22);
\draw[-] (12)--(14);
\draw[<-] (12)--(24);
\draw[<-] (18)--(26);
\end{tikzpicture}
\end{minipage}
\\[1cm]

\begin{minipage}[c]{0.47\textwidth}
\hspace*{1em}Since we have already dealt with $6+$, we can pop it from the queue
$$Queue:8+,2-,4-,10-,14+,12-,16+$$
Subquently, we come to the next node in the queue which is $8+$. Since the node $8+$ has no child, we skip and pop it from the queue
$$Queue:2-,4-,10-,14+,12-,16+$$
Then, we come to the node $2-$, which has both a middle and right child. We first label the middle child with $18$ and then the right child with $20$.
\end{minipage}
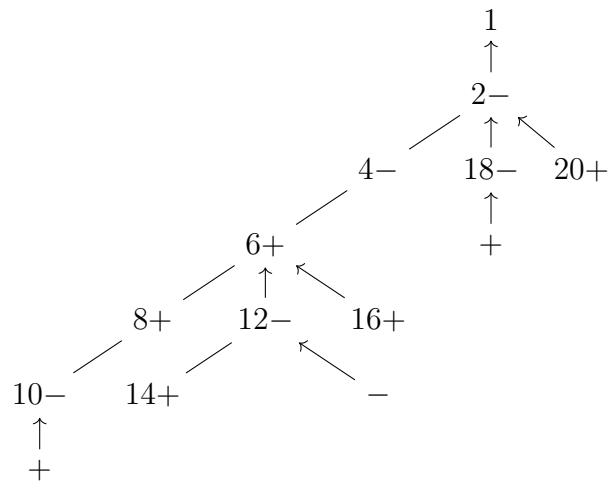
\begin{minipage}[c]{0.5\textwidth}
\begin{tikzpicture}
\node (1) at (0,0) {1};
\node (2) at (0,-1) {2$-$};
\node (4) at (-1.5,-2) {4$-$};
\node (6) at (-3,-3) {6+};
\node (8) at (-4.5,-4) {8+};
\node (10) at (-6,-5) {10$-$};
\node (12) at (-3,-4) {12$-$};
\node (14) at (-4.5,-5) {14+};
\node (16) at (-1.5,-4) {16+};
\node (18) at (0,-2) {18$-$};
\node (20) at (1.2,-2) {20+};
\node (22) at (-6,-6) {+};
\node (24) at (-1.5,-5) {$-$};
\node (26) at (0,-3) {+};
\draw[<-] (1)--(2);
\draw[-] (2)--(4);
\draw[<-] (2)--(18);
\draw[<-] (2)--(20);
\draw[-] (4)--(6);
\draw[-] (6)--(8);
\draw[<-] (6)--(12);
\draw[<-] (6)--(16);
\draw[-] (8)--(10);
\draw[<-] (10)--(22);
\draw[-] (12)--(14);
\draw[<-] (12)--(24);
\draw[<-] (18)--(26);
\end{tikzpicture}
\end{minipage}

\begin{minipage}{0.47\textwidth}
\hspace*{1em}From the queue, we pop $2-$ and add $18-$, $20+$:
$$Queue:4-,10-,14+,12-,16+,18-,20+$$
Since $4-$ has no child, we pop it and proceed to $10-$, which has a middle child. We label it with $22$. The queue is updated:
$$Queue:14+,12-,16+,18-,20+,22+$$
By turn, we arrive at the fully enumberated tree.
\end{minipage}
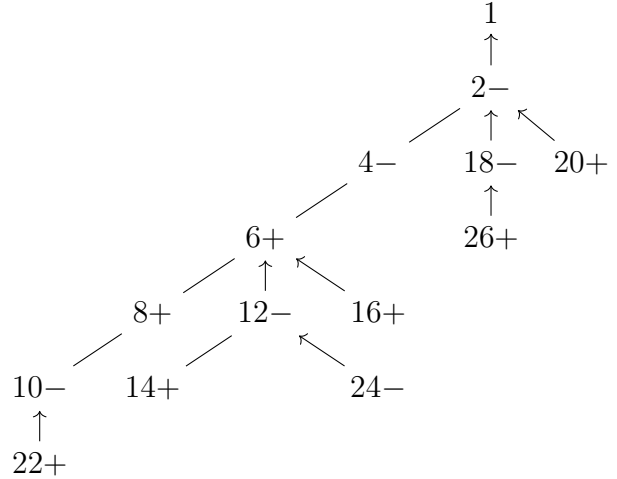
\begin{minipage}{0.5\textwidth}
\begin{tikzpicture}
\node (1) at (0,0) {1};
\node (2) at (0,-1) {2$-$};
\node (4) at (-1.5,-2) {4$-$};
\node (6) at (-3,-3) {6+};
\node (8) at (-4.5,-4) {8+};
\node (10) at (-6,-5) {10$-$};
\node (12) at (-3,-4) {12$-$};
\node (14) at (-4.5,-5) {14+};
\node (16) at (-1.5,-4) {16+};
\node (18) at (0,-2) {18$-$};
\node (20) at (1.2,-2) {20+};
\node (22) at (-6,-6) {22+};
\node (24) at (-1.5,-5) {24$-$};
\node (26) at (0,-3) {26+};
\draw[<-] (1)--(2);
\draw[-] (2)--(4);
\draw[<-] (2)--(18);
\draw[<-] (2)--(20);
\draw[-] (4)--(6);
\draw[-] (6)--(8);
\draw[<-] (6)--(12);
\draw[<-] (6)--(16);
\draw[-] (8)--(10);
\draw[<-] (10)--(22);
\draw[-] (12)--(14);
\draw[<-] (12)--(24);
\draw[<-] (18)--(26);
\end{tikzpicture}
\end{minipage}

\end{example}

Here is the general algorithm to generate a tamed tree from a given signed skeleton tree.

\begin{algorithm}[Generate a Tamed tree]\label{algorithm:generate a tamed tree}
~\\
\hspace*{1em}$(1)$ Start with a queue that first contains only $1$.

$(2)$ If the queue is empty, then stop. If not, dequeue the leftmost entry $l$ of the queue and go to step $(3)$.

$(3)$ If there is a middle child of $l$, pass to the middle child of $l$, and label its left branch with the next available label $2j,2(j+1),...,2(j+q)$. If not, go to step $(5)$.

$(4)$ Take the left branch enumerated in step $(3)$ and first list all $+$ nodes in order from $2j,2(j+1),...,2(j+q)$ and add them to the right side of the queue, and then list in order all $-$ nodes from $2j,2(j+1),...,2(j+q)$ and add them to the right side of the queue. Set the next available label to be $2(j+q+1)$, and go to step $(5)$.

$(5)$ If there is a right child of $l$, pass to the right child of $l$, and label its left branch with the next available label $2j,2(j+1),...,2(j+q)$. If not, go to step $(2)$.

$(6)$ Take the left branch enumerated in step $(5)$ and first list all $+$ nodes in order from $2j,2(j+1),...,2(j+q)$ and add them to the right side of the queue, and then list in order all $-$ nodes from $2j,2(j+1),...,2(j+q)$ and add them to the right side of the queue.
Set the next available label to be $2(j+q+1)$, and go to step $(2)$.

\end{algorithm}

Next, we will explain how to execute a sequence of signed KM acceptable moves to bring a collapsing
 map pair $(\mu,sgn)$ into the tamed form. After presenting an example, we will give the general
form of the algorithm.

\begin{example}We consider the following collapsing map
$$
\begin{tabular}{c|ccccccccccccc}
$2j$&2&4&6&8&10&12&14&16&18&20&22&24&26\\
\hline
$\mu(2j)$ &1&1&1&6&1&6&7&1&2&16&9&18&3\\
$sgn(2j)$&$-$&$-$&+&$-$&+&+&+&$-$&$-$&+&$-$&+&+
\end{tabular}
$$
By Algorithm $\ref{algorithm:generate an admissible tree}$, it generates $T(\mu,sgn)$ as shown in Fig. $\ref{figure:admissible tree, to be tamed}$, which
has the same signed skeleton tree with the collapsing map $(\mu_{*},sgn_{*})$ in Example $\ref{example:from skeleton to tamed tree}$.

Comparing Fig. $\ref{figure:admissible tree, to be tamed}$ $T(u,sgn)$ with Fig. $\ref{figure:tamed tree, compared}$ $T(\mu_{*},sgn_{*})$, we note that the node $8$ on the $T(\mu_{*},sgn_{*})$ is the first one that differs from the one on the $T(\mu,sgn)$, which is labeled 10. To change node $10$ into node $8$, we do KM(8,10) on $(\mu,sgn)$.

\begin{figure}[htb]
\begin{minipage}{0.5\textwidth}
\begin{tikzpicture}
\node (1) at (0,0) {1};
\node (2) at (0,-1) {2$-$};
\node (4) at (-1.5,-2) {4$-$};
\node (6) at (-3,-3) {6+};
\node[draw,circle,inner sep=0.05cm,outer sep=0.1cm] (8) at (-4.5,-4) {10+};
\node (10) at (-6,-5) {16$-$};
\node (12) at (-3,-4) {8$-$};
\node (14) at (-4.5,-5) {12+};
\node (16) at (-1.5,-4) {14+};
\node (18) at (0,-2) {18$-$};
\node (20) at (1.2,-2) {26+};
\node (22) at (-6,-6) {20+};
\node (24) at (-1.5,-5) {22$-$};
\node (26) at (0,-3) {24+};
\draw[<-] (1)--(2);
\draw[-] (2)--(4);
\draw[<-] (2)--(18);
\draw[<-] (2)--(20);
\draw[-] (4)--(6);
\draw[-] (6)--(8);
\draw[<-] (6)--(12);
\draw[<-] (6)--(16);
\draw[-] (8)--(10);
\draw[<-] (10)--(22);
\draw[-] (12)--(14);
\draw[<-] (12)--(24);
\draw[<-] (18)--(26);
\end{tikzpicture}
\caption{$T(\mu,sgn)$}
\label{figure:admissible tree, to be tamed}
\end{minipage}%
\begin{minipage}{0.5\textwidth}
\begin{tikzpicture}
\node (1) at (0,0) {1};
\node (2) at (0,-1) {2$-$};
\node (4) at (-1.5,-2) {4$-$};
\node (6) at (-3,-3) {6+};
\node[draw,circle,inner sep=0.05cm,outer sep=0.1cm] (8) at (-4.5,-4) {8+};
\node (10) at (-6,-5) {10$-$};
\node (12) at (-3,-4) {12$-$};
\node (14) at (-4.5,-5) {14+};
\node (16) at (-1.5,-4) {16+};
\node (18) at (0,-2) {18$-$};
\node (20) at (1.2,-2) {20+};
\node (22) at (-6,-6) {22+};
\node (24) at (-1.5,-5) {24$-$};
\node (26) at (0,-3) {26+};
\draw[<-] (1)--(2);
\draw[-] (2)--(4);
\draw[<-] (2)--(18);
\draw[<-] (2)--(20);
\draw[-] (4)--(6);
\draw[-] (6)--(8);
\draw[<-] (6)--(12);
\draw[<-] (6)--(16);
\draw[-] (8)--(10);
\draw[<-] (10)--(22);
\draw[-] (12)--(14);
\draw[<-] (12)--(24);
\draw[<-] (18)--(26);
\end{tikzpicture}
\caption{$T(\mu_{*},sgn_{*})$}
\label{figure:tamed tree, compared}
\end{minipage}
\end{figure}

The KM(8,10) move is
\begin{align*}
&\mu_{1}=(8,10)\circ (9,11)\circ \mu \circ (8,10)\circ (9,11),\\
&sgn_{1}=sgn\circ (8,10)\circ (9,11).
\end{align*}
It gives that
$$
\begin{tabular}{c|ccccccccccccc}
$2j$&2&4&6&8&10&12&14&16&18&20&22&24&26\\
\hline
$\mu_{1}(2j)$ &1&1&1&1&6&6&7&1&2&16&11&18&3\\
$sgn_{1}(2j)$&$-$&$-$&+&+&$-$&+&+&$-$&$-$&+&$-$&+&+
\end{tabular}
$$

\begin{figure}[htb]
\begin{minipage}{0.5\textwidth}
\begin{tikzpicture}
\node (1) at (0,0) {1};
\node (2) at (0,-1) {2$-$};
\node (4) at (-1.5,-2) {4$-$};
\node (6) at (-3,-3) {6+};
\node (8) at (-4.5,-4) {8+};
\node[draw,circle,inner sep=0.05cm,outer sep=0.1cm] (10) at (-6,-5) {16$-$};
\node (12) at (-3,-4) {10$-$};
\node (14) at (-4.5,-5) {12+};
\node (16) at (-1.5,-4) {14+};
\node (18) at (0,-2) {18$-$};
\node (20) at (1.2,-2) {26+};
\node (22) at (-6,-6.2) {20+};
\node (24) at (-1.5,-5) {22$-$};
\node (26) at (0,-3) {24+};
\draw[<-] (1)--(2);
\draw[-] (2)--(4);
\draw[<-] (2)--(18);
\draw[<-] (2)--(20);
\draw[-] (4)--(6);
\draw[-] (6)--(8);
\draw[<-] (6)--(12);
\draw[<-] (6)--(16);
\draw[-] (8)--(10);
\draw[<-] (10)--(22);
\draw[-] (12)--(14);
\draw[<-] (12)--(24);
\draw[<-] (18)--(26);
\end{tikzpicture}
\caption{$T(\mu_{1},sgn_{1})$}
\label{figure:admissible tree, to be tamed, 1}
\end{minipage}%
\begin{minipage}{0.5\textwidth}
\begin{tikzpicture}
\node (1) at (0,0) {1};
\node (2) at (0,-1) {2$-$};
\node (4) at (-1.5,-2) {4$-$};
\node (6) at (-3,-3) {6+};
\node (8) at (-4.5,-4) {8+};
\node[draw,circle,inner sep=0.05cm,outer sep=0.1cm] (10) at (-6,-5) {10$-$};
\node (12) at (-3,-4) {12$-$};
\node (14) at (-4.5,-5) {14+};
\node (16) at (-1.5,-4) {16+};
\node (18) at (0,-2) {18$-$};
\node (20) at (1.2,-2) {20+};
\node (22) at (-6,-6.2) {22+};
\node (24) at (-1.5,-5) {24$-$};
\node (26) at (0,-3) {26+};
\draw[<-] (1)--(2);
\draw[-] (2)--(4);
\draw[<-] (2)--(18);
\draw[<-] (2)--(20);
\draw[-] (4)--(6);
\draw[-] (6)--(8);
\draw[<-] (6)--(12);
\draw[<-] (6)--(16);
\draw[-] (8)--(10);
\draw[<-] (10)--(22);
\draw[-] (12)--(14);
\draw[<-] (12)--(24);
\draw[<-] (18)--(26);
\end{tikzpicture}
\caption{$T(\mu_{*},sgn_{*})$}
\end{minipage}
\end{figure}
Next, we compare Fig. $\ref{figure:admissible tree, to be tamed, 1}$ $T(u_{1},sgn_{1})$ with $T(\mu_{*},sgn_{*})$ and find that the next different node is $10$ in the tree $T(u_{*},sgn_{*})$, which is corresponding to node 16 in the tree $T(u_{1},sgn_{1})$. Hence we do KM(14,16), KM(12,14) and KM(10,12) on $(\mu_{1},sgn_{1})$. Then we have
\begin{align*}
(\mu_{2},sgn_{2})=KM(10,12)\circ KM(12,14)\circ KM(14,16)(\mu_{1},sgn_{1})
\end{align*}
and
$$
\begin{tabular}{c|ccccccccccccc}
$2j$&2&4&6&8&10&12&14&16&18&20&22&24&26\\
\hline
$\mu_{2}(2j)$ &1&1&1&1&1&6&6&7&2&10&13&18&3\\
$sgn_{2}(2j)$&$-$&$-$&+&+&$-$&$-$&+&+&$-$&+&$-$&+&+
\end{tabular}
$$

\begin{figure}[htb]
\begin{minipage}{0.5\textwidth}
\begin{tikzpicture}
\node (1) at (0,0) {1};
\node (2) at (0,-1) {2$-$};
\node (4) at (-1.5,-2) {4$-$};
\node (6) at (-3,-3) {6+};
\node (8) at (-4.5,-4) {8+};
\node (10) at (-6,-5) {10$-$};
\node (12) at (-3,-4) {12$-$};
\node (14) at (-4.5,-5) {14+};
\node (16) at (-1.5,-4) {16+};
\node (18) at (0,-2) {18$-$};
\node[draw,circle,inner sep=0.05cm,outer sep=0.1cm] (20) at (1.2,-2) {26+};
\node (22) at (-6,-6) {20+};
\node (24) at (-1.5,-5) {22$-$};
\node (26) at (0,-3) {24+};
\draw[<-] (1)--(2);
\draw[-] (2)--(4);
\draw[<-] (2)--(18);
\draw[<-] (2)--(20);
\draw[-] (4)--(6);
\draw[-] (6)--(8);
\draw[<-] (6)--(12);
\draw[<-] (6)--(16);
\draw[-] (8)--(10);
\draw[<-] (10)--(22);
\draw[-] (12)--(14);
\draw[<-] (12)--(24);
\draw[<-] (18)--(26);
\end{tikzpicture}
\caption{$T(\mu_{2},sgn_{2})$}
\label{figure:admissible tree, to be tamed, 2}
\end{minipage}%
\begin{minipage}{0.5\textwidth}
\begin{tikzpicture}
\node (1) at (0,0) {1};
\node (2) at (0,-1) {2$-$};
\node (4) at (-1.5,-2) {4$-$};
\node (6) at (-3,-3) {6+};
\node (8) at (-4.5,-4) {8+};
\node (10) at (-6,-5) {10$-$};
\node (12) at (-3,-4) {12$-$};
\node (14) at (-4.5,-5) {14+};
\node (16) at (-1.5,-4) {16+};
\node (18) at (0,-2) {18$-$};
\node[draw,circle,inner sep=0.05cm,outer sep=0.1cm] (20) at (1.2,-2) {20+};
\node (22) at (-6,-6) {22+};
\node (24) at (-1.5,-5) {24$-$};
\node (26) at (0,-3) {26+};
\draw[<-] (1)--(2);
\draw[-] (2)--(4);
\draw[<-] (2)--(18);
\draw[<-] (2)--(20);
\draw[-] (4)--(6);
\draw[-] (6)--(8);
\draw[<-] (6)--(12);
\draw[<-] (6)--(16);
\draw[-] (8)--(10);
\draw[<-] (10)--(22);
\draw[-] (12)--(14);
\draw[<-] (12)--(24);
\draw[<-] (18)--(26);
\end{tikzpicture}
\caption{$T(\mu_{*},sgn_{*})$}
\end{minipage}
\end{figure}

Comparing Fig. $\ref{figure:admissible tree, to be tamed, 2}$ $T(u_{2},sgn_{2})$ and $T(\mu_{*},sgn_{*})$, we find that the next different node is 20 in the tree $T(\mu_{*},sgn_{*})$, which is corresponding to node 26 in the tree $T(\mu_{3},sgn_{3})$. Hence, we
do KM(24,26), KM(22,24) and KM(20,22) on $(u_{2},sgn_{2})$ to obtain
\begin{align*}
(\mu_{3},sgn_{3})=KM(20,22)\circ KM(22,24)\circ KM(24,26)(\mu_{2},sgn_{2})
\end{align*}
and
$$
\begin{tabular}{c|ccccccccccccc}
$2j$&2&4&6&8&10&12&14&16&18&20&22&24&26\\
\hline
$\mu_{3}(2j)$ &1&1&1&1&1&6&6&7&2&3&10&13&18\\
$sgn_{3}(2j)$&$-$&$-$&+&+&$-$&$-$&+&+&$-$&+&+&$-$&+
\end{tabular}
$$
We see that $(\mu_{3},sgn_{3})$ is just the tamed pair $(\mu_{*},sgn_{*})$ as shown in Example \ref{example:from skeleton to tamed tree}.

Here is a general algorithm to bring a collapsing map $(\mu,sgn)$ into the tamed form.
\begin{algorithm}[Tamed form]
~\\
\hspace*{1em}$(1)$ Given a collapsing map pair $(\mu,sgn)$, by Algorithm $\ref{algorithm:generate an admissible tree}$, we obtain a signed admissible tree $T(\mu,sgn)$. From the signed skeleton tree, by Algorithm $\ref{algorithm:generate a tamed tree}$, it generates a tamed tree $\al$.

$(2)$ Set counter $j=1$.

$(3)$ If the node $2j$ in the tame tree $\al$ is corresponding to $2l$ in the tree $T(\mu,sgn)$, then set
$$(u',sgn')=KM(2j,2j+2)\circ \ccc \circ KM(2l-4,2l-2)\circ KM(2l-2,2l) (\mu,sgn).$$

$(4)$ Set $(\mu,sgn)=(\mu',sgn')$. If $j=k$, then stop, otherwise set $j=j+1$ and go to step $(3)$.

\end{algorithm}
%\begin{remark}
%Then we can see that $(\mu_{*},sgn_{*})$ is in tamed form and the tree generated by $(\mu_{*},sgn_{*})$ equals to the tree $\al$.
%\end{remark}

\end{example}

Next, we arrive at the main part, that is, the following adaptation of Proposition \ref{prop:upper echelon form, time integration domain}.

\begin{proposition} \label{prop:tamed form, time integration domain}
Within a signed KM-relatable equivalence class of collapsing map pairs $(\mu,sgn)$, there is a unique tamed $(\mu_{*},sgn_{*})$. Moreover,
\begin{equation}\label{equ:tamed form, time integration domain}
\sum_{(\mu,sgn)\sim (\mu_{*},sgn_{*})}I(\mu,id,sgn,\ga^{(2k+1)})
=\int_{T_{D}(\mu_{*})}J_{\mu_{*},sgn_{*}}^{(2k+1)}(\ga^{(2k+1)})(t_{1},\underline{t}_{2k+1})
d\underline{t}_{2k+1}
\end{equation}
where $T_{D}(\mu_{*})$ is defined by $(\ref{equ:time integration domain generated by mu})$. Consequently,
\begin{equation} \label{equ:integration,tamed form, time integration domain}
\ga^{(1)}(t_{1})=\sum_{(\mu_{*},sgn_{*})\ \text{tamed}}\int_{T_{D}(\mu_{*})}J_{\mu_{*},sgn_{*}}^{(2k+1)}(\ga^{(2k+1)})(t_{1},\underline{t}_{2k+1})
d\underline{t}_{2k+1}
\end{equation}
where the number of tamed forms can be controlled by $16^{k}$.
\end{proposition}
\begin{proof}
The existence and uniqueness follow from Algorithm $\ref{algorithm:generate a tamed tree}$.
For $(\ref{equ:tamed form, time integration domain})$, the proof is the same as Proposition $\ref{prop:upper echelon form, time integration domain}$. As shown in $(\ref{equ:catalan number})$, the number of different ternary tree structures of $k$ nodes can be controlled by $8^{k}$. By paying an extra factor of $2^{k}$, which comes from the signs, there are at most $16^{k}$ tamed forms.
\end{proof}
The next step will be to rearrange the tamed pairs $(\mu_{*},sgn_{*})$ via wild moves, as defined
and discussed in the next section. This will produce a further reduction of $(\ref{equ:integration,tamed form, time integration domain})$.
\subsection{Wild Moves}\label{subsection:Wild Moves}
We then introduce wild moves which keeps the tamed form invariant so that we can partition the class of tamed
pairs $(\mu, sgn)$ into equivalence classes of wildly relatable forms.
\begin{definition}\label{def:allowable permutation}
Given a collapsing map pair $(\mu,sgn)$, define $G_{i}=\lr{2j:\mu(2j)=i}$ for $i=1,2,...,2k-1$. We call $\rho\in P$ allowable with respect to $(\mu,sgn)$ if it satisfies the following two conditions:

$(1)$ $\rho(G_{i})=G_{i}$ for $i=1,2,...,2k-1$.

$(2)$ If $2q < 2s$, $\mu(2q)=\mu(2s)$ and $sgn(2q)=sgn(2s)$, then $\rho(2q)< \rho(2s)$.

We denote the set of all allowable permutations $\rho$ with respect to $(\mu,sgn)$ by $P(\mu,sgn)$.
Note that condition $(1)$ is equivalent to $\mu\circ \rho =\mu\circ \rho^{-1}=\mu$, which leaves all left branchs invariant. Moreover, if $(\mu,sgn)$ is in tamed form, $G_{i}$ will be the form $\lr{2l,2l+2,...,2r}$.
\end{definition}

\begin{definition}[Wild move]\label{definiton:wild moves}
%A wild move $W(\rho)$ is defined as follows. Suppose $(\mu,sgn)$ is a collapsing operator/sgn map pair in tamed form, and $\lr{l,...,r}$ is a full left branch, i.e.
%\begin{align*}
%z:=\mu(2l)=\mu(2l+2)=\ccc=\mu(2r)
%\end{align*}
%but $\mu(2l-2)\neq z$ (or is undefined) and $\mu(2r+2)\neq z$ (or is undefined).
%
%Let $\rho$ be a permutation of $\lr{2l,2l+2,...,2r}$ that satisfies the following condition:
%if $2l<2q<2s<2r$ and $sgn(2q)=sgn(2s)$, then $\rho(2q)<\rho(2s)$.

Given a signed collapsing map $(\mu,sgn)$ and $\rho\in P(\mu,sgn)$, then the wild move $W(\rho)$ is defined as an action on a ternary $(\mu,\sigma,sgn)$, where
$$(\mu',\sigma',sgn')=W(\rho)(\mu,\sigma,sgn)$$
with
\begin{align*}
&\mu'=\rho \circ \mu \circ \rho^{-1}=\rho \circ \mu,\\
&\sigma'=\rho \circ \sigma,\\
&sgn'=sgn\circ \rho^{-1}.
\end{align*}
\end{definition}

%\begin{remark}
%We first consider $\rho=\tau_{2}\circ \tau_{1}$. Let $\tau_{1}$ be allowable with respect to $\mu$ and
%$\tau_{2}$ be allowable with respect to $W(\tau_{1})(\mu,sgn)$. Then $\rho$ is allowable with respect to $\mu$.
%\end{remark}
It is fairly straightforward to show that wild moves preserve the tamed class by using the definition of tamed form. It is noteworthy that the analogous statement for upper echelon forms does not hold and
it is the purpose of introducing the tamed class.
\begin{proposition}
Suppose $(\mu,sgn)$ is in tamed form, and $W(\rho)$ is a wild move defined as above.
Then $(\mu',sgn')$ is also tamed.
\end{proposition}
\begin{proof}
It follows from the definition of tamed form.
\end{proof}
Thus we can say that two tamed forms $(\mu,sgn)$ and $(\mu',sgn')$ are wildly relatable if
there exists an allowable permutation $\rho$ such that
$$(\mu',\sigma',sgn')=W(\rho)(\mu,\sigma,sgn).$$
This is an equivalence relation that partitions the set
of tamed forms into equivalence classes of wildly relatable forms.

The main result of this section is

\begin{proposition} \label{prop:wild moves,integration change}
Given a signed collapsing map $(u,sgn)$ in tamed form and $\rho\in P(\mu,sgn)$ as in Definition $\ref{def:allowable permutation}$, let
$$(\mu',\sigma',sgn')=W(\rho)(\mu,\sigma,sgn).$$
Then for any symmetric density $\ga^{(2k+1)}$, we have
\begin{equation} \label{equ:wild moves, expansion}
J_{\mu',sgn'}^{(2k+1)}(\ga^{(2k+1)})(t_{1},\sigma '^{-1}(\underline{t}_{2k+1}))=
J_{\mu,sgn}^{(2k+1)}(\ga^{(2k+1)})(t_{1},\sigma^{-1}(\underline{t}_{2k+1})).
\end{equation}
Consequently,
\begin{align}
\int_{\sigma[T_{D}(\mu)]}J_{\mu,sgn}^{(2k+1)}(\ga^{(2k+1)})(t_{1},
\underline{t}_{2k+1})d\underline{t}_{2k+1}=
\int_{\sigma'[T_{D}(\mu)]}J_{\mu',sgn'}^{(2k+1)}(\ga^{(2k+1)})(t_{1},
\underline{t}_{2k+1})d\underline{t}_{2k+1}
\end{align}
where $\sigma[T_{D}(\mu)]$ is defined as follows
$$\sigma[T_{D}(\mu)]=:\lr{t_{\sigma(2j)+1}\geq t_{\sigma(2l)+1}:
2l\to 2j\ \text{in the tree}\ T(\mu)}\bigcap \lr{t_{1}\geq t_{\sigma(2)+1}}.$$
%\begin{align*}
%J_{\mu',sgn'}^{(2k+3)}(f^{(2k+3)})(t_{1},\sigma'^{-1}(t_{2k+3}))=J_{\mu,sgn}^{(2k+3)}(f^{(2k+3)})
%(t_{1},\underline{t}_{2k+3})
%\end{align*}
%where
\end{proposition}
\begin{proof}
Since $(\mu,sgn)$ is a tamed pair, $G_{i}$ will be the form $\lr{2p,2p+2,...,2q}$. Then, $\rho\in P(\mu,sgn)$ can be written as a composition of permutations
$$\rho=\tau_{1}\circ \ccc\circ \tau_{s}$$
with the property that each $\tau=(2l,2l+2)\circ (2l+1,2l+3)$ and $sgn(2l)\neq sgn(2l+2)$.
Thus, it suffices to prove
\begin{align*}
&U^{(2j-1)}(-t_{2j+1})B_{\mu(2j);2j,2j+1}^{-}U^{(2j+1)}(t_{2j+1}-t_{2j+3})B^{+}_{\mu(2j+2);2j+2,2j+3}
U^{(2j+3)}(t_{2j+3})\\
=&U^{(2j-1)}(-t_{2j+3})B_{\mu(2j);2j+2,2j+3}^{+}\wt{U}^{(2j+1)}(t_{2j+3}-t_{2j+1})B^{-}_{\mu(2j+2);2j,2j+1}
U^{(2j+3)}(t_{2j+1})
\end{align*}
where $\wt{U}^{(2j+1)}(t)=U^{(2j-1)}(t)e^{it(\Delta_{x_{2j+2}}-\Delta_{x_{2j+2}'})}
e^{it(\Delta_{x_{2j+3}}-\Delta_{x_{2j+3}'})}$.

Without loss, we might as well take $j=1$ and $\mu(2)=\mu(4)=1$ so that this becomes
\begin{align}
U^{(1)}(-t_{3})B_{1;2,3}^{-}U^{(3)}(t_{3}-t_{5})B_{1;4,5}^{+}U^{(5)}(t_{5})
=U^{(1)}(-t_{5})B_{1;4,5}^{+}\wt{U}^{(3)}(t_{5}-t_{3})B_{1;2,3}^{-}U^{(5)}(t_{3}).
\end{align}
For simplicity, we take the following notations
\begin{align*}
&U^{(1)}(-t_{3})=U^{1}_{-3}U^{1'}_{-3},\\
&U^{(3)}(t_{3}-t_{5})=U_{3,5}^{1}U^{1'}_{3,5}U^{2}_{3,5}U^{2'}_{3,5}U^{3}_{3,5}U^{3'}_{3,5},\\
&U^{(5)}(t_{5})=U^{1}_{5}U^{1'}_{5}U^{2}_{5}U^{2'}_{5}U^{3}_{5}U^{3'}_{5}
U^{4}_{5}U^{4'}_{5}U^{5}_{5}U^{5'}_{5}.
\end{align*}
where $U_{\pm j}^{l}=e^{\pm it_{j}\Delta_{x_{l}}}$, $U_{\pm j}^{l'}=e^{\mp it_{j}\Delta_{x_{l}'}}$, $U_{j,k}^{l}=U_{j}^{l}U_{-k}^{l}$ and $U_{j,k}^{l'}=U_{j}^{l'}U_{-k}^{l'}$.

Expanding $U^{(1)}(-t_{3})$, $U^{(3)}(t_{3}-t_{5})$ and $U^{(5)}(t_{5})$ gives
\begin{align*}
&U^{(1)}(-t_{3})B_{1;2,3}^{-}U^{(3)}(t_{3}-t_{5})B_{1;4,5}^{+}U^{(5)}(t_{5})\\
=&U^{1}_{-3}U^{1'}_{-3}B_{1;2,3}^{-}
U_{3,5}^{1}U^{1'}_{3,5}U^{2}_{3,5}U^{2'}_{3,5}U^{3}_{3,5}U^{3'}_{3,5}
B_{1;4,5}^{+}U^{1}_{5}U^{1'}_{5}U^{2}_{5}U^{2'}_{5}U^{3}_{5}U^{3'}_{5}
U^{4}_{5}U^{4'}_{5}U^{5}_{5}U^{5'}_{5}
\end{align*}
Since $B_{1;2,3}^{-}$ acts only on the $2$, $2'$, $3$, $3'$ and $1'$ coordinates, we exchange $B_{1;2,3}^{-}$ with $U_{3,5}^{1}$. In the same way, $B_{1;4,5}^{+}$ acts only on $4$, $4'$, $5$, $5'$ and $1$ coordinates, so we exchange $B_{1;4,5}^{+}$ with $U^{1'}_{3,5}U^{2}_{3,5}U^{2'}_{3,5}U^{3}_{3,5}U^{3'}_{3,5}$. Thus, we have
\begin{align*}
&U^{(1)}(-t_{3})B_{1;2,3}^{-}U^{(3)}(t_{3}-t_{5})B_{1;4,5}^{+}U^{(5)}(t_{5})\\
=&U^{1}_{-3}U^{1'}_{-3}U_{3,5}^{1}B_{1;2,3}^{-}
B_{1;4,5}^{+}U^{1'}_{3,5}U^{2}_{3,5}U^{2'}_{3,5}U^{3}_{3,5}U^{3'}_{3,5}
U^{1}_{5}U^{1'}_{5}U^{2}_{5}U^{2'}_{5}U^{3}_{5}U^{3'}_{5}
U^{4}_{5}U^{4'}_{5}U^{5}_{5}U^{5'}_{5}
\end{align*}
Exchanging $B_{1;2,3}^{-}$ with $B_{1;4,5}^{+}$ gives
\begin{align*}
=&U^{1}_{-3}U^{1'}_{-3}U_{3,5}^{1}
B_{1;4,5}^{+}
B_{1;2,3}^{-}
U^{1'}_{3,5}U^{2}_{3,5}U^{2'}_{3,5}U^{3}_{3,5}U^{3'}_{3,5}
U^{1}_{5}U^{1'}_{5}U^{2}_{5}U^{2'}_{5}U^{3}_{5}U^{3'}_{5}
U^{4}_{5}U^{4'}_{5}U^{5}_{5}U^{5'}_{5}
\end{align*}
With $U^{1}_{-3}U^{1'}_{-3}U_{3,5}^{1}=U^{1}_{-5}U^{1'}_{-5}U_{5,3}^{1'}$ and $U^{1'}_{3,5}
U^{1}_{5}U^{1'}_{5}=U^{1}_{5,3}U^{1}_{3}U^{1'}_{3}$, we obtain
\begin{align*}
=&U^{1}_{-5}U^{1'}_{-5}U_{5,3}^{1'}
B_{1;4,5}^{+}
B_{1;2,3}^{-}
U^{1}_{5,3}U^{1}_{3}U^{1'}_{3}U^{2}_{3}U^{2'}_{3}U^{3}_{3}U^{3'}_{3}
U^{4}_{5}U^{4'}_{5}U^{5}_{5}U^{5'}_{5}
\end{align*}
Exchanging  $U_{5,3}^{1'}$ with
$B_{1;4,5}^{+}$ and $B_{1;2,3}^{-}$ with $U^{1}_{5,3}U^{4}_{3}U^{4'}_{3}U^{5}_{3}U^{5'}_{3}$, we have
\begin{align*}
=&U^{1}_{-5}U^{1'}_{-5}
B_{1;4,5}^{+}U_{5,3}^{1'}U^{1}_{5,3}U^{4}_{5,3}U^{4'}_{5,3}U^{5}_{5,3}U^{5'}_{5,3}
B_{1;2,3}^{-}
U^{1}_{3}U^{1'}_{3}U^{2}_{3}U^{2'}_{3}U^{3}_{3}U^{3'}_{3}
U^{4}_{3}U^{4'}_{3}U^{5}_{3}U^{5'}_{3}\\
=&U^{(1)}(-t_{5})B_{1;4,5}^{+}\wt{U}^{(3)}(t_{5}-t_{3})B_{1;2,3}^{-}U^{(5)}(t_{3}).
\end{align*}
Since $\ga^{(2k+1)}$ is a symmetric density, one can permute
$$(x_{2},x_{3},x_{4},x_{5};x_{2}',x_{3}',x_{4}',x_{5}')\leftrightarrow
(x_{4},x_{5},x_{2},x_{3};x_{4}',x_{5}',x_{2}',x_{3}').$$
Then it gives that
\begin{align*}
&U^{(1)}(-t_{5})B_{1;4,5}^{+}\wt{U}^{(3)}(t_{5}-t_{3})B_{1;2,3}^{-}U^{(5)}(t_{3})\longmapsto
U^{(1)}(-t_{5})B_{1;2,3}^{+}U^{(3)}(t_{5}-t_{3})B_{1;4,5}^{-}U^{(5)}(t_{3}),\\
&B^{\pm}_{\mu(2l);2l,2l+1}\longmapsto B^{\pm}_{(2,4)\circ(3,5)\circ \mu(2l);2l,2l+1},\quad l\geq 3,
\end{align*}
which proves equality $(\ref{equ:wild moves, expansion})$.
\end{proof}
\begin{example}\label{example:reference tree}Let us work with the following pair $(\mu_{1},sgn_{1})$
$$
\begin{tabular}{c|ccccccc}
$2j$&2&4&6&8&10&12&14\\
\hline
$\mu_{1}(2j)$ &1&1&1&2&3&7&7\\
$sgn_{1}(2j)$ &+&+&$-$&$-$&+&+&$-$
\end{tabular}
$$
There are six wild moves as follows:
\begin{align*}
(\mu_{j},\sigma_{j},sgn_{j})=W(\rho_{j})(\mu_{1},id,sgn_{1}),
\end{align*}
$$
\begin{tabular}{c|ccc|cc||c|ccc|cc|c}
&2&4&6&12&14&&2&4&6&12&14 &$\rho_{j}^{-1}[T_{D}(\mu_{j})]$\\
\hline
$\rho_{1}$&2&4&6&12&14&$\rho_{1}^{-1}$&2&4&6&12&14&$\lr{t_{3}\geq t_{5}\geq t_{7},t_{7}\geq t_{13}\geq t_{15},t_{3}\geq t_{9},t_{3}\geq t_{11}}$\\
$\rho_{2}$&2&6&4&12&14&$\rho_{2}^{-1}$&2&6&4&12&14&$\lr{t_{3}\geq t_{7}\geq t_{5},t_{7}\geq t_{13}\geq t_{15},t_{3}\geq t_{9},t_{3}\geq t_{11}}$\\
$\rho_{3}$&4&6&2&12&14&$\rho_{3}^{-1}$&6&2&4&12&14&$\lr{t_{7}\geq t_{3}\geq t_{5},t_{7}\geq t_{13}\geq t_{15},t_{3}\geq t_{9},t_{3}\geq t_{11}}$\\
$\rho_{4}$&2&4&6&14&12&$\rho_{4}^{-1}$&2&4&6&14&12&$\lr{t_{3}\geq t_{5}\geq t_{7},t_{7}\geq t_{15}\geq t_{13},t_{3}\geq t_{9},t_{3}\geq t_{11}}$\\
$\rho_{5}$&2&6&4&14&12&$\rho_{5}^{-1}$&2&6&4&14&12&$\lr{t_{3}\geq t_{7}\geq t_{5},t_{7}\geq t_{15}\geq t_{13},t_{3}\geq t_{9},t_{3}\geq t_{11}}$\\
$\rho_{6}$&4&6&2&14&12&$\rho_{6}^{-1}$&6&2&4&14&12&$\lr{t_{7}\geq t_{3}\geq t_{5},t_{7}\geq t_{15}\geq t_{13},t_{3}\geq t_{9},t_{3}\geq t_{11}}$
\end{tabular}
$$
The collapsing mappings $(\mu_{j},sgn_{j})$ and corresponding trees are indicated below. We notice that all $(\mu_{j},sgn_{j})$ are
tamed and also that wild moves, unlike the KM moves, do change the skeleton.
$$
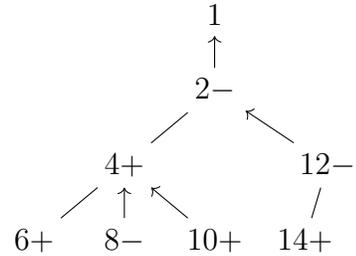
\begin{tabular}{c|ccccccc||c|ccccccc}
$2j$&2&4&6&8&10&12&14 & &2&4&6&8&10&12&14\\
\hline
$\mu_{1}(2j)$ &1&1&1&2&3&7&7 &$sgn_{1}(2j)$ &+&+&$-$&$-$&+&+&$-$\\
$\mu_{2}(2j)$ &1&1&1&2&3&5&5 &$sgn_{2}(2j)$ &+&$-$&$+$&$-$&+&+&$-$\\
$\mu_{3}(2j)$ &1&1&1&4&5&3&3 &$sgn_{3}(2j)$ &$-$&+&$+$&$-$&+&+&$-$\\
$\mu_{4}(2j)$ &1&1&1&2&3&7&7 &$sgn_{4}(2j)$ &+&+&$-$&$-$&+&$-$&+\\
$\mu_{5}(2j)$ &1&1&1&2&3&5&5 &$sgn_{5}(2j)$ &+&$-$&+&$-$&+&$-$&+\\
$\mu_{6}(2j)$ &1&1&1&4&5&3&3 &$sgn_{6}(2j)$ &$-$&+&+&$-$&+&$-$&+\\
\end{tabular}
$$

\begin{minipage}[t]{0.3\textwidth}
\centering
Tree for $(\mu_{1},sgn_{1})$
\begin{align*}
\begin{tikzpicture}
\node (1) at (0,0) {1};
\node (2) at (0,-1) {2+};
\node (4) at (-1.2,-2) {4+};
\node (8) at (0,-2) {8$-$};
\node (10) at (1.2,-2) {10+};
\node (6) at (-2.4,-3) {6$-$};
\node (12) at  (-1.2,-4) {12+};
\node (14) at  (-2.4,-5) {14$-$};
\draw[<-] (1)--(2);
\draw[-] (2)--(4);
\draw[<-] (2)--(8);
\draw[<-] (2)--(10);
\draw[-] (4)--(6);
\draw[<-] (6)--(12);
\draw[-] (12)--(14);
\end{tikzpicture}
\end{align*}
\end{minipage}
\begin{minipage}[t]{0.3\textwidth}
\centering

Tree for $(\mu_{2},sgn_{2})$
\begin{align*}
\begin{tikzpicture}
\node (1) at (0,0) {1};
\node (2) at (0,-1) {2+};
\node (4) at (-1.2,-2) {4$-$};
\node (8) at (0,-2) {8$-$};
\node (10) at (1.2,-2) {10+};
\node (6) at (-2.4,-3) {6+};
\node (12) at  (0,-3) {12+};
\node (14) at  (-1.2,-4) {14$-$};
\draw[<-] (1)--(2);
\draw[-] (2)--(4);
\draw[<-] (2)--(8);
\draw[<-] (2)--(10);
\draw[-] (4)--(6);
\draw[<-] (4)--(12);
\draw[-] (12)--(14);
\end{tikzpicture}
\end{align*}
\end{minipage}
\begin{minipage}[t]{0.3\textwidth}
\centering
Tree for $(\mu_{3},sgn_{3})$
\begin{tikzpicture}
\node (1) at (0,0) {1};
\node (2) at (0,-1) {2$-$};
\node (4) at (-1.2,-2) {4+};
\node (8) at (-1.2,-3) {8$-$};
\node (10) at (0,-3) {10+};
\node (6) at (-2.4,-3) {6+};
\node (12) at  (1.5,-2) {12+};
\node (14) at  (1,-3) {14$-$};
\draw[<-] (1)--(2);
\draw[-] (2)--(4);
\draw[<-] (2)--(12);
\draw[-] (12)--(14);
\draw[-] (4)--(6);
\draw[<-] (4)--(8);
\draw[<-] (4)--(10);
\end{tikzpicture}
\end{minipage}

\begin{minipage}[t]{0.3\textwidth}
\centering
Tree for $(\mu_{4},sgn_{4})$
\begin{tikzpicture}
\node (1) at (0,0) {1};
\node (2) at (0,-1) {2+};
\node (4) at (-1.2,-2) {4+};
\node (8) at (0,-2) {8$-$};
\node (10) at (1.2,-2) {10+};
\node (6) at (-2.4,-3) {6$-$};
\node (12) at  (-1.2,-4) {12$-$};
\node (14) at  (-2.4,-5) {14+};
\draw[<-] (1)--(2);
\draw[-] (2)--(4);
\draw[<-] (2)--(8);
\draw[<-] (2)--(10);
\draw[-] (4)--(6);
\draw[<-] (6)--(12);
\draw[-] (12)--(14);
\end{tikzpicture}
\end{minipage}
\begin{minipage}[t]{0.3\textwidth}
\centering
Tree for $(\mu_{5},sgn_{5})$
\begin{tikzpicture}
\node (1) at (0,0) {1};
\node (2) at (0,-1) {2+};
\node (4) at (-1.2,-2) {4$-$};
\node (8) at (0,-2) {8$-$};
\node (10) at (1.2,-2) {10+};
\node (6) at (-2.4,-3) {6+};
\node (12) at  (0,-3) {12$-$};
\node (14) at  (-1.2,-4) {14+};
\draw[<-] (1)--(2);
\draw[-] (2)--(4);
\draw[<-] (2)--(8);
\draw[<-] (2)--(10);
\draw[-] (4)--(6);
\draw[<-] (4)--(12);
\draw[-] (12)--(14);
\end{tikzpicture}
\end{minipage}
\begin{minipage}[t]{0.3\textwidth}
\centering
Tree for $(\mu_{6},sgn_{6})$
\begin{tikzpicture}
\node (1) at (0,0) {1};
\node (2) at (0,-1) {2$-$};
\node (4) at (-1.2,-2) {4+};
\node (8) at (-1.2,-3) {8$-$};
\node (10) at (0,-3) {10+};
\node (6) at (-2.4,-3) {6+};
\node (12) at  (1.5,-2) {12$-$};
\node (14) at  (1.2,-3) {14+};
\draw[<-] (1)--(2);
\draw[-] (2)--(4);
\draw[<-] (2)--(12);
\draw[-] (12)--(14);
\draw[-] (4)--(6);
\draw[<-] (4)--(8);
\draw[<-] (4)--(10);
\end{tikzpicture}
\end{minipage}

\end{example}

%\subsection{Reference Forms and the Tamed Integration Domain}\label{subsection:Reference Forms and the Tamed Integration Domain}
\subsection{Reference Form and Proof of Compatibility}\label{subsection:Reference Form and Proof of Compatibility}
 We will prove that, given a tamed class, there is a reference form representing the
tamed class. Moreover, the tamed time integration domain for the whole tamed class, which can be directly read
out from the reference form, is just the compatible time integration domain introduced in Section $\ref{subsection:Compatible Time Integration Domain}$.
\begin{definition}\label{def:reference pair}
A tamed pair $(\hat{\mu},\hat{sgn})$ will be called a reference pair provided that
in every left branch, all the $+$ nodes come before all the $-$ nodes.
\end{definition}
%\begin{remark}
%Equivalently, if $\mu(2j)=\mu(2l)$, $sgn(2j)\neq sgn(2l)$ and $2j<2l$, then $sgn(2j)=+$ and $sgn(2l)=-$.
%\end{remark}

\begin{example} The collapsing pair $(\mu_{1},sgn_{1})$ in Example $\ref{example:reference tree}$ is a reference pair.
$$
\begin{tabular}{c|ccc|c|c|cc}
$2j$&2&4&6&8&10&12&14\\
\hline
$sgn_{1}(2j)$ &+&+&$-$&$-$&+&+&$-$
\end{tabular}
$$
From the table, we can see that the $+$ nodes come before all the $-$ nodes in left branches $(2,4,6)$ and $(12,14)$.
\end{example}
%\begin{definition}
%Given a reference pair $(\hat{\mu},\hat{sgn})$ and $\rho\in P$, we will call $\rho$ allowable if it meets the conditions in Definition $\ref{definiton:wild moves}$, i.e. it leaves all left branches invariant and moreover, for each left branch $(2l,...,2r)$, if $2l<2q<2s<2r$ and $sgn(2q)=sgn(2s)$, then $\rho(2q)<\rho(2s)$. If the left branch only contains one node $2q$, then $\rho(2q)=2q$.
%\end{definition}
As we infer from the examples, each class can be represented by a unique reference pair $(\mu,sgn)$. Exactly, we have the following proposition.
\begin{proposition}\label{prop:reference form, uniqueness, tamed pairs}
An equivalence class of wildly relatable tamed pairs
$$Q=\lr{(\mu,sgn)} $$
contains a unique reference pair $(\hat{\mu},\hat{sgn})$. For every $(u,sgn)\in Q$, there is a unique allowable permutation $\rho\in P(\hat{\mu},\hat{sgn})$ such that
$$(\mu,sgn)=W(\rho)(\hat{\mu},\hat{sgn}).$$
\end{proposition}
\begin{proof}
%It suffices to consider one left branch.
Note that wild moves will not destroy the left branch but permute the signs. Thus, there exists an allowable permutation such that the $+$ nodes come before all the $-$ nodes in every left branch. The uniqueness follows from the conditions $(1)$ and $(2)$ in Definition \ref{def:allowable permutation}.
\end{proof}

Next, we get into the analysis of the main result.

\begin{proposition}\label{prop:reference form, time integration domain}
The Duhamel expansion to coupling order $k$ can be grouped into at most
$16^{k}$ terms:
\begin{equation}\label{equ:reference form, time integration domain}
\ga^{(1)}(t_{1})=\sum_{reference\ (\hat{\mu},\hat{sgn})}
\int_{T_{R}(\hat{\mu},\hat{sgn})}J_{\hat{\mu},\hat{sgn}}^{(2k+1)}(\ga^{(2k+1)})
(t_{1},\underline{t}_{2k+1})d\underline{t}_{2k+1}
\end{equation}
where
\begin{equation}\label{equ:reference form domain}
T_{R}(\hat{\mu},\hat{sgn})=\bigcup_{\rho\in P(\hat{\mu},\hat{sgn})}\rho^{-1}[T_{D}(\rho\circ \hat{\mu})].
\end{equation}
and $T_{D}(\mu)$ is defined by $(\ref{equ:time integration domain generated by mu})$.
\end{proposition}
\begin{proof}
Recall
\begin{equation*}
\ga^{(1)}(t_{1})=\sum_{(\mu_{*},sgn_{*})\ \text{tamed}}\int_{T_{D}(\mu_{*})}J_{\mu_{*},sgn_{*}}^{(2k+1)}(\ga^{(2k+1)})(t_{1},\underline{t}_{2k+1})
d\underline{t}_{2k+1}
\end{equation*}
where the number of tamed forms can be controlled by $16^{k}$.
In this sum, group together equivalence classes $Q$ of wildly relatable $(\mu,sgn)$.
\begin{equation*}
\ga^{(1)}(t_{1})=\sum_{\text{class}\ Q}\sum_{(\mu,sgn)\in Q}\int_{T_{D}(\mu)}J_{\mu,sgn}^{(2k+1)}(\ga^{(2k+1)})(t_{1},\underline{t}_{2k+1})
d\underline{t}_{2k+1}.
\end{equation*}
There exists exact one reference $(\hat{\mu},\hat{sgn})$ in each equivalence class $Q$.
By Proposition $\ref{prop:reference form, uniqueness, tamed pairs}$, for each $(\mu,sgn)\in Q$, there is a unique allowable
$\rho\in P(\hat{\mu},\hat{sgn})$ such that
$$(\mu,sgn)=W(\rho)(\hat{\mu},\hat{sgn}).$$
Since $W$ is an action, we can write
$$(\hat{\mu},\rho^{-1},\hat{sgn})=W(\rho^{-1})(\mu,id,sgn).$$
Then by Proposition $\ref{prop:wild moves,integration change}$,
\begin{align*}
\int_{T_{D}(\mu)}J_{\mu,sgn}^{(2k+1)}(\ga^{(2k+1)})(t_{1},\underline{t}_{2k+1})d\underline{t}_{2k+1}=
\int_{\rho^{-1}[T_{D}(\mu)]}J_{\hat{\mu},\hat{sgn}}^{(2k+1)}(\ga^{(2k+1)})(t_{1},\underline{t}_{2k+1})
d\underline{t}_{2k+1}.
\end{align*}
Consequently, we obtain
\begin{equation*}
\ga^{(1)}(t_{1})=\sum_{reference\ (\hat{\mu},\hat{sgn})}\sum_{\rho\in P(\hat{\mu},\hat{sgn})}
\int_{\rho^{-1}[T_{D}(\rho\circ \hat{\mu})]}J_{\hat{\mu},\hat{sgn}}^{(2k+1)}(\ga^{(2k+1)})
(t_{1},\underline{t}_{2k+1})d\underline{t}_{2k+1}.
\end{equation*}
Since $\lr{\rho^{-1}[T_{D}(\rho\circ \hat{\mu})]}$ is a collection of disjoint sets, we obtain the equality $(\ref{equ:reference form, time integration domain})$.

\end{proof}

We are left to calculate the time integration domain $T_{D}(\mu)$ and $T_{R}(\hat{\mu},\hat{sgn})$.

\begin{proposition}\label{prop:reference time integration domain}
Let $\rho\in P(\hat{\mu},\hat{sgn})$ and $(\mu,sgn)=W(\rho)(\hat{\mu},\hat{sgn})$, then
\begin{align} \label{equ:time integration domain under wild move}
T_{D}(\mu)=&\lr{t_{2j+1}\geq t_{2l+1}:\hat{\mu}(2j)=\hat{\mu}(2l),2j<2l}\\
&\bigcap
\lr{t_{\rho(2j)+1}\geq t_{\rho(2l)+1}:\hat{\mu}(2l)=2j\ \text{or}\ \hat{\mu}(2l)=2j+1},\notag
\end{align}

\begin{align} \label{equ:reference time integration domain}
T_{R}(\hat{\mu},\hat{sgn})
=&\lr{t_{2j+1}\geq t_{2l+1}:2j<2l,\hat{\mu}(2l)=\hat{\mu}(2j),\ \hat{sgn}(2j)=\hat{sgn}(2l)}\\
&\bigcap \lr{t_{2j+1}\geq t_{2l+1}:\hat{\mu}(2l)=2j\ \text{or}\
\hat{\mu}(2l)=2j+1}.\notag
\end{align}
\end{proposition}
\begin{proof}
Since $\mu=\rho \circ \hat{\mu}$, we can write
\begin{align*}
&T_{D}(\mu)\\
=&\lr{t_{2j+1}\geq t_{2l+1}:\hat{\mu}(2j)=\hat{\mu}(2l),2j<2l}\bigcap
\lr{t_{2j+1}\geq t_{2l+1}:\mu(2l)=2j\ \text{or}\ \mu(2l)=2j+1}.
\end{align*}

It remains to prove
\begin{align}\label{equ:time integration domain under wild move, equivalent form}
&\lr{t_{2j+1}\geq t_{2l+1}:\mu(2l)=2j\ \text{or}\ \mu(2l)=2j+1}\\
=&
\lr{t_{\rho(2j)+1}\geq t_{\rho(2l)+1}:\hat{\mu}(2l)=2j\ \text{or}\ \hat{\mu}(2l)=2j+1}.\notag
\end{align}
Actually, with $\mu=\rho\circ \hat{\mu}$ and $\hat{\mu}\circ \rho^{-1}=\hat{\mu}$, we have
\begin{align*}
&\mu(2l)=2j\Longleftrightarrow \hat{\mu}(\rho^{-1}(2l))=\rho^{-1}(2j),\\
&\mu(2l)=2j+1\Longleftrightarrow \hat{\mu}(\rho^{-1}(2l))=\rho^{-1}(2j)+1,
\end{align*}
which implies $(\ref{equ:time integration domain under wild move, equivalent form})$.

Then by $(\ref{equ:time integration domain under wild move})$, we can rewrite
\begin{align*}
\rho^{-1}[T_{D}(\rho\circ \hat{\mu})]=&\lr{t_{\rho^{-1}(2j)+1}\geq t_{\rho^{-1}(2l)+1}:\hat{\mu}(2j)=\hat{\mu}(2l),2j<2l}\\
&\bigcap
\lr{t_{2j+1}\geq t_{2l+1}:\hat{\mu}(2l)=2j\ \text{or}\
\hat{\mu}(2l)=2j+1}.
\end{align*}
It suffices to prove
\begin{align}
&\bigcup_{\rho\in P(\hat{\mu},\hat{sgn})}\lr{t_{\rho^{-1}(2j)+1}\geq t_{\rho^{-1}(2l)+1}:\hat{\mu}(2j)=\hat{\mu}(2l),2j<2l}\\
=&\lr{t_{2j+1}\geq t_{2l+1}:2j<2l,\hat{\mu}(2l)=\hat{\mu}(2j),\ \hat{sgn}(2j)=\hat{sgn}(2l)}.\notag
\end{align}

For simplicity, we take the notations
\begin{align*}
&A_{j,l}(\rho)=\lr{t_{\rho^{-1}(2j)+1}\geq t_{\rho^{-1}(2l)+1}:2j<2l,\hat{\mu}(2l)=\hat{\mu}(2j)},\\
&B_{j,l}=\lr{t_{2j+1}\geq t_{2l+1}:2j<2l,\hat{\mu}(2l)=\hat{\mu}(2j),\ \hat{sgn}(2j)=\hat{sgn}(2l)},
\end{align*}
where $A_{j,l}$ and $B_{j,l}$ will be the full space if $(j,l)$ does not satisfy the corresponding requirement. We are left to prove that
$$\bigcup_{\rho\in P(\hat{\mu},\hat{sgn})}\bigcap_{j,l}A_{j,l}(\rho)=\bigcap_{j,l}B_{j,l}.$$

Given $\rho\in P(\hat{\mu},\hat{sgn})$, we will prove $\bigcap_{j,l}A_{j,l}(\rho)\subset B_{j_{0},l_{0}}$ for every pair $(j_{0},l_{0})$ which satisfies $2j_{0}<2l_{0}$,
$\hat{\mu}(2l_{0})=\hat{\mu}(2j_{0})$ and
$\hat{sgn}(2j_{0})=\hat{sgn}(2l_{0})$. Let $2j_{1}=\rho(2j_{0})$ and $2l_{1}=\rho(2j_{0})$. Since $\rho\in P(\hat{\mu},\hat{sgn})$, we obtain $\hat{\mu}(2l_{1})=\hat{\mu}(2j_{1})$ and $2j_{1}< 2l_{1}$. Hence,
\begin{align*}
\bigcap_{j,l}A_{j,l}(\rho)\subset A_{j_{1},l_{1}}(\rho)=B_{j_{0},l_{0}}.
\end{align*}

Conversely, suppose that $(t_{1},t_{3},...,t_{2k+1})\in \bigcap_{j,l}B_{j,l}$. Note that $\lr{G_{i}=\lr{2r:\hat{\mu}(2r)=i}}_{i=1}^{2k-1}$ is a partition of $\lr{2,4,...,2k}$. Thus there exists a unique $\sigma\in P$ such that
\begin{align}\label{equ:construct wild move}
\begin{cases}
&\sigma(G_{i})=G_{i},\\
&t_{\sigma^{-1}(2j)+1}\geq t_{\sigma^{-1}(2l)+1}.
\end{cases}
\end{align}
where $2j<2l$ and $\hat{\mu}(2j)=\hat{\mu}(2l)$. It implies that $(t_{1},t_{3},...,t_{2k+1})\in \bigcap_{j,l}A_{j,l}(\sigma)$.

We are left to prove that $\sigma\in P(\hat{\mu},\hat{sgn})$. For any pair $(j_{0},l_{0})$ which satisfies $2l_{0}<2j_{0}$, $\hat{\mu}(2l_{0})=\hat{\mu}(2j_{0})$ and  $\hat{sgn}(2j_{0})=\hat{sgn}(2l_{0})$, we have
$(t_{1},...,t_{2k+1})\in B_{j_{0},l_{0}}$, which implies that $t_{2j_{0}+1}\geq t_{2l_{0}+1}$.
Combining with $(\ref{equ:construct wild move})$, we obtain $\sigma(2l_{0})<\sigma(2j_{0})$, which shows that $\sigma\in P(\hat{\mu},\hat{sgn})$.
\end{proof}

With Propositions \ref{prop:reference form, time integration domain} and \ref{prop:reference time integration domain}, we arrive at the main result as follows.
\begin{proposition}\label{prop:compatibility for time integration}
The time integration domain obtained in $(\ref{equ:reference form domain})$ is compatible in the sense that
\begin{align}
T_{R}(\hat{\mu},\hat{sgn})=T_{C}(\hat{\mu},\hat{sgn})
\end{align}
and hence
\begin{equation} \label{equ:duhamel expansion, reference form}
\ga^{(1)}(t_{1})=\sum_{reference\ (\hat{\mu},\hat{sgn})}
\int_{T_{C}(\hat{\mu},\hat{sgn})}J_{\hat{\mu},\hat{sgn}}^{(2k+1)}(\ga^{(2k+1)})
(t_{1},\underline{t}_{2k+1})d\underline{t}_{2k+1}
\end{equation}
where $T_{C}(\hat{\mu},\hat{sgn})=\lr{t_{2j+1}\geq t_{2l+1}:D^{(2l)}\to D^{(2j)}}$ is the compatible time integration domain defined by $(\ref{definition:compatible time integration domain})$.
\begin{proof}

From the definition of $T_{C}(\hat{\mu},\hat{sgn})$, we have that $t_{2j+1}\geq t_{2l+1}$ if and only if one of the following cases holds
$$\begin{cases}
&\hat{\mu}(2j)=\hat{\mu}(2l), \quad \hat{sgn}(2j)=\hat{sgn}(2l),\\
&\hat{\mu}(2l)=2j,\hat{sgn}(2l)=+,\\
&\hat{\mu}(2l)=2j,\hat{sgn}(2l)=-,\\
&\hat{\mu}(2l)=2j+1,\hat{sgn}(2l)=+,\\
&\hat{\mu}(2l)=2j+1,\hat{sgn}(2l)=-,
\end{cases} $$
where $2l>2j$ is the the minimal index for which the corresponding equalities hold. The requirement that $2l$ is the minimal index can be removed by induction argument. Thus, these cases are respectively corresponding to
$$\begin{cases}
&\hat{\mu}(2l)=\hat{\mu}(2j),\ \hat{sgn}(2j)=\hat{sgn}(2l),\\
&\hat{\mu}(2l)=2j,\\
&\hat{\mu}(2l)=2j+1,
\end{cases} $$
which implies that $T_{R}(\hat{\mu},\hat{sgn})=T_{C}(\hat{\mu},\hat{sgn})$.
\end{proof}
\end{proposition}

\section{$U$-$V$ Multilinear Estimates} \label{section:Multilinear Estimates}
 Our proof of $U$-$V$ multilinear estimates will focus on the $\T^{d}$ case, as it works the same for $\R^{d}$ with the homogeneous norm. We recall the definition of $U$-$V$ spaces in Section \ref{subsection:Estimates using the $U$-$V$ multilinear estimates} and use the following tools to prove $U$-$V$ multilinear estimates.
\begin{lemma}\cite[Propositions 2.11]{herr2011global} \label{lemma:U-V estimate,dual argument}
 For $f\in L^{1}(0,T;H^{s}(\T^{d}))$, we have
\begin{align}
&\bbn{\int_{a}^{t}e^{i(t-\tau)\Delta}f(\tau,\cdot)d\tau}_{X^{s}([0,T))}
&\leq \sup_{g\in Y^{-s}([0,T)):\n{g}_{Y^{-s}}=1}\bbabs{\int_{0}^{T}\int_{\T^{d}}
f(t,x)\ol{g(t,x)}dtdx},
\end{align}
for all $a\in[0,T)$.
\end{lemma}

%\begin{lemma}\label{lemma:Strichartz estimate}
%For $p>\frac{2(d+2)}{d},$
%\begin{align}
%\n{P_{\leq M}u}_{L_{t,x}^{p}}\lesssim M^{\frac{d}{2}-\frac{d+2}{p}}\n{u}_{U_{\Delta}^{p}L^{2}}
%\lesssim M^{\frac{d}{2}-\frac{d+2}{p}}\n{P_{\leq M}u}_{Y^{0}([0,T))}
%\end{align}
%\end{lemma}
\begin{lemma}[Strichartz estimate on $\T^{d}$\cite{bourgain2015the,killip2016scale}]\label{lemma:Strichartz estimate}
For $p>\frac{2(d+2)}{d},$
\begin{align}\label{equ:Strichartz estimate}
\n{P_{\leq M}u}_{L_{t,x}^{p}}
\lesssim M^{\frac{d}{2}-\frac{d+2}{p}}\n{P_{\leq M}u}_{Y^{0}([0,T))}
\end{align}
\end{lemma}

\begin{lemma}\label{lemma:strichartz estimate with noncertered frequency localization}
Let $M$ be a dyadic value and let $Q$ be a(possibly) noncentered $M$-cube
in Fourier space
$$Q=\lr{\xi_{0}+\eta:|\eta|<M}.$$
Let $P_{Q}$ be the corresponding Littlewood-Paley projection, then by the
Galilean invariance, we have
\begin{align}\label{equ:strichartz estimate with noncertered frequency localization}
\n{P_{Q}u}_{L_{t,x}^{p}}\lesssim M^{\frac{d}{2}-\frac{d+2}{p}}
\n{P_{Q}u}_{Y^{0}([0,T))}
\end{align}
for $p>\frac{2(d+2)}{d}.$
\end{lemma}

\begin{lemma}[Bernstein with noncentered frequency projection]\label{lemma:bernstein with noncentered frequency projection}
Let $M$ and $Q$ be as in Lemma $\ref{lemma:strichartz estimate with noncertered frequency localization}$, then for $1\leq p\leq q\leq \wq$
\begin{align}\label{equ:bernstein with noncentered frequency projection}
\n{P_{Q}f}_{L_{x}^{q}}\lesssim M^{\frac{d}{p}-\frac{d}{q}}\n{P_{Q}f}_{L_{x}^{p}}.
\end{align}
\end{lemma}
Lemmas \ref{lemma:strichartz estimate with noncertered frequency localization} and \ref{lemma:bernstein with noncentered frequency projection} are very well-known, and are available in many references, for example, see \cite{chen2020unconditional}.

%For the proof of Lemma $\ref{lemma:strichartz estimate with noncertered frequency localization}$ and $\ref{lemma:bernstein with noncentered frequency projection}$, see \cite[Corollary 5.18 and Lemma 5.19]{chen2019the}.

\subsection{Trilinear Estimates}\label{subsection:Trilinear Estimates}

%\begin{lemma}[Improved bilinear Strichartz estimate]\label{lemma:improved bilinear strichartz estimate}
%Fix $d\geq 3$ and $T\leq 1.$ Then there exists $\delta>0$ such that for every
%$1\leq M_{2}\leq M_{1}$ we have
%\begin{align}
%\n{u_{M_{1}}u_{M_{2}}}_{L_{t,x}^{2}([0,T)\times \T^{d})}\lesssim
%M_{2}^{\frac{d-2}{2}}\lrs{\frac{M_{2}}{M_{1}}+\frac{1}{M_{2}}}^{\delta}
%\n{u_{M_{1}}}_{Y^{0}([0,T))}\n{u_{M_{2}}}_{Y^{0}([0,T))}
%\end{align}
%
%\end{lemma}
To deal with the cubic energy-supercritical NLS, we prove the following $U$-$V$ trilinear estimates at critical regularity. Let $\wt{u}\in \lr{u,\ol{u}}$.

\begin{lemma}\label{lemma:trilinear estimate d>=4}
On $\T^{d}$ with $d\geq 4$ and $s\in\lr{\frac{d-6}{2},\frac{d-2}{2}}$, we have the high frequency estimate
\begin{align}\label{equ:trilinear estimate d>=5 high frequency estimate}
\iint_{x,t}\wt{u}_{1}(t,x)\wt{u}_{2}(t,x)\wt{u}_{3}(t,x)\wt{g}(t,x)dxdt\lesssim \n{u_{1}}_{Y^{s}}\n{u_{2}}_{Y^{\frac{d-2}{2}}}\n{u_{3}}_{Y^{\frac{d-2}{2}}}
\n{g}_{Y^{-s}},
\end{align}
and the low frequency estimate
\begin{align}\label{equ:trilinear estimate d>=5 low frequency estimate}
&\iint_{x,t}\wt{u}_{1}(t,x)(P_{\leq M_{0}}\wt{u}_{2})(t,x)\wt{u}_{3}(t,x)\wt{g}(t,x)dxdt\\
\lesssim& T^{\frac{1}{d+3}}M_{0}^{\frac{2(d+2)}{3(d+3)}} \n{u_{1}}_{Y^{s}}\n{P_{\leq M_{0}}u_{2}}_{Y^{\frac{d-2}{2}}}\n{u_{3}}_{Y^{\frac{d-2}{2}}}
\n{g}_{Y^{-s}},\notag
\end{align}
for all $T\leq 1$ and all frequencies $M_{0}\geq 1.$ Then by Lemma $\ref{lemma:U-V estimate,dual argument}$, $(\ref{equ:trilinear estimate d>=5 high frequency estimate})$ and $(\ref{equ:trilinear estimate d>=5 low frequency estimate})$, we have
\begin{align} \label{equ:trilinear estimate d>=5 high frequency estimate,dual argument}
\bbn{\int_{a}^{t}e^{i(t-\tau)\Delta}(\wt{u}_{1}\wt{u}_{2}\wt{u}_{3})d\tau}_{X^{s}}\lesssim \n{u_{1}}_{Y^{s}}
\n{u_{2}}_{Y^{\frac{d-2}{2}}}\n{u_{3}}_{Y^{\frac{d-2}{2}}}
\end{align}
and
\begin{align} \label{equ:trilinear estimate d>=5 low frequency estimate,dual argument}
&\bbn{\int_{a}^{t}e^{i(t-\tau)\Delta}(\wt{u}_{1}\wt{u}_{2}\wt{u}_{3})d\tau}_{X^{s}}\\
\lesssim& \n{u_{1}}_{Y^{s}}
\lrs{T^{\frac{1}{d+3}}M_{0}^{\frac{2(d+2)}{3(d+3)}}\n{P_{\leq M_{0}}u_{2}}_{Y^{\frac{d-2}{2}}}+\n{P_{>M_{0}}u_{2}}_{Y^{\frac{d-2}{2}}}}\n{u_{3}}_{Y^{\frac{d-2}{2}}}.\notag
\end{align}

\end{lemma}

%By the symmetry of $u_{1}$ and $g$, we have $X^{\frac{6-d}{2}}$ estimate.
%\begin{lemma}\label{lemma:trilinear estimate d>=5, second}
%On $\T^{d}$ with $d\geq 4$, we have the high frequency estimate
%\begin{align}\label{equ:trilinear estimate d>=5 high frequency estimate, second}
%\iint_{x,t}u_{1}(t,x)u_{2}(t,x)u_{3}(t,x)g(t,x)dxdt\lesssim \n{u_{1}}_{Y^{\frac{6-d}{2}}}\n{u_{2}}_{Y^{\frac{d-2}{2}}}\n{u_{3}}_{Y^{\frac{d-2}{2}}}
%\n{g}_{Y^{-\frac{6-d}{2}}},
%\end{align}
%and the low frequency estimate
%\begin{align}\label{equ:trilinear estimate d>=5 low frequency estimate, second}
%\iint_{x,t}u_{1}(t,x)(P_{\leq M_{0}}u_{2})(t,x)u_{3}(t,x)g(t,x)dxdt\lesssim T^{\frac{1}{d+3}}M_{0}^{\frac{d+2}{d+3}} \n{u_{1}}_{Y^{\frac{6-d}{2}}}\n{P_{\leq M_{0}}u_{2}}_{Y^{\frac{d-2}{2}}}\n{u_{3}}_{Y^{\frac{d-2}{2}}}
%\n{g}_{Y^{-\frac{6-d}{2}}},
%\end{align}
%for all $T\leq 1$ and all frequencies $M_{0}\geq 1.$ Or
%\begin{align}
%\bbn{\int_{a}^{t}e^{i(t-s)\Delta}(u_{1}u_{2}u_{3})ds}_{X^{\frac{6-d}{2}}}\lesssim \n{u_{1}}_{Y^{\frac{6-d}{2}}}
%\lrs{T^{\frac{1}{d+3}}M_{0}^{\frac{d+2}{d+3}}\n{P_{\leq M_{0}}u_{2}}_{Y^{\frac{d-2}{2}}}+\n{P_{>M_{0}}u_{2}}_{Y^{\frac{d-2}{2}}}}\n{u_{3}}_{Y^{\frac{d-2}{2}}}
%\end{align}
%and
%\begin{align}
%\bbn{\int_{a}^{t}e^{i(t-s)\Delta}(u_{1}u_{2}u_{3})ds}_{X^{\frac{6-d}{2}}}\lesssim \n{u_{1}}_{Y^{\frac{6-d}{2}}}
%\n{u_{2}}_{Y^{\frac{d-2}{2}}}\n{u_{3}}_{Y^{\frac{d-2}{2}}}.
%\end{align}
%\end{lemma}
\begin{proof}

%By the symmetry of $u_{1}$ and $g$, we might as well assume $s\in[0,\frac{d-2}{2}]$.
%By the symmetry of $u_{1}$ and $g$, we might as well consider the case $s\geq 0$.

 It suffices to prove high and low frequency estimates $(\ref{equ:trilinear estimate d>=5 high frequency estimate})$ and $(\ref{equ:trilinear estimate d>=5 low frequency estimate})$.
For simplicity, we take $\wt{u}=u$ and $\wt{g}=g$.

 For the high frequency estimate $(\ref{equ:trilinear estimate d>=5 high frequency estimate})$, decompose the 4 factors into Littlewood-Paley pieces so that
\begin{align*}
I=\sum_{M_{1},M_{2},M_{3},M_{4}}I_{M_{1},M_{2},M_{3},M_{4}}
\end{align*}
where
\begin{align*}
I_{M_{1},M_{2},M_{3},M_{4}}=\iint_{x,t}u_{1,M_{1}}u_{2,M_{2}}u_{3,M_{3}}g_{M_{4}}dxdt
\end{align*}
with $u_{j,M_{j}}=P_{M_{j}}u_{j}$ and $g_{M_{4}}=P_{M_{4}}g$. By orthogonality, we know that these cases are as follows
$$M_{\sigma(1)}\sim M_{\sigma(2)}\geq  M_{\sigma(3)}\geq M_{\sigma(4)}$$
where $\sigma$ is a permutation on $\lr{1,2,3,4}$. By symmetry, we might as well assume without loss that $M_{2}\geq M_{3}$.

First, we consider the most difficulty case, namely,
Case A. $M_{1}\sim M_{4}\geq M_{2}\geq M_{3}$.
Then, we need only to deal with one such as
Case B. $M_{1}\sim M_{2}\geq  M_{4}\geq M_{3}$,
since other cases can be treated in the same way.

Let $I_{A}$ denote the integral restricted to the Case A. Decompose the $M_{1}$ and $M_{4}$ dyadic spaces into $M_{2}$ size cubes. Due to the frequency constraint $\xi_{2}\sim -(\xi_{1}+\xi_{3}+\xi_{4})$, for each choice $Q$ of an $M_{2}$ size cube within the $\xi_{1}$ space, the variable $\xi_{2}$
is constrained to at most $3^{d}$ of $M_{2}$ size cubes. For
convenience, we denote these cubes by a single cube $Q_{c}$ that
corresponds to $Q$. Then
\begin{align*}
I_{A}\lesssim& \sum_{\substack{M_{1},M_{2},M_{3},M_{4}\\M_{1}\sim M_{4}\geq  M_{2}\geq M_{3}}} \sum_{Q}\n{P_{Q}u_{1,M_{1}}u_{2,M_{2}}
u_{3,M_{3}}P_{Q_{c}}g_{M_{4}}}_{L_{t,x}^{1}}\\
\lesssim& \sum_{\substack{M_{1},M_{2},M_{3},M_{4}\\ M_{1} \sim M_{4}\geq  M_{2}\geq M_{3}}}\sum_{Q}\n{P_{Q}u_{1,M_{1}}}_{L_{t,x}^{\frac{3(d+3)}{d+2}}}
\n{u_{2,M_{2}}}_{L_{t,x}^{\frac{3(d+3)}{d+2}}}
\n{u_{3,M_{3}}}_{L_{t,x}^{d+3}}\n{P_{Q_{c}}g_{M_{4}}}_{L_{t,x}^{\frac{3(d+3)}{d+2}}}
\end{align*}
where the factor corresponding to the smallest size cubes (here $M_{3}$ size cubes) is put in $L_{t,x}^{d+3}$ and the others are put in $L_{t,x}^{\frac{3(d+3)}{d+2}}$. By $(\ref{equ:Strichartz estimate})$ and $(\ref{equ:strichartz estimate with noncertered frequency localization})$,
%\begin{align*}
%\lesssim& \sum_{\substack{M_{1},M_{2},M_{3},M_{4}\\ M_{1}\sim M_{4} \geq M_{2}\geq
%M_{3}}}\sum_{Q}M_{2}^{\frac{d^{2}+d-8}{2(d+3)}}\n{P_{Q}u_{1,M_{1}}}_{Y^{0}}
%\n{P_{Q_{c}}u_{2,M_{2}}}_{Y^{0}}
%M_{3}^{\frac{d^{2}+d-4}{2(d+3)}}\n{u_{3,M_{3}}}_{Y^{0}}\n{g_{M_{4}}}_{Y^{0}}\\
%\end{align*}
\begin{align*}
\lesssim& \sum_{\substack{M_{1},M_{2},M_{3},M_{4}\\ M_{1}\sim M_{4} \geq M_{2}\geq
M_{3}}}\sum_{Q}M_{2}^{\frac{d-2}{2}-\frac{1}{d+3}}\n{P_{Q}u_{1,M_{1}}}_{Y^{0}}
\n{u_{2,M_{2}}}_{Y^{0}}
M_{3}^{\frac{d-2}{2}+\frac{1}{d+3}}\n{u_{3,M_{3}}}_{Y^{0}}\n{P_{Q_{c}}g_{M_{4}}}_{Y^{0}}
\end{align*}
Applying Cauchy-Schwarz to sum in $Q$,
\begin{align}\label{equ:estiamte, wanted form}
\lesssim& \sum_{\substack{M_{1},M_{4}\\M_{1}\sim M_{4}}}M_{1}^{-s}M_{4}^{s}\n{u_{1,M_{1}}}_{Y^{s}}
\n{g_{M_{4}}}_{Y^{-s}}
\sum_{\substack{M_{2},M_{3}\\M_{2}\geq M_{3}}}M_{2}^{-\frac{1}{d+3}}M_{3}^{\frac{1}{d+3}}
\n{u_{2,M_{2}}}_{Y^{\frac{d-2}{2}}}\n{u_{3,M_{3}}}_{Y^{\frac{d-2}{2}}}
\end{align}
Applying Cauchy-Schwarz,
\begin{align*}
\lesssim&  \lrs{\sum_{M_{1}}\n{u_{1,M_{1}}}_{Y^{s}}^{2}}^{\frac{1}{2}}
\lrs{\sum_{M_{4}}\n{g_{4,M_{4}}}_{Y^{-s}}^{2}}^{\frac{1}{2}}\\
&\lrs{\sum_{\substack{M_{2},M_{3}\\M_{2}\geq M_{3}}}\lrs{\frac{M_{3}}{M_{2}}}^{\frac{1}{d+3}}\n{u_{2,M_{2}}}_{Y^{\frac{d-2}{2}}}^{2}}^{\frac{1}{2}}
\lrs{\sum_{\substack{M_{2},M_{3}\\M_{2}\geq M_{3}}}\lrs{\frac{M_{3}}{M_{2}}}^{\frac{1}{d+3}}\n{u_{3,M_{3}}}_{Y^{\frac{d-2}{2}}}^{2}}^{\frac{1}{2}}\\
\lesssim &\n{u_{1}}_{Y^{s}}\n{u_{2}}_{Y^{\frac{d-2}{2}}}\n{u_{3}}_{Y^{\frac{d-2}{2}}}
\n{g}_{Y^{-s}}.
\end{align*}

Case B. $M_{1}\sim M_{2}\geq  M_{4}\geq M_{3}$. Decompose the $M_{1}$ and $M_{2}$ dyadic spaces into $M_{4}$ size cubes and we have
\begin{align*}
I_{B}\lesssim& \sum_{\substack{M_{1},M_{2},M_{3},M_{4}\\M_{1}\sim M_{2}\geq  M_{4}\geq M_{3}}} \sum_{Q}\n{P_{Q}u_{1,M_{1}}P_{Q_{c}}u_{2,M_{2}}
u_{3,M_{3}}g_{M_{4}}}_{L_{t,x}^{1}}\\
\lesssim& \sum_{\substack{M_{1},M_{2},M_{3},M_{4}\\ M_{1}\sim M_{2}\geq  M_{4}\geq M_{3}}}\sum_{Q}\n{P_{Q}u_{1,M_{1}}}_{L_{t,x}^{\frac{3(d+3)}{d+2}}}
\n{P_{Q_{c}}u_{2,M_{2}}}_{L_{t,x}^{\frac{3(d+3)}{d+2}}}
\n{u_{3,M_{3}}}_{L_{t,x}^{d+3}}\n{g_{M_{4}}}_{L_{t,x}^{\frac{3(d+3)}{d+2}}}
\end{align*}
By $(\ref{equ:Strichartz estimate})$ and $(\ref{equ:strichartz estimate with noncertered frequency localization})$,
\begin{align*}
\lesssim& \sum_{\substack{M_{1},M_{2},M_{3},M_{4}\\ M_{1}\sim M_{2}\geq  M_{4}\geq M_{3}}}\sum_{Q}M_{4}^{\frac{d-2}{2}-\frac{1}{d+3}}\n{P_{Q}u_{1,M_{1}}}_{Y^{0}}
\n{P_{Q_{c}}u_{2,M_{2}}}_{Y^{0}}
M_{3}^{\frac{d-2}{2}+\frac{1}{d+3}}\n{u_{3,M_{3}}}_{Y^{0}}\n{g_{M_{4}}}_{Y^{0}}
\end{align*}
Applying Cauchy-Schwarz to sum in $Q$,
\begin{align*}
\lesssim& \sum_{\substack{M_{1},M_{2}\\M_{1}\sim M_{2}}}\left(M_{1}^{-s}M_{2}^{-\frac{d-2}{2}}\n{u_{1,M_{1}}}_{Y^{s}}\n{u_{2,M_{2}}}_{Y^{\frac{d-2}{2}}}
\right.\\
&\left.\sum_{\substack{M_{3},M_{4}\\M_{1}\sim M_{2}\geq  M_{4}\geq M_{3}}}M_{4}^{\frac{d-2}{2}-\frac{1}{d+3}+s}M_{3}^{\frac{1}{d+3}}
\n{u_{3,M_{3}}}_{Y^{\frac{d-2}{2}}}\n{g_{4,M_{4}}}_{Y^{-s}}\right)
\end{align*}
%By the fact that $\frac{d-2}{2}+s\geq 0$,
%\begin{align*}
%=& \sum_{\substack{M_{1},M_{2}\\M_{1}\sim M_{2}}}\left(M_{1}^{-s}M_{2}^{-\frac{d-2}{2}}\n{u_{1,M_{1}}}_{Y^{s}}
%\n{u_{2,M_{2}}}_{Y^{\frac{d-2}{2}}}\right.\\
%&\left.M_{2}^{\frac{d-2}{2}+s}\sum_{\substack{M_{3},M_{4}\\M_{1}\sim M_{2}\geq M_{4}\geq M_{3}}}\lrs{\frac{M_{4}}{M_{2}}}^{\frac{d-2}{2}+s}M_{4}^{-\frac{1}{d+3}}M_{3}^{\frac{1}{d+3}}
%\n{u_{3,M_{3}}}_{Y^{\frac{d-2}{2}}}\n{g_{M_{4}}}_{Y^{-s}}\right)\\
%\leq& \sum_{\substack{M_{1},M_{2}\\M_{1}\sim M_{2}}}M_{1}^{-s}M_{2}^{s}\n{u_{1,M_{1}}}_{Y^{s}}
%\n{u_{2,M_{2}}}_{Y^{\frac{d-2}{2}}}\sum_{\substack{M_{3},M_{4}\\M_{1}\sim M_{2}\geq M_{4}\geq M_{3}}}M_{4}^{-\frac{1}{d+3}}M_{3}^{\frac{1}{d+3}}
%\n{u_{3,M_{3}}}_{Y^{\frac{d-2}{2}}}\n{g_{M_{4}}}_{Y^{-s}}
%\end{align*}
%Applying Cauchy-Schwarz,
%\begin{align*}
%\lesssim &\n{u_{1}}_{Y^{s}}\n{u_{2}}_{Y^{\frac{d-2}{2}}}\n{u_{3}}_{Y^{\frac{d-2}{2}}}
%\n{g}_{Y^{-s}}.
%\end{align*}
If $s+\frac{d-2}{2}=0$, it can be estimated in the same way as $(\ref{equ:estiamte, wanted form})$. Thus we need only to treat the case $s+\frac{d-2}{2}>0$. Supping out $\n{u_{3,M_{3}}}_{Y^{\frac{d-2}{2}}}$ and $\n{g_{M_{4}}}_{Y^{-s}}$ in $M_{3}$ and $M_{4}$, we have
\begin{align*}
\lesssim& \n{u_{3}}_{Y^{\frac{d-2}{2}}}\n{g}_{Y^{-s}}\sum_{\substack{M_{1},M_{2}\\M_{1}\sim M_{2}}}M_{1}^{-s}M_{2}^{-\frac{d-2}{2}}\n{u_{1,M_{1}}}_{Y^{s}}\n{u_{2,M_{2}}}_{Y^{\frac{d-2}{2}}}
\sum_{\substack{M_{3},M_{4}\\M_{1}\sim M_{2}\geq  M_{4}\geq M_{3}}}M_{4}^{\frac{d-2}{2}-\frac{1}{d+3}+s}M_{3}^{\frac{1}{d+3}}
\end{align*}
By the fact that $\frac{d-2}{2}+s>0$,
\begin{align*}
\lesssim& \n{u_{3}}_{Y^{\frac{d-2}{2}}}\n{g}_{Y^{-s}}\sum_{\substack{M_{1},M_{2}\\M_{1}\sim M_{2}}}M_{1}^{-s}M_{2}^{-\frac{d-2}{2}}\n{u_{1,M_{1}}}_{Y^{s}}\n{u_{2,M_{2}}}_{Y^{\frac{d-2}{2}}}
M_{2}^{\frac{d-2}{2}+s}\\
=&\n{u_{3}}_{Y^{\frac{d-2}{2}}}\n{g}_{Y^{-s}}\sum_{\substack{M_{1},M_{2}\\M_{1}\sim M_{2}}}M_{1}^{-s}M_{2}^{s}\n{u_{1,M_{1}}}_{Y^{-s}}\n{u_{2,M_{2}}}_{Y^{\frac{d-2}{2}}}
\end{align*}
Applying Cauchy-Schwarz,
\begin{align*}
\lesssim &\n{u_{1}}_{Y^{s}}\n{u_{2}}_{Y^{\frac{d-2}{2}}}\n{u_{3}}_{Y^{\frac{d-2}{2}}}
\n{g}_{Y^{-s}}.
\end{align*}

Case B requires that $\frac{d-2}{2}+s\geq 0$. If we exchange $M_{1}$ and $M_{4}$ in Case B, we find another requirement that $\frac{d-2}{2}-s\geq 0$ is also needed. In this case, it becomes Case A again if $s=\frac{d-2}{2}$ and it becomes similar but a bit
different if $s=\frac{d-6}{2}$ as $M_{1}$ and $M_{2}$ are not symmetric.

%For any case $M_{\sigma(1)}\sim M_{\sigma(2)}\geq M_{\sigma(3)}\geq M_{\sigma(4)}$, following the above steps, we decompose the $M_{\sigma(1)}$ and $M_{\sigma(2)}$ dyadic spaces into $M_{\sigma(3)}$ size cubes. Then, we apply H\"{o}lder with $d+3$ at the $M_{\sigma(4)}$ term while others are put with $\frac{3(d+3)}{d+2}$. Since we assume $s\leq \frac{d-2}{2}$, we arrive at

Proof of the low frequency estimate $(\ref{equ:trilinear estimate d>=5 low frequency estimate})$.
We first deal with the most difficult Case A. $M_{1}\sim M_{4}\geq M_{3}\geq M_{2}$.
Decompose the $M_{1}$ and $M_{4}$ dyadic spaces into $M_{3}$ size cubes, we have
\begin{align*}
I_{M_{1},M_{2},M_{3},M_{4}}\lesssim &\sum_{Q}\n{P_{Q}u_{1,M_{1}}(P_{\leq M_{0}}u_{2,M_{2}})
u_{3,M_{3}}P_{Q_{c}}g_{M_{4}}}_{L_{t,x}^{1}}\\
\leq& \sum_{Q}\n{P_{Q}u_{1,M_{1}}}_{L_{t,x}^{\frac{3(d+3)}{d+2}}}
\n{P_{\leq M_{0}}u_{2,M_{2}}}_{L_{t,x}^{d+3}}
\n{u_{3,M_{3}}}_{L_{t,x}^{\frac{3(d+3)}{d+2}}}
\n{P_{Q_{c}}g_{M_{4}}}_{L_{t,x}^{\frac{3(d+3)}{d+2}}}
\end{align*}
where the factor corresponding to the smallest size cubes (here $M_{2}$ size cubes) is put in $L_{t,x}^{d+3}$ and the others are put in $L_{t,x}^{\frac{3(d+3)}{d+2}}$.

By H\"{o}lder, Bernstein inequalities and $(\ref{equ:u-v and sobolev})$,
\begin{align*}
\n{P_{\leq M_{0}}u_{2,M_{2}}}_{L_{t,x}^{d+3}}\lesssim & T^{\frac{1}{d+3}}M_{0}^{\frac{2}{d+3}}M_{2}^{\frac{d-2}{2}+\frac{1}{d+3}}\n{P_{\leq M_{0}}u_{2,M_{2}}}_{L_{t}^{\wq}L_{x}^{2}}\\
\lesssim &T^{\frac{1}{d+3}}M_{0}^{\frac{2}{d+3}}M_{2}^{\frac{d-2}{2}+\frac{1}{d+3}}\n{P_{\leq M_{0}}u_{2,M_{2}}}_{Y^{0}}
\end{align*}
By $(\ref{equ:Strichartz estimate})$ and $(\ref{equ:strichartz estimate with noncertered frequency localization})$,
\begin{align*}
&I_{M_{1},M_{2},M_{3},M_{4}}\\
\lesssim& T^{\frac{1}{d+3}}M_{0}^{\frac{2}{d+3}}
\sum_{Q}M_{2}^{\frac{d-2}{2}+\frac{1}{d+3}}M_{3}^{\frac{d-2}{2}-\frac{1}{d+3}}
\n{P_{Q}u_{1,M_{1}}}_{Y^{0}}
\n{P_{\leq M_{0}}u_{2,M_{2}}}_{Y^{0}}
\n{u_{3,M_{3}}}_{Y^{0}}
\n{P_{Q_{c}}g_{M_{4}}}_{Y^{0}}
\end{align*}
Applying Cauchy-Schwarz to sum in $Q$, we arrive at
\begin{align*}
I_{M_{1},M_{2},M_{3},M_{4}}
\lesssim& T^{\frac{1}{d+3}}M_{0}^{\frac{2}{d+3}}M_{3}^{\frac{d-2}{2}-\frac{1}{d+3}}M_{2}^{\frac{d-2}{2}
+\frac{1}{d+3}}
\n{u_{1,M_{1}}}_{Y^{0}}\n{P_{\leq M_{0}}u_{2,M_{2}}}_{Y^{0}}
\n{u_{3,M_{3}}}_{Y^{0}}\n{g_{M_{4}}}_{Y^{0}}.
\end{align*}
Then by $M_{1}\sim M_{4}$, we obtain
\begin{align*}
I_{M_{1},M_{2},M_{3},M_{4}}
\lesssim&
T^{\frac{1}{d+3}}M_{0}^{\frac{2}{d+3}}M_{3}^{\frac{d-2}{2}-\frac{1}{d+3}}M_{2}^{\frac{d-2}{2}
+\frac{1}{d+3}}
M_{1}^{-s}\n{u_{1,M_{1}}}_{Y^{s}}
M_{2}^{-\frac{d-2}{2}}\n{P_{\leq M_{0}}u_{2,M_{2}}}_{Y^{\frac{d-2}{2}}}\\
&M_{3}^{-\frac{d-2}{2}}\n{u_{3,M_{3}}}_{Y^{\frac{d-2}{2}}}
M_{4}^{s}\n{g_{M_{4}}}_{Y^{-s}}\\
\lesssim&  T^{\frac{1}{d+3}}M_{0}^{\frac{2}{d+3}}M_{3}^{-\frac{1}{d+3}}M_{2}^{\frac{1}{d+3}}
\n{u_{1,M_{1}}}_{Y^{s}}
\n{P_{\leq M_{0}}u_{2,M_{2}}}_{Y^{\frac{d-2}{2}}}
\n{u_{3,M_{3}}}_{Y^{\frac{d-2}{2}}}\n{g_{M_{4}}}_{Y^{-s}}.
\end{align*}
Thus, we have
\begin{align*}
I_{A}\lesssim&  T^{\frac{1}{d+3}}M_{0}^{\frac{2}{d+3}}
\sum_{\substack{M_{1},M_{2},M_{3},M_{4}\\M_{1}\sim M_{4}\geq M_{3}\geq M_{2}}}\n{u_{1,M_{1}}}_{Y^{s}}
\n{g_{M_{4}}}_{Y^{-s}}M_{3}^{-\frac{1}{d+3}}M_{2}^{\frac{1}{d+3}}
\n{P_{\leq M_{0}}u_{2}}_{Y^{\frac{d-2}{2}}}
\n{u_{3,M_{3}}}_{Y^{\frac{d-2}{2}}}\\
\lesssim&
T^{\frac{1}{d+3}}M_{0}^{\frac{2}{d+3}}
\n{u_{1}}_{Y^{s}}\n{P_{\leq M_{0}}u_{2}}_{Y^{\frac{d-2}{2}}}\n{u_{3}}_{Y^{\frac{d-2}{2}}}
\n{g}_{Y^{-s}}
\end{align*}

Next, we deal with one such as Case B. $M_{1}\sim M_{2}\geq M_{4}\geq M_{3}$,
since other cases can be treated in the same way. Decompose the $M_{1}$ and $M_{2}$ dyadic spaces into $M_{4}$ size cubes and we have
\begin{align*}
I_{M_{1},M_{2},M_{3},M_{4}}\leq \sum_{Q}\n{P_{Q}u_{1,M_{1}}}_{L_{t,x}^{\frac{3(d+3)}{d+2}}}
\n{P_{Q_{c}}P_{\leq M_{0}}u_{2,M_{2}}}_{L_{t,x}^{\frac{3(d+3)}{d+2}}}
\n{u_{3,M_{3}}}_{L_{t,x}^{d+3}}
\n{g_{M_{4}}}_{L_{t,x}^{\frac{3(d+3)}{d+2}}}
\end{align*}
By H\"{o}lder, Bernstein inequalities and $(\ref{equ:u-v and sobolev})$,
\begin{align*}
\n{P_{Q_{c}}P_{\leq M_{0}}u_{2,M_{2}}}_{L_{t,x}^{\frac{3(d+3)}{d+2}}}\lesssim& T^{\frac{d+2}{3(d+3)}}M_{0}^{\frac{2(d +2)}{3(d+3)}}M_{4}^{\frac{d-2}{6}-\frac{1}{3(d+3)}}\n{P_{Q_{c}}P_{\leq M_{0}}u_{2,M_{2}}}_{L_{t}^{\wq}L_{x}^{2}}\\
\lesssim&T^{\frac{d+2}{3(d+3)}}M_{0}^{\frac{2(d +2)}{3(d+3)}}M_{4}^{\frac{d-2}{6}-\frac{1}{3(d+3)}}\n{P_{Q_{c}}P_{\leq M_{0}}u_{2,M_{2}}}_{Y^{0}}
\end{align*}
By $(\ref{equ:Strichartz estimate})$ and $(\ref{equ:strichartz estimate with noncertered frequency localization})$,
\begin{align*}
I_{M_{1},M_{2},M_{3},M_{4}}\lesssim &T^{\frac{d+2}{3(d+3)}}M_{0}^{\frac{2(d+2)}{3(d+3)}}
M_{4}^{\frac{d-2}{2}-\frac{1}{d+3}}
M_{3}^{\frac{d-2}{2}+\frac{1}{d+3}}\\
&\sum_{Q}\n{P_{Q}u_{1,M_{1}}}_{Y^{0}}
\n{P_{Q_{c}}P_{\leq M_{0}}u_{2,M_{2}}}_{Y^{0}}
\n{u_{3,M_{3}}}_{Y^{0}}
\n{g_{M_{4}}}_{Y^{0}}
\end{align*}
Applying Cauchy-Schwarz to sum in $Q$,
we arrive at
\begin{align*}
&I_{M_{1},M_{2},M_{3},M_{4}}\\
\lesssim&
T^{\frac{d+2}{3(d+3)}}M_{0}^{\frac{2(d+2)}{3(d+3)}}
M_{4}^{\frac{d-2}{2}-\frac{1}{d+3}}
M_{3}^{\frac{d-2}{2}+\frac{1}{d+3}}
\n{u_{1,M_{1}}}_{Y^{0}}\n{P_{\leq M_{0}}u_{2,M_{2}}}_{Y^{0}}
\n{u_{3,M_{3}}}_{Y^{0}}\n{g_{M_{4}}}_{Y^{0}}\\
\lesssim &T^{\frac{d+2}{3(d+3)}}M_{0}^{\frac{2(d+2)}{3(d+3)}}
M_{1}^{-s}\n{u_{1,M_{1}}}_{Y^{s}}
M_{2}^{-\frac{d-2}{2}}\n{P_{\leq M_{0}}u_{2,M_{2}}}_{Y^{\frac{d-2}{2}}}\\
&M_{3}^{\frac{1}{d+3}}\n{u_{3,M_{3}}}_{Y^{\frac{d-2}{2}}}
M_{4}^{\frac{d-2}{2}-\frac{1}{d+3}+s}\n{g_{M_{4}}}_{Y^{-s}}
\end{align*}
If $s+\frac{d-2}{2}=0$, it can be estimated in the same way as $(\ref{equ:estiamte, wanted form})$. Thus we need only to treat the case $s+\frac{d-2}{2}>0$. Supping out $\n{u_{3,M_{3}}}_{Y^{\frac{d-2}{2}}}$ and $\n{g_{M_{4}}}_{Y^{-s}}$ in $M_{3}$ and $M_{4}$, we have
\begin{align*}
I_{B}\lesssim & \sum_{\substack{M_{1},M_{2},M_{3},M_{4}\\M_{1}\sim M_{2}\geq M_{4}\geq M_{3}}}
I_{M_{1},M_{2},M_{3},M_{4}}\\
\lesssim &T^{\frac{d+2}{3(d+3)}}M_{0}^{\frac{2(d+2)}{3(d+3)}}\n{u_{3}}_{Y^{\frac{d-2}{2}}}\n{g}_{Y^{-s}}
\sum_{\substack{M_{1},M_{2},M_{3},M_{4}\\M_{1}\sim M_{2}\geq M_{4}\geq M_{3}}}
\left(M_{1}^{-s}\n{u_{1,M_{1}}}_{Y^{s}}
M_{2}^{-\frac{d-2}{2}}\n{P_{\leq M_{0}}u_{2,M_{2}}}_{Y^{\frac{d-2}{2}}}\right.\\
&\left. M_{3}^{\frac{1}{d+3}}
M_{4}^{\frac{d-2}{2}-\frac{1}{d+3}+s}\right)\\
\lesssim & T^{\frac{d+2}{3(d+3)}}M_{0}^{\frac{2(d+2)}{3(d+3)}}\n{u_{3}}_{Y^{\frac{d-2}{2}}}\n{g}_{Y^{-s}}
\sum_{\substack{M_{1},M_{2}\\M_{1}\sim M_{2}}}
M_{1}^{-s}\n{u_{1,M_{1}}}_{Y^{s}}
M_{2}^{s}\n{P_{\leq M_{0}}u_{2,M_{2}}}_{Y^{\frac{d-2}{2}}}
\end{align*}
Applying Cauchy-Schwarz,
\begin{align*}
\lesssim &T^{\frac{d+2}{3(d+3)}}M_{0}^{\frac{2(d+2)}{3(d+3)}}\n{u_{1}}_{Y^{s}}\n{P_{\leq M_{0}}u_{2}}_{Y^{\frac{d-2}{2}}}\n{u_{3}}_{Y^{\frac{d-2}{2}}}
\n{g}_{Y^{-s}}.
\end{align*}

%Notice that we get a $T^{\frac{1}{d+3}}M_{0}^{\frac{2}{d+3}}$ in Case A and a $T^{\frac{d+2}{3(d+3)}}M_{0}^{\frac{2(d+2)}{3(d+3)}}$ in Case B. There is no need to pursue the best power, as we just need $(\ref{equ:trilinear estimate d>=5 low frequency estimate})$ to
%hold with some powers of $T$ and $M_{0}$.
\end{proof}

\subsection{Quintilinear Estimates} \label{subsection:Quintilinear Estimates}

\begin{lemma}\label{lemma:multilinear estimate d>=5}
On $\T^{d}$ with $d\geq 3$ and $s\in\lr{\frac{d-5}{2},\frac{d-1}{2}}$, we have the high frequency estimate
\begin{align}\label{equ:multilinear estimate d>=5 high frequency estimate}
&\iint_{x,t}\wt{u}_{1}(t,x)\wt{u}_{2}(t,x)\wt{u}_{3}(t,x)\wt{u}_{4}(t,x)\wt{u}_{5}(t,x)\wt{g}(t,x)dxdt\\
\lesssim& \n{u_{1}}_{Y^{s}}\n{u_{2}}_{Y^{\frac{d-1}{2}}}\n{u_{3}}_{Y^{\frac{d-1}{2}}}
\n{u_{4}}_{Y^{\frac{d-1}{2}}}\n{u_{5}}_{Y^{\frac{d-1}{2}}}
\n{g}_{Y^{-s}},\notag
\end{align}
and the low frequency estimate
\begin{align}\label{equ:multilinear estimate d>=5 low frequency estimate}
&\iint_{x,t}\wt{u}_{1}(t,x)(P_{\leq M_{0}}\wt{u}_{2})(t,x)\wt{u}_{3}(t,x)\wt{u}_{4}(t,x)\wt{u}_{5}(t,x)\wt{g}(t,x)dxdt\\
\lesssim& T^{\frac{1}{2(d+3)}}M_{0}^{\frac{2d+3}{3(d+3)}} \n{u_{1}}_{Y^{s}}\n{P_{\leq M_{0}}u_{2}}_{Y^{\frac{d-1}{2}}}\n{u_{3}}_{Y^{\frac{d-1}{2}}}
\n{u_{4}}_{Y^{\frac{d-1}{2}}}\n{u_{5}}_{Y^{\frac{d-1}{2}}}
\n{g}_{Y^{-s}},\notag
\end{align}
for all $T\leq 1$ and all frequencies $M_{0}\geq 1.$ Then by Lemma $\ref{lemma:U-V estimate,dual argument}$, $(\ref{equ:multilinear estimate d>=5 high frequency estimate})$ and $(\ref{equ:multilinear estimate d>=5 low frequency estimate})$, we have
\begin{align}
&\bbn{\int_{a}^{t}e^{i(t-\tau)\Delta}(\wt{u}_{1}\wt{u}_{2}\wt{u}_{3}\wt{u}_{4}\wt{u}_{5})d\tau}_{X^{s}}
\lesssim \n{u_{1}}_{Y^{s}}
\n{u_{2}}_{Y^{\frac{d-1}{2}}}\n{u_{3}}_{Y^{\frac{d-1}{2}}}
\n{u_{4}}_{Y^{\frac{d-1}{2}}}\n{u_{5}}_{Y^{\frac{d-1}{2}}}
\end{align}
and
\begin{align}
&\bbn{\int_{a}^{t}e^{i(t-\tau)\Delta}(\wt{u}_{1}\wt{u}_{2}\wt{u}_{3}\wt{u}_{4}\wt{u}_{5})d\tau}_{X^{s}}\\
\lesssim& \n{u_{1}}_{Y^{s}}
\lrs{T^{\frac{1}{2(d+3)}}M_{0}^{\frac{2d+3}{3(d+3)}}\n{P_{\leq M_{0}}u_{2}}_{Y^{\frac{d-1}{2}}}+\n{P_{>M_{0}}u_{2}}_{Y^{\frac{d-1}{2}}}}\n{u_{3}}_{Y^{\frac{d-1}{2}}}
\n{u_{4}}_{Y^{\frac{d-1}{2}}}\n{u_{5}}_{Y^{\frac{d-1}{2}}}.\notag
\end{align}
\end{lemma}
\begin{remark}
Notice that $(\ref{equ:multilinear estimate d>=5 high frequency estimate})$ for $d=3$ is implied by \cite[Lemma 5.15]{chen2019the}. One can compare the proof of the stronger $L_{t}^{1}H^{s}$ estimate with the proof here and see that the proof of the weaker $U$-$V$ estimates are indeed much less technical, and the current method to incorporate these weaker estimates is indeed stronger.
\end{remark}
\begin{proof}
For the high frequency estimate $(\ref{equ:multilinear estimate d>=5 high frequency estimate})$, decompose the 6 factors into Littlewood-Paley pieces so that
\begin{align*}
I=\sum_{M_{1},M_{2},M_{3},M_{4},M_{5},M_{6}}I_{M_{1},M_{2},M_{3},M_{4},M_{5},M_{6}}
\end{align*}
where
\begin{align*}
I_{M_{1},M_{2},M_{3},M_{4},M_{5},M_{6}}=\iint_{x,t}u_{1,M_{1}}u_{2,M_{2}}u_{3,M_{3}}
u_{4,M_{4}}u_{5,M_{5}}g_{M_{6}}dxdt
\end{align*}
with $u_{j,M_{j}}=P_{M_{j}}u_{j}$ and $g_{M_{6}}=P_{M_{6}}g$.

 We first take care of the most difficult Case A. $M_{1}\sim M_{6}\geq M_{2}\geq M_{3}\geq M_{4}\geq M_{5}$. Then, we need only to deal with one such as
Case B. $M_{1}\sim M_{2}\geq  M_{6}\geq M_{3}\geq M_{4}\geq M_{5}$,
since other cases can be treated in the same way. Decompose the $M_{1}$ and $M_{6}$ dyadic spaces into $M_{2}$ size cubes, then
\begin{align*}
I_{A}\lesssim& \sum_{\substack{M_{1},M_{2},M_{3},M_{4},M_{5},M_{6}\\M_{1}\sim M_{6}\geq  M_{2}\geq M_{3}\geq M_{4}\geq M_{5}}} \sum_{Q}\n{P_{Q}u_{1,M_{1}}P_{Q_{c}}u_{2,M_{2}}
u_{3,M_{3}}u_{4,M_{4}}u_{5,M_{5}}g_{M_{6}}}_{L_{t,x}^{1}}\\
\lesssim& \sum_{\substack{M_{1},M_{2},M_{3},M_{4},M_{5},M_{6}\\M_{1}\sim M_{6}\geq  M_{2}\geq M_{3}\geq M_{4}\geq M_{5}}}\sum_{Q}\left(\n{P_{Q}u_{1,M_{1}}}_{L_{t,x}^{\frac{6(d+3)}{2d+3}}}
\n{u_{2,M_{2}}}_{L_{t,x}^{\frac{6(d+3)}{2d+3}}}
\n{u_{3,M_{3}}}_{L_{t,x}^{2(d+3)}}\right.\\
&\left.\n{u_{4,M_{4}}}_{L_{t,x}^{2(d+3)}}\n{u_{5,M_{5}}}_{L_{t,x}^{2(d+3)}}
\n{P_{Q_{c}}g_{M_{6}}}_{L_{t,x}^{\frac{6(d+3)}{2d+3}}}\right)
\end{align*}
where three factors corresponding to small size cubes (here $M_{3}$, $M_{4}$, $M_{5}$ size cubes) are put in $L_{t,x}^{2(d+3)}$ and the others are put in $L_{t,x}^{\frac{6(d+3)}{2d+3}}$. By $(\ref{equ:Strichartz estimate})$ and $(\ref{equ:strichartz estimate with noncertered frequency localization})$,
\begin{align*}
\lesssim& \sum_{\substack{M_{1},M_{2},M_{3},M_{4},M_{5},M_{6}\\M_{1}\sim M_{6}\geq  M_{2}\geq M_{3}\geq M_{4}\geq M_{5}}} \sum_{Q}\left( M_{2}^{\frac{d-1}{2}-\frac{3}{2(d+3)}}\n{P_{Q}u_{1,M_{1}}}_{Y^{0}}
\n{u_{2,M_{2}}}_{Y^{0}}
M_{3}^{\frac{d-1}{2}+\frac{1}{2(d+3)}}\n{u_{3,M_{3}}}_{Y^{0}}\right.\\
&\left.M_{4}^{\frac{d-1}{2}+\frac{1}{2(d+3)}}\n{u_{4,M_{4}}}_{Y^{0}}
M_{5}^{\frac{d-1}{2}+\frac{1}{2(d+3)}}\n{u_{5,M_{5}}}_{Y^{0}}
\n{P_{Q_{c}}g_{M_{6}}}_{Y^{0}}\right).
\end{align*}
Applying Cauchy-Schwarz to sum in $Q$,
\begin{align}\label{equ:estiamte, wanted form, quintic}
\lesssim& \sum_{\substack{M_{1},M_{6}\\M_{1}\sim M_{6}}}M_{1}^{-s}M_{6}^{s}\n{u_{1,M_{1}}}_{Y^{s}}
\n{g_{M_{6}}}_{Y^{-s}}\sum_{\substack{M_{2},M_{3},M_{4},M_{5}\\M_{2}\geq M_{3}\geq M_{4}\geq M_{5}}}
\left(M_{2}^{-\frac{3}{2 (d+3)}}M_{3}^{\frac{1}{2(d+3)}}M_{4}^{\frac{1}{2(d+3)}}M_{5}^{\frac{1}{2(d+3)}}\right.\\
&\left.\n{u_{2,M_{2}}}_{Y^{\frac{d-1}{2}}}\n{u_{3,M_{3}}}_{Y^{\frac{d-1}{2}}}
\n{u_{4,M_{4}}}_{Y^{\frac{d-1}{2}}}
\n{u_{5,M_{5}}}_{Y^{\frac{d-1}{2}}}\right)\notag
\end{align}
Supping out $\n{u_{4,M_{4}}}_{Y^{\frac{d-1}{2}}}$ and $\n{u_{5,M_{5}}}_{Y^{\frac{d-1}{2}}}$ in
$M_{4}$ and $M_{5}$, and then applying Cauchy-Schwarz as shown in $(\ref{equ:estiamte, wanted form})$,
\begin{align*}
\lesssim &\n{u_{1}}_{Y^{s}}\n{u_{2}}_{Y^{\frac{d-1}{2}}}\n{u_{3}}_{Y^{\frac{d-1}{2}}}
\n{u_{4}}_{Y^{\frac{d-1}{2}}}\n{u_{5}}_{Y^{\frac{d-1}{2}}}
\n{g}_{Y^{-s}}.
\end{align*}

Case B. $M_{1}\sim M_{2}\geq  M_{6}\geq M_{3}\geq M_{4}\geq M_{5}$. Decompose the $M_{1}$ and $M_{2}$ dyadic spaces into $M_{6}$ size cubes and we have
\begin{align*}
I_{B}\lesssim& \sum_{\substack{M_{1},M_{2},M_{3},M_{4},M_{5},M_{6}\\M_{1}\sim M_{2}\geq  M_{6}\geq M_{3}\geq M_{4}\geq M_{5}}}  \sum_{Q}\n{P_{Q}u_{1,M_{1}}P_{Q_{c}}u_{2,M_{2}}
u_{3,M_{3}}u_{4,M_{4}}u_{5,M_{5}}g_{M_{6}}}_{L_{t,x}^{1}}\\
\lesssim& \sum_{\substack{M_{1},M_{2},M_{3},M_{4},M_{5},M_{6}\\M_{1}\sim M_{2}\geq  M_{6}\geq M_{3}\geq M_{4}\geq M_{5}}}\sum_{Q}\left(\n{P_{Q}u_{1,M_{1}}}_{L_{t,x}^{\frac{6(d+3)}{2d+3}}}
\n{P_{Q_{c}}u_{2,M_{2}}}_{L_{t,x}^{\frac{6(d+3)}{2d+3}}}
\n{u_{3,M_{3}}}_{L_{t,x}^{2(d+3)}}\right.\\
&\left.\n{u_{4,M_{4}}}_{L_{t,x}^{2(d+3)}}\n{u_{5,M_{5}}}_{L_{t,x}^{2(d+3)}}
\n{g_{M_{6}}}_{L_{t,x}^{\frac{6(d+3)}{2d+3}}}\right)
\end{align*}
By $(\ref{equ:Strichartz estimate})$ and $(\ref{equ:strichartz estimate with noncertered frequency localization})$,
\begin{align*}
\lesssim& \sum_{\substack{M_{1},M_{2},M_{3},M_{4},M_{5},M_{6}\\M_{1}\sim M_{2}\geq  M_{6}\geq M_{3}\geq M_{4}\geq M_{5}}} \sum_{Q}\left( M_{6}^{\frac{d-1}{2}-\frac{3}{2(d+3)}}\n{P_{Q}u_{1,M_{1}}}_{Y^{0}}
\n{P_{Q_{c}}u_{2,M_{2}}}_{Y^{0}}
M_{3}^{\frac{d-1}{2}+\frac{1}{2(d+3)}}\n{u_{3,M_{3}}}_{Y^{0}}\right.\\
&\left.M_{4}^{\frac{d-1}{2}+\frac{1}{2(d+3)}}\n{u_{4,M_{4}}}_{Y^{0}}
M_{5}^{\frac{d-1}{2}+\frac{1}{2(d+3)}}\n{u_{5,M_{5}}}_{Y^{0}}
\n{g_{M_{6}}}_{Y^{0}}\right).
\end{align*}
Applying Cauchy-Schwarz to sum in $Q$,
\begin{align*}
\lesssim& \sum_{\substack{M_{1},M_{2},M_{3},M_{4},M_{5},M_{6}\\M_{1}\sim M_{2}\geq  M_{6}\geq M_{3}\geq M_{4}\geq M_{5}}} \left( M_{6}^{\frac{d-1}{2}-\frac{3}{2(d+3)}}\n{u_{1,M_{1}}}_{Y^{0}}
\n{u_{2,M_{2}}}_{Y^{0}}
M_{3}^{\frac{d-1}{2}+\frac{1}{2(d+3)}}\n{u_{3,M_{3}}}_{Y^{0}}\right.\\
&\left.M_{4}^{\frac{d-1}{2}+\frac{1}{2(d+3)}}\n{u_{4,M_{4}}}_{Y^{0}}
M_{5}^{\frac{d-1}{2}+\frac{1}{2(d+3)}}\n{u_{5,M_{5}}}_{Y^{0}}
\n{g_{M_{6}}}_{Y^{0}}\right)\\
\lesssim&\sum_{\substack{M_{1},M_{2},M_{3},M_{4},M_{5},M_{6}\\M_{1}\sim M_{2}\geq  M_{6}\geq M_{3}\geq M_{4}\geq M_{5}}} \left( M_{1}^{-s}\n{u_{1,M_{1}}}_{Y^{s}}
M_{2}^{-\frac{d-1}{2}}\n{u_{2,M_{2}}}_{Y^{\frac{d-1}{2}}}
M_{3}^{\frac{1}{2(d+3)}}\n{u_{3,M_{3}}}_{Y^{\frac{d-1}{2}}}\right.\\
&\left.M_{4}^{\frac{1}{2(d+3)}}\n{u_{4,M_{4}}}_{Y^{\frac{d-1}{2}}}
M_{5}^{\frac{1}{2(d+3)}}\n{u_{5,M_{5}}}_{Y^{\frac{d-1}{2}}}
M_{6}^{\frac{d-1}{2}-\frac{3}{2(d+3)}+s}\n{g_{M_{6}}}_{Y^{-s}}\right)
\end{align*}
If $s=\frac{d-1}{2}$, it can be estimated in the same way as $(\ref{equ:estiamte, wanted form, quintic})$. Thus, we need only to treat the case $\frac{d-1}{2}+s>0$. Supping out $\n{u_{3,M_{3}}}_{Y^{\frac{d-1}{2}}}$, $\n{u_{4,M_{4}}}_{Y^{\frac{d-1}{2}}}$, $\n{u_{5,M_{5}}}_{Y^{\frac{d-1}{2}}}$ and $\n{g_{M_{6}}}_{Y^{-s}}$, in
$M_{3}$, $M_{4}$, $M_{5}$ and $M_{6}$, we have
\begin{align*}
\lesssim&
\n{u_{3}}_{Y^{\frac{d-1}{2}}}
\n{u_{4}}_{Y^{\frac{d-1}{2}}}
\n{u_{5}}_{Y^{\frac{d-1}{2}}}
\n{g}_{Y^{-s}}
\sum_{\substack{M_{1},M_{2},M_{3},M_{4},M_{5},M_{6}\\M_{1}\sim M_{2}\geq  M_{6}\geq M_{3}\geq M_{4}\geq M_{5}}} \left(M_{1}^{-s}\n{u_{1,M_{1}}}_{Y^{s}}
\right.\\
&\left.
M_{2}^{-\frac{d-1}{2}}\n{u_{2,M_{2}}}_{Y^{\frac{d-1}{2}}}  M_{6}^{\frac{d-1}{2}-\frac{3}{2(d+3)}+s}
M_{3}^{\frac{1}{2(d+3)}}
M_{4}^{\frac{1}{2(d+3)}}
M_{5}^{\frac{1}{2(d+3)}}
\right)
\end{align*}
By the fact that $\frac{d-1}{2}+s>0$,
\begin{align*}
\lesssim &
\n{u_{3}}_{Y^{\frac{d-1}{2}}}
\n{u_{4}}_{Y^{\frac{d-1}{2}}}
\n{u_{5}}_{Y^{\frac{d-1}{2}}}
\n{g}_{Y^{-s}}
\sum_{\substack{M_{1},M_{2}\\M_{1}\sim M_{2}}}
M_{1}^{-s}\n{u_{1,M_{1}}}_{Y^{s}}
M_{2}^{-\frac{d-1}{2}}\n{u_{2,M_{2}}}_{Y^{\frac{d-1}{2}}}
M_{2}^{\frac{d-1}{2}+s}\\
=&\n{u_{1}}_{Y^{s}}
\n{u_{3}}_{Y^{\frac{d-1}{2}}}
\n{u_{4}}_{Y^{\frac{d-1}{2}}}
\n{u_{5}}_{Y^{\frac{d-1}{2}}}
\sum_{\substack{M_{1},M_{2}\\M_{1}\sim M_{2}}} M_{1}^{-s}\n{u_{1,M_{1}}}_{Y^{s}}
M_{2}^{s}\n{u_{2,M_{2}}}_{Y^{\frac{d-1}{2}}}
\end{align*}
Applying Cauchy-Schwarz,
\begin{align*}
\lesssim &\n{u_{1}}_{Y^{s}}\n{u_{2}}_{Y^{\frac{d-1}{2}}}\n{u_{3}}_{Y^{\frac{d-1}{2}}}
\n{u_{4}}_{Y^{\frac{d-1}{2}}}\n{u_{5}}_{Y^{\frac{d-1}{2}}}
\n{g}_{Y^{-s}}.
\end{align*}

Case B requires that $\frac{d-1}{2}+s\geq 0$. If we exchange $M_{1}$ and $M_{6}$ in Case B, we find another requirement that $\frac{d-1}{2}-s\geq 0$ is also needed. In this case, it becomes Case A again if $s=\frac{d-1}{2}$ and it becomes similar but a bit
different if $s=\frac{d-5}{2}$ as $M_{1}$ and $M_{2}$ are not symmetric.

Proof of the low frequency estimate $(\ref{equ:multilinear estimate d>=5 low frequency estimate})$.
At first, we deal with the most difficult Case A. $M_{1}\sim M_{6}\geq M_{5}\geq M_{4}\geq M_{3}\geq M_{2}$.
Decompose the $M_{1}$ and $M_{6}$ dyadic spaces into $M_{5}$ size cubes,
\begin{align*}
I_{M_{1},M_{2},M_{3},M_{4},M_{5},M_{6}}\lesssim &\sum_{Q}\n{P_{Q}u_{1,M_{1}}(P_{\leq M_{0}}u_{2,M_{2}})
u_{3,M_{3}}u_{4,M_{4}}u_{5,M_{5}}P_{Q_{c}}g_{M_{6}}}_{L_{t,x}^{1}}\\
\leq& \sum_{Q}\left(\n{P_{Q}u_{1,M_{1}}}_{L_{t,x}^{\frac{6(d+3)}{2d+3}}}
\n{P_{\leq M_{0}}u_{2,M_{2}}}_{L_{t,x}^{2(d+3)}}
\n{u_{3,M_{3}}}_{L_{t,x}^{2(d+3)}}\right.\\
&\left.\n{u_{4,M_{4}}}_{L_{t,x}^{2(d+3)}}
\n{u_{5,M_{5}}}_{L_{t,x}^{\frac{6(d+3)}{2d+3}}}
\n{P_{Q_{c}}g_{M_{6}}}_{L_{t,x}^{\frac{6(d+3)}{2d+3}}}\right)
\end{align*}
where three factors corresponding to small size cubes (here $M_{4}$, $M_{3}$, $M_{2}$ size cubes) are put in $L_{t,x}^{2(d+3)}$ and the others are put in $L_{t,x}^{\frac{6(d+3)}{2d+3}}$.

By H\"{o}lder, Bernstein inequalities and $(\ref{equ:u-v and sobolev})$,
\begin{align*}
\n{P_{\leq M_{0}}u_{2,M_{2}}}_{L_{t,x}^{2(d+3)}}\lesssim& T^{\frac{1}{2(d +3)}}M_{0}^{\frac{1}{d+3}}M_{2}^{\frac{d-1}{2}+\frac{1}{2(d+3)}}\n{P_{\leq M_{0}}u_{2,M_{2}}}_{L_{t}^{\wq}L_{x}^{2}}\\
\lesssim&T^{\frac{1}{2(d +3)}}M_{0}^{\frac{1}{d+3}}M_{2}^{\frac{d-1}{2}+\frac{1}{2(d+3)}}\n{P_{\leq M_{0}}u_{2,M_{2}}}_{Y^{0}}
\end{align*}
By $(\ref{equ:Strichartz estimate})$ and $(\ref{equ:strichartz estimate with noncertered frequency localization})$,
\begin{align*}
&I_{M_{1},M_{2},M_{3},M_{4},M_{5},M_{6}}\\
\lesssim& T^{\frac{1}{2(d+3)}}M_{0}^{\frac{1}{d+3}}
\sum_{Q}\left(
\n{P_{Q}u_{1,M_{1}}}_{Y^{0}}
M_{2}^{\frac{d-1}{2}+\frac{1}{2(d+3)}}\n{P_{\leq M_{0}}u_{2,M_{2}}}_{Y^{0}}
M_{3}^{\frac{d-1}{2}+\frac{1}{2(d+3)}}\n{u_{3,M_{3}}}_{Y^{0}}\right.\\
&\left.M_{4}^{\frac{d-1}{2}+\frac{1}{2(d+3)}}\n{u_{4,M_{4}}}_{Y^{0}}
M_{5}^{\frac{d-1}{2}-\frac{3}{2(d+3)}}\n{u_{5,M_{5}}}_{Y^{0}}
\n{P_{Q_{c}}g_{M_{6}}}_{Y^{0}}\right)
\end{align*}
Applying Cauchy-Schwarz to sum in $Q$,
\begin{align*}
I_{M_{1},M_{2},M_{3},M_{4},M_{5},M_{6}}
\lesssim& T^{\frac{1}{2(d+3)}}M_{0}^{\frac{1}{d+3}}M_{5}^{\frac{d-1}{2}-\frac{3}{2(d+3)}}
M_{4}^{\frac{d-1}{2}+\frac{1}{2(d+3)}}M_{3}^{\frac{d-1}{2}+\frac{1}{2(d+3)}}
M_{2}^{\frac{d-1}{2}+\frac{1}{2(d+3)}}\\
&\n{u_{1,M_{1}}}_{Y^{0}}\n{P_{\leq M_{0}}u_{2,M_{2}}}_{Y^{0}}
\n{u_{3,M_{3}}}_{Y^{0}}\n{u_{4,M_{4}}}_{Y^{0}}\n{u_{5,M_{5}}}_{Y^{0}}\n{g_{M_{6}}}_{Y^{0}}.
\end{align*}
Then by $M_{1}\sim M_{6}$, we obtain
\begin{align*}
&I_{M_{1},M_{2},M_{3},M_{4},M_{5},M_{6}}\\
\lesssim& T^{\frac{1}{2(d+3)}}M_{0}^{\frac{1}{d+3}}M_{5}^{-\frac{3}{2(d+3)}}
M_{4}^{\frac{1}{2(d+3)}}M_{3}^{\frac{1}{2(d+3)}}M_{2}^{\frac{1}{2(d+3)}}\\
&\n{u_{1,M_{1}}}_{Y^{s}}\n{P_{\leq M_{0}}u_{2,M_{2}}}_{Y^{\frac{d-1}{2}}}
\n{u_{3,M_{3}}}_{Y^{\frac{d-1}{2}}}\n{u_{4,M_{4}}}_{Y^{\frac{d-1}{2}}}
\n{u_{5,M_{5}}}_{Y^{\frac{d-1}{2}}}\n{g_{M_{6}}}_{Y^{-s}}.
\end{align*}
and hence
\begin{align*}
I_{A}\lesssim& T^{\frac{1}{2(d+3)}}M_{0}^{\frac{1}{d+3}}
\sum_{\substack{M_{1},M_{6}\\M_{1}\sim M_{6}}}\n{u_{1,M_{1}}}_{Y^{s}}
\n{g_{M_{6}}}_{Y^{-s}}
\sum_{\substack{M_{2},M_{3},M_{4},M_{5}\\M_{5}\geq M_{4}\geq M_{3}\geq M_{2}}}
\left(M_{5}^{-\frac{3}{2(d+3)}}
M_{4}^{\frac{1}{2(d+3)}}M_{3}^{\frac{1}{2(d+3)}}
M_{2}^{\frac{1}{2(d+3)}}\right.\\
&\left.
\n{P_{\leq M_{0}}u_{2,M_{2}}}_{Y^{\frac{d-1}{2}}}
\n{u_{3,M_{3}}}_{Y^{\frac{d-1}{2}}}\n{u_{4,M_{4}}}_{Y^{\frac{d-1}{2}}}\n{u_{5,M_{5}}}_{Y^{\frac{d-1}{2}}}
\right)
\end{align*}
Estimating it in the same way as $(\ref{equ:estiamte, wanted form, quintic})$, we have
\begin{align*}
I_{A}\lesssim T^{\frac{1}{2(d+3)}}M_{0}^{\frac{3}{2(d+3)}}\n{u_{1}}_{Y^{s}}\n{P_{\leq M_{0}}u_{2}}_{Y^{\frac{d-1}{2}}}\n{u_{3}}_{Y^{\frac{d-1}{2}}}\n{u_{4}}_{Y^{\frac{d-1}{2}}}\n{u_{5}}_{Y^{\frac{d-1}{2}}}
\n{g}_{Y^{-s}}.
\end{align*}

Next, we deal with one such as Case B. $M_{1}\sim M_{2}\geq  M_{3}\geq M_{4}\geq M_{5}\geq M_{6}$,
since other cases can be treated in the same way. Decompose the $M_{1}$ and $M_{2}$ dyadic spaces into $M_{3}$ size cubes and we have
\begin{align*}
I_{M_{1},M_{2},M_{3},M_{4},M_{5},M_{6}}\leq&
\sum_{Q}\n{P_{Q}u_{1,M_{1}}}_{L_{t,x}^{\frac{6(d+3)}{2d+3}}}
\n{P_{Q_{c}}P_{\leq M_{0}}u_{2,M_{2}}}_{L_{t,x}^{\frac{6(d+3)}{2d+3}}}
\n{u_{3,M_{3}}}_{L_{t,x}^{\frac{6(d+3)}{2d+3}}}\\
&\n{u_{4,M_{5}}}_{L_{t,x}^{2(d+3)}}
\n{u_{4,M_{5}}}_{L_{t,x}^{2(d+3)}}
\n{g_{M_{6}}}_{L_{t,x}^{2(d+3)}},
\end{align*}
By H\"{o}lder, Bernstein inequalities and $(\ref{equ:u-v and sobolev})$,
\begin{align*}
\n{P_{Q_{c}}P_{\leq M_{0}}u_{2,M_{2}}}_{L_{t,x}^{\frac{6(d+3)}{2d+3}}}\lesssim& T^{\frac{2d+3}{6(d+3)}}M_{0}^{\frac{2d+3}{3(d +3)}}M_{3}^{\frac{d-1}{6}-\frac{1}{2(d +3)}}\n{P_{Q_{c}}P_{\leq M_{0}}u_{2,M_{2}}}_{L_{t}^{\wq}L_{x}^{2}}\\
\lesssim &T^{\frac{2d+3}{6(d+3)}}M_{0}^{\frac{2d+3}{3(d +3)}}M_{3}^{\frac{d-1}{6}-\frac{1}{2(d +3)}}\n{P_{Q_{c}}P_{\leq M_{0}}u_{2,M_{2}}}_{Y^{0}}
\end{align*}
By $(\ref{equ:Strichartz estimate})$ and $(\ref{equ:strichartz estimate with noncertered frequency localization})$,
\begin{align*}
&I_{M_{1},M_{2},M_{3},M_{4},M_{5},M_{6}}\\
\lesssim &T^{\frac{2d+3}{6(d+3)}}M_{0}^{\frac{2d+3}{3(d +3)}}M_{3}^{\frac{d-1}{2}-\frac{3}{2(d+3)}}
M_{4}^{\frac{d-1}{2}+\frac{1}{2(d+3)}}
M_{5}^{\frac{d-1}{2}+\frac{1}{2(d+3)}}
M_{6}^{\frac{d-1}{2}+\frac{1}{2(d+3)}}\\
&\times \sum_{Q}\n{P_{Q}u_{1,M_{1}}}_{Y^{0}}
\n{P_{Q_{c}}P_{\leq M_{0}}u_{2,M_{2}}}_{Y^{0}}
\n{u_{3,M_{3}}}_{Y^{0}}
\n{u_{4,M_{4}}}_{Y^{0}}
\n{u_{5,M_{5}}}_{Y^{0}}
\n{g_{M_{6}}}_{Y^{0}}
\end{align*}
Applying Cauchy-Schwarz to sum in $Q$,
\begin{align*}
\lesssim &T^{\frac{2d+3}{6(d+3)}}M_{0}^{\frac{2d+3}{3(d +3)}}M_{3}^{\frac{d-1}{2}-\frac{3}{2(d+3)}}
M_{4}^{\frac{d-1}{2}+\frac{1}{2(d+3)}}
M_{5}^{\frac{d-1}{2}+\frac{1}{2(d+3)}}
M_{6}^{\frac{d-1}{2}+\frac{1}{2(d+3)}}\\
&\n{u_{1,M_{1}}}_{Y^{0}}
\n{P_{\leq M_{0}}u_{2,M_{2}}}_{Y^{0}}
\n{u_{3,M_{3}}}_{Y^{0}}
\n{u_{4,M_{4}}}_{Y^{0}}
\n{u_{5,M_{5}}}_{Y^{0}}
\n{g_{M_{6}}}_{Y^{0}}\\
\lesssim &T^{\frac{2d+3}{6(d+3)}}M_{0}^{\frac{2d+3}{3(d +3)}}M_{1}^{-s}M_{2}^{-\frac{d-1}{2}}M_{3}^{-\frac{3}{2(d+3)}}
M_{4}^{\frac{1}{2(d+3)}}
M_{5}^{\frac{1}{2(d+3)}}
M_{6}^{\frac{d-1}{2}+\frac{1}{2(d+3)}+s}\\
&\n{u_{1,M_{1}}}_{Y^{s}}
\n{P_{\leq M_{0}}u_{2,M_{2}}}_{Y^{\frac{d-1}{2}}}
\n{u_{3,M_{3}}}_{Y^{\frac{d-1}{2}}}
\n{u_{4,M_{4}}}_{Y^{\frac{d-1}{2}}}
\n{u_{5,M_{5}}}_{Y^{\frac{d-1}{2}}}
\n{g_{M_{6}}}_{Y^{-s}}
\end{align*}
If $s+\frac{d-1}{2}=0$, it can be estimated in the same way as $(\ref{equ:estiamte, wanted form, quintic})$. Thus we need only to treat the case $s+\frac{d-1}{2}>0$. Supping out $\n{u_{3,M_{3}}}_{Y^{\frac{d-1}{2}}}$, $\n{u_{4,M_{4}}}_{Y^{\frac{d-1}{2}}}$, $\n{u_{5,M_{5}}}_{Y^{\frac{d-1}{2}}}$ and $\n{g_{M_{6}}}_{Y^{-s}}$ in
$M_{3}$, $M_{4}$, $M_{5}$ and $M_{6}$, we have
\begin{align*}
I_{B}\lesssim &
T^{\frac{2d+3}{6(d+3)}}M_{0}^{\frac{2d+3}{3(d +3)}}
\n{u_{3}}_{Y^{\frac{d-1}{2}}}
\n{u_{4}}_{Y^{\frac{d-1}{2}}}
\n{u_{5}}_{Y^{\frac{d-1}{2}}}
\n{g}_{Y^{-s}}\sum_{\substack{M_{1},M_{2},M_{3},M_{4},M_{5},M_{6}\\M_{1}\sim M_{2}\geq   M_{3}\geq M_{4}\geq M_{5}\geq M_{6}}} \left(M_{1}^{-s}\right.\\
&\left.
\n{u_{1,M_{1}}}_{Y^{s}}M_{2}^{-\frac{d-1}{2}}\n{P_{\leq M_{0}}u_{2,M_{2}}}_{Y^{\frac{d-1}{2}}}
M_{3}^{-\frac{3}{2(d+3)}}
M_{4}^{\frac{1}{2(d+3)}}
M_{5}^{\frac{1}{2(d+3)}}
M_{6}^{\frac{d-1}{2}+\frac{1}{2(d+3)}+s}
\right)
\end{align*}
By the fact that $s+\frac{d-1}{2}>0$,
\begin{align*}
\lesssim &T^{\frac{2d+3}{6(d+3)}}M_{0}^{\frac{2d+3}{3(d +3)}}
\n{u_{3}}_{Y^{\frac{d-1}{2}}}
\n{u_{4}}_{Y^{\frac{d-1}{2}}}
\n{u_{5}}_{Y^{\frac{d-1}{2}}}
\n{g}_{Y^{-s}}\sum_{\substack{M_{1},M_{2}\\M_{1}\sim M_{2}}} \left(M_{1}^{-s}\n{u_{1,M_{1}}}_{Y^{s}}\right.\\
&\left.
M_{2}^{-\frac{d-1}{2}}\n{P_{\leq M_{0}}u_{2,M_{2}}}_{Y^{\frac{d-1}{2}}}
M_{2}^{\frac{d-1}{2}+s}
\right)\\
\lesssim&T^{\frac{2d+3}{6(d+3)}}M_{0}^{\frac{2d+3}{3(d +3)}}
\n{u_{3}}_{Y^{\frac{d-1}{2}}}
\n{u_{4}}_{Y^{\frac{d-1}{2}}}
\n{u_{5}}_{Y^{\frac{d-1}{2}}}
\n{g}_{Y^{-s}}
\sum_{\substack{M_{1},M_{2}\\M_{1}\sim M_{2}}}
\n{u_{1,M_{1}}}_{Y^{s}}\n{P_{\leq M_{0}}u_{2,M_{2}}}_{Y^{\frac{d-1}{2}}}
\end{align*}
Applying Cauchy-Schwarz,
\begin{align*}
\lesssim& T^{\frac{2d+3}{6(d+3)}}M_{0}^{\frac{2d+3}{3(d +3)}}\n{u_{1}}_{Y^{s}}\n{P_{\leq M_{0}}u_{2}}_{Y^{\frac{d-1}{2}}}\n{u_{3}}_{Y^{\frac{d-1}{2}}}\n{u_{4}}_{Y^{\frac{d-1}{2}}}\n{u_{5}}_{Y^{\frac{d-1}{2}}}
\n{g}_{Y^{-s}}.
\end{align*}

\end{proof}

\appendix
\section{Miscellaneous Lemmas}
We provide the following lemmas under the $\T^{d}$ setting, as they work the same for the $\R^{d}$ case with the homogeneous norm.
\begin{lemma}\label{lemma:uniqueness from tensor form}
Let $u_{1}$ and $u_{2}$ be the $C([0,T_{0}];H^{s_{c}})$ solutions to $(\ref{equ:NLS})$ with the same initial datum such that
\begin{align}\label{equ:uniqueness from tensor form}
u_{1}(t,x)\ol{u}_{1}(t,x')=u_{2}(t,x)\ol{u}_{2}(t,x').
\end{align}
Then $u_{1}(t,x)=u_{2}(t,x)$.
\end{lemma}
\begin{proof}
% Here, as we will treat the $\R^{d}$ case, we give a proof for the homogeneous norm $\dot{H}^{s_{c}}$, which can also be replaced by the inhomogeneous norm $H^{s_{c}}$.
From the proof of Corollary \ref{cor:uniqueness for nls}, we have obtained the uniqueness for the trivial solution $u\equiv 0$, so we might as well assume that $u_{1}(t)\neq 0$ for all $t\in [0,T_{0}]$.
On the other hand, we note that
\begin{align}
\lra{\nabla}^{s_{c}}u_{1}(t,x)\n{\lra{\nabla}^{s_{c}}u_{1}(t)}_{L^{2}}^{2}=\lra{\nabla}^{s_{c}}u_{2}(t,x)
\lra{\lra{\nabla}^{s_{c}}u_{2}(t),\lra{\nabla}^{s_{c}}u_{1}(t)}
\end{align}
which implies that
\begin{align} \label{equ:a(t)}
\lra{\nabla}^{s_{c}}u_{1}(t)=a(t)\lra{\nabla}^{s_{c}}u_{2}(t),
\end{align}
where
$$a(t)=\frac{\lra{\lra{\nabla}^{s_{c}}u_{2}(t),\lra{\nabla}^{s_{c}}u_{1}(t)}}
{\n{\lra{\nabla}^{s_{c}}u_{1}(t)}_{L^{2}}^{2}}.$$
Since $u_{1}\in C([0,T_{0}];H^{s_{c}})$, we have that
\begin{align}
c_{0}:=\inf_{t\in [0,T_{0}]}\n{\lra{\nabla}^{s_{c}}u_{1}}_{L^{2}}>0,
\end{align}
which implies that $a(t)$ is well-defined. We are left to prove $a(t)=1$ for every $t\in [0,T_{0}]$. Taking differences gives that
\begin{align}
P_{\leq M}(u_{2}-u_{1})=-i\int_{0}^{t}e^{i(t-\tau)\Delta}P_{\leq M}(|u_{1}|^{p-1}
u_{1})(\tau,x)(a(\tau)-1)d\tau,
\end{align}
where we used $(\ref{equ:a(t)})$ for $u_{2}$.

On the one hand, by $(\ref{equ:a(t)})$ and the UTFL property in Lemma $\ref{lemma:UTFL}$, we obtain
\begin{align}\label{equ:lower bound, a(t)}
\n{P_{\leq M}(u_{2}-u_{1})}_{H^{s_{c}}}=\n{P_{\leq M}(a(t)-1)u_{1}}_{H^{s_{c}}}\geq \frac{c_{0}|a(t)-1|}{2}.
\end{align}
On the other hand, by $\n{P_{\leq M}\lra{\nabla}^{s}f}_{L^{2}}\lesssim M^{s}\n{P_{\leq M}f}_{L^{2}}$ and Sobolev embedding $\lrs{\ref{lemma:sobolev embedding critical case}}$ and $(\ref{equ:sobolev inequality, multilinear estimate, endpoint})$, we get
\begin{align}\label{equ:upper bound, a(t)}
&\bbn{\int_{0}^{t}e^{i(t-\tau)\Delta}P_{\leq M}(|u_{1}|^{p-1}
u_{1})(\tau,x)(a(\tau)-1)d\tau}_{H^{s_{c}}}\\
\lesssim &\int_{0}^{t}|a(\tau)-1|\n{P_{\leq M}(|u_{1}|^{p-1}u_{1})}_{H^{s_{c}}}d\tau\notag\\
\lesssim &M^{2}\int_{0}^{t}|a(\tau)-1|\n{P_{\leq M}(|u_{1}|^{p-1}u_{1})}_{H^{s_{c}-2}}d\tau\notag\\
\lesssim &M^{2}\int_{0}^{t}|a(\tau)-1|\n{u_{1}}_{H^{s_{c}}}^{p}d\tau\notag\\
\leq &M^{2}C_{0}^{p}\int_{0}^{t}|a(\tau)-1|d\tau.\notag
\end{align}
Combining estimates $(\ref{equ:lower bound, a(t)})$ and $(\ref{equ:upper bound, a(t)})$, we have
\begin{align}
|a(t)-1|\lesssim \int_{0}^{t}|a(\tau)-1|d\tau
\end{align}
which implies that $a(t)\equiv1$ by Gronwall's inequality.
\end{proof}

\begin{lemma}
%Let $s\in[0,1]$ and $1<r,r_{1},r_{2},q_{1},q_{2}<\wq$ such that
%$\frac{1}{r}=\frac{1}{r_{i}}+\frac{1}{q_{i}}$ for $i=1,2$. Then,
%\begin{align}
%\n{|\nabla|^{s}(fg)}_{L^{r}}\lesssim
%\n{f}_{L^{r_{1}}}\n{|\nabla|^{s}g}_{L^{q_{1}}}+
%\n{|\nabla|^{s}f}_{L^{r_{2}}}\n{g}_{L^{q_{2}}}
%\end{align}

%\begin{align}
%\n{fg}_{\dot{H}^{s}(\R^{d})}\lesssim \n{f}_{\dot{H}^{s+s_{1}}(\R^{d})}\n{g}_{\dot{H}^{s_{2}}(\R^{d})}
%+\n{f}_{\dot{H}^{\wt{s}_{1}}(\R^{d})}\n{g}_{\dot{H}^{s+\wt{s}_{2}}(\R^{d})}
%\end{align}
\begin{align}\label{equ:sharp sobolev inequality two appendix}
\n{fg}_{H^{s}(\T^{d})}\lesssim \n{f}_{H^{s+s_{1}}(\T^{d})}\n{g}_{H^{s_{2}}(\T^{d})}
+\n{f}_{H^{\wt{s}_{1}}(\T^{d})}\n{g}_{H^{s+\wt{s}_{2}}(\T^{d})}
\end{align}
where $s\geq 0$, $s_{i}>0$, $\wt{s}_{i}>0$, $s_{1}+s_{2}=\frac{d}{2}$ and $\wt{s}_{1}+\wt{s}_{2}=\frac{d}{2}$.
\end{lemma}
\begin{proof}
This has certainly been studied by many authors. For completeness, we include a proof.
We note that
$$P_{N}(P_{<N-1}fP_{<N-1}g)=0$$
for $N\geq 2$ and hence
$$P_{N}(fg)=P_{N}\lrc{(P_{\geq N-1}f)g+(P_{<N-1}f)(P_{\geq N-1}g)}.$$
We expand
\begin{align*}
\n{fg}_{H^{s}(\T^{d})}^{2}\simeq& \sum_{N=0}\lra{N}^{2s}\n{P_{N}(fg)}_{L^{2}}^{2}\\
\lesssim& \n{fg}_{L^{2}}^{2}+\sum_{N=2}\lra{N}^{2s}\n{P_{N}\lrc{(P_{\geq N-1}f)g+(P_{<N-1}f)(P_{\geq N-1}g)}}_{L^{2}}^{2}\\
\lesssim&\n{fg}_{L^{2}}^{2}+I^{2}+II^{2}
\end{align*}
where
\begin{align*}
&I=\n{\lra{N}^{s}P_{N}\lrc{(P_{\geq N-1}f)g}}_{l^{2}L^{2}},\\
&II=\n{\lra{N}^{s}P_{N}\lrc{(P_{<N-1}f)(P_{\geq N-1}g)}}_{l^{2}L^{2}}.
\end{align*}
For $II$, by H\"{o}lder inequality, we have
\begin{align*}
II=&\n{\lra{N}^{s}P_{N}\lrc{(P_{<N-1}f)(P_{\geq N-1}g)}}_{l^{2}L^{2}}\\
\leq&\n{\lra{N}^{s}(P_{<N-1}f)(P_{\geq N-1}g)}_{l^{2}L^{2}}\\
\leq &\n{P_{<N-1}f}_{l^{\wq}L^{p_{1}}}\n{\lra{N}^{s}P_{\geq N-1}g}_{l^{2}L^{p_{2}}},
\end{align*}
where $\frac{1}{2}=\frac{1}{p_{1}}+\frac{1}{p_{2}}$. Then by Sobolev inequality,
\begin{align*}
II\lesssim &\n{f}_{H^{s_{1}}}\n{\lra{N}^{s}\lra{\nabla}^{s_{2}}P_{\geq N-1}g}_{l^{2}L^{2}}\\
= &\n{f}_{H^{s_{1}}}\bbn{\sum_{M\geq N}\lra{N}^{s}\lra{M}^{-s}\lra{\nabla}^{s_{2}}\lra{M}^{s}P_{M}g}_{L^{2}l^{2}}
\end{align*}
where $s_{i}\in(0,\frac{d}{2})$ for $i=1$, $2$. By Young's inequality,
\begin{align*}
II\lesssim&\n{f}_{H^{s_{1}}}\n{\lra{N}^{s}\lra{\nabla}^{s_{2}}P_{N}g}_{L^{2}l^{2}}\\
\lesssim &\n{f}_{H^{s_{1}}}\n{g}_{H^{s+s_{2}}}
\end{align*}
$I$ can be estimated in the same way as $II$.
\end{proof}

\begin{lemma}[Sobolev embedding]\label{lemma:sobolev embedding critical case}
\begin{align}
&\n{f_{1}f_{2}f_{3}}_{H^{s}(\T^{d})}\lesssim \prod_{j=1}^{3}\n{f_{j}}_{H^{\frac{s+d}{3}}(\T^{d})} \label{equ:sobolev inequality,trilinear estimate, endpoint}\\
&\n{f_{1}f_{2}f_{3}f_{4}f_{5}}_{H^{s}(\T^{d})}\lesssim \prod_{j=1}^{5}\n{f_{j}}_{H^{\frac{s+2d}{5}}(\T^{d})}\label{equ:sobolev inequality, multilinear estimate, endpoint}
\end{align}
for $s\in(-\frac{d}{2},\frac{d}{2})$.
\end{lemma}
\begin{proof}
For $s=0$, it follows from H\"{o}lder and Sobolev inequalities.

For $s\in (-\frac{d}{2},0)$, by duality, we have
\begin{align*}
&\n{f_{1}f_{2}f_{3}}_{H^{s}}\lesssim \n{f_{1}f_{2}f_{3}}_{L^{\frac{2d}{d-2s}}},\\
&\n{f_{1}f_{2}f_{3}f_{4}f_{5}}_{H^{s}}\lesssim \n{f_{1}f_{2}f_{3}f_{4}f_{5}}_{L^{\frac{2d}{d-2s}}}.
\end{align*}
Then by H\"{o}lder inequality and the Sobolev embedding,
\begin{align*}
&\n{f_{1}f_{2}f_{3}}_{H^{s}}\lesssim \prod_{j=1}^{3}\n{f_{j}}_{L^{\frac{6d}{d-2s}}}\lesssim  \prod_{j=1}^{3}\n{f_{j}}_{H^{\frac{s+d}{3}}(\T^{d})},\\
&\n{f_{1}f_{2}f_{3}f_{4}f_{5}}_{H^{s}}\lesssim \prod_{j=1}^{5}\n{f_{j}}_{L^{\frac{10d}{d-2s}}}\lesssim  \prod_{j=1}^{5}\n{f_{j}}_{H^{\frac{s+2d}{5}}(\T^{d})}.
\end{align*}

%For $s\in (0,\frac{d}{2})$, we use the periodic Kato-Ponce inequality $(\ref{equ:periodic kato-ponce inequality appendix})$. Taking $f=f_{1}f_{2}$ and $g=f_{3}$ with $r=2$, $p_{1}=\frac{3d}{d+s}$, $q_{1}=\frac{6d}{d-2s}$,
%$p_{2}=\frac{3d}{d-2s}$ and $q_{2}=\frac{6d}{d+4s}$, we have
%\begin{align*}
%\n{\lra{\nabla}^{s}(f_{1}f_{2} f_{3})}_{L^{2}}\lesssim &\n{\lra{\nabla}^{s}(f_{1}f_{2})}_{L^{\frac{3d}{d+s}}}\n{f_{3}}_{L^{\frac{6d}{d-2s}}}
%+\n{f_{1}f_{2}}_{L^{\frac{3d}{d-2s}}}\n{\lra{\nabla}^{s}f_{3}}_{L^{\frac{6d}{d+4s}}}
%\end{align*}
%For $\n{\lra{\nabla}^{s}(f_{1}f_{2})}_{L^{\frac{3d}{d+s}}}$, using it again with $r=\frac{3d}{d+s}$, $p_{1}=\frac{6d}{d+4s}$, $q_{1}=\frac{6d}{d-2s}$, $p_{2}=\frac{6d}{d-2s}$, $q_{2}=\frac{6d}{d+4s}$, we obtain
%\begin{align*}
%\n{\lra{\nabla}^{s}(f_{1}f_{2})}_{L^{\frac{3d}{d+s}}}\lesssim \n{\lra{\nabla}^{s}f_{1}}_{L^{\frac{6d}{d+4s}}}\n{f_{2}}_{L^{\frac{6d}{d-2s}}}+
%\n{f_{1}}_{L^{\frac{6d}{d-2s}}}\n{\lra{\nabla}^{s}f_{2}}_{L^{\frac{6d}{d+4s}}}.
%\end{align*}
%Then by Sobolev embedding,
%\begin{align*}
%\n{\lra{\nabla}^{s}(f_{1}f_{2}f_{3})}_{L^{2}}\lesssim \prod_{j=1}^{3}\n{f_{j}}_{H^{\frac{s+d}{3}}(\T^{d})}.
%\end{align*}

For $s\in (0,\frac{d}{2})$, we use Sobolev inequality $(\ref{equ:sharp sobolev inequality two appendix})$. Taking $f=f_{1}f_{2}$ and $g=f_{3}$ with $s_{1}=\frac{d-2s}{6}$, $s_{2}=\frac{s+d}{3}$, $\wt{s}_{1}=\frac{d+4s}{6}$ and $\wt{s}_{2}=\frac{d-2s}{3}$, we have
\begin{align*}
\n{f_{1}f_{2}f_{3}}_{H^{s}}\lesssim
\n{f_{1}f_{2}}_{H^{\frac{d+4s}{6}}}\n{f_{3}}_{H^{\frac{s+d}{3}}}.
\end{align*}
For $\n{f_{1}f_{2}}_{H^{\frac{d+4s}{6}}}$, using it again with $s_{1}=\frac{d-2s}{6}$, $s_{2}=\frac{s+d}{3}$, $\wt{s}_{1}=\frac{s+d}{3}$ and $\wt{s}_{2}=\frac{d-2s}{6}$, we obtain
\begin{align*}
\n{f_{1}f_{2}}_{H^{\frac{d+4s}{6}}}\lesssim \n{f_{1}}_{H^{\frac{s+d}{3}}}\n{f_{2}}_{H^{\frac{s+d}{3}}}.
\end{align*}

Taking $f=f_{1}f_{2}$ and $g=f_{3}f_{4}f_{5}$ with $s_{1}=\frac{3(d-2s)}{10}$, $s_{2}=\frac{3s+d}{5}$, $\wt{s}_{1}=\frac{4s+3d}{10}$ and $\wt{s}_{2}=\frac{d-2s}{5}$, we obtain
\begin{align*}
\n{f_{1}f_{2}f_{3}f_{4}f_{5}}_{H^{s}}\lesssim&
\n{f_{1}f_{2}}_{H^{s+s_{1}}}\n{f_{3}f_{4}f_{5}}_{H^{s_{2}}}+\n{f_{1}f_{2}}_{H^{\wt{s}_{1}}}
\n{f_{3}f_{4}f_{5}}_{H^{s+\wt{s}_{2}}}\\
\lesssim& \prod_{j=1}^{2}\n{f_{j}}_{H^{\frac{2(s+s_{1})+d}{4}}}\prod_{j=3}^{5}
\n{f_{j}}_{H^{\frac{s_{2}+d}{3}}}+\prod_{j=1}^{2}\n{f_{j}}_{H^{\frac{2\wt{s}_{1}+d}{4}}}
\prod_{j=3}^{5}\n{f_{j}}_{\frac{s+\wt{s}_{2}+d}{3}}\\
\lesssim& \prod_{j=1}^{5}\n{f_{j}}_{H^{\frac{s+2d}{5}}}.
\end{align*}

\end{proof}
\section{Results for Some $H^{1}$-subcritical Cases}
%In the energy-subcritical case, a regularity threshold, that is,
%\begin{align}
%s_{e}:=\frac{d(p-2)}{2p},
%\end{align}
%is needed for the embedding $H^{s}\hookrightarrow L^{p}$ so that the nonlinearity
Note that the proof of Theorem $\ref{thm:uniqueness for nls}$ works uniformly in all dimensions, $d\geq 4$ for quintic case and $d\geq 5$ for cubic case. For completeness, we present some results for low dimensions using our method. As we are limited by the Sobolev embedding in Lemma \ref{lemma:sobolev embedding critical case}, the regularity requirements are higher than the critical scaling exponent $s_{c}$. Certainly, it is still an open problem to push $s$ down to $s_{c}$ for $H^{1}$-subcritical problems in both $\R^{d}$ and $\T^{d}$.
\begin{theorem}\label{thm:uniqueness for energy-subcritical nls}
~\\
$(a)$. There is at most one $C([0,T_{0}];\dot{H}^{\frac{d}{4}}(\Lambda^{d}))$ solution to $(\ref{equ:NLS})$ where $p=3$ and $d=2,3$.
~\\
$(b)$. There is at most one $C([0,T_{0}];\dot{H}^{\frac{2}{3}}(\Lambda^{2}))$ solution to $(\ref{equ:NLS})$ where $p=5$.
\end{theorem}

\begin{lemma} \label{lemma:trilinear estimate d=2,3}
On $\T^{2}$,
\begin{align}
&\iint_{x,t} \wt{u}_{1}(t,x)\wt{u}_{2}(t,x)\wt{u}_{3}(t,x)\wt{g}(t,x)dtdx\lesssim T^{\frac{1}{2}}\n{u_{1}}_{Y^{-\frac{1}{2}}}
\n{u_{2}}_{Y^{\frac{1}{2}}}\n{u_{3}}_{Y^{\frac{1}{2}}}\n{g}_{Y^{\frac{1}{2}}},\label{equ:trilinear estimate,d=2}\\
&\iint_{x,t} \wt{u}_{1}(t,x)\wt{u}_{2}(t,x)\wt{u}_{3}(t,x)\wt{g}(t,x)dtdx\lesssim T^{\frac{1}{2}}\n{u_{1}}_{Y^{\frac{1}{2}}}
\n{u_{2}}_{Y^{\frac{1}{2}}}\n{u_{3}}_{Y^{\frac{1}{2}}}\n{g}_{Y^{-\frac{1}{2}}}.
\end{align}
On $\T^{3}$,
\begin{align}
\iint_{x,t} \wt{u}_{1}(t,x)\wt{u}_{2}(t,x)\wt{u}_{3}(t,x)\wt{g}(t,x)dtdx\lesssim T^{\frac{1}{4}}\n{u_{1}}_{Y^{-\frac{3}{4}}}
\n{u_{2}}_{Y^{\frac{3}{4}}}\n{u_{3}}_{Y^{\frac{3}{4}}}\n{g}_{Y^{\frac{3}{4}}},
\label{equ:trilinear estimate,d=3}\\
\iint_{x,t} \wt{u}_{1}(t,x)\wt{u}_{2}(t,x)\wt{u}_{3}(t,x)\wt{g}(t,x)dtdx\lesssim T^{\frac{1}{4}}\n{u_{1}}_{Y^{\frac{3}{4}}}
\n{u_{2}}_{Y^{\frac{3}{4}}}\n{u_{3}}_{Y^{\frac{3}{4}}}\n{g}_{Y^{-\frac{3}{4}}}
\end{align}
\end{lemma}
\begin{proof}
By the symmetry of $u_{1}$ and $g$, it suffices to prove $(\ref{equ:trilinear estimate,d=2})$ and
$(\ref{equ:trilinear estimate,d=3})$. For simplicity, we take $\wt{u}=u$ and $\wt{g}=g$.
Decompose the 4 factors into Littlewood-Paley pieces so that
\begin{align*}
I=\sum_{M_{1},M_{2},M_{3},M_{4}}I_{M_{1},M_{2},M_{3},M_{4}}
\end{align*}
where
\begin{align*}
I_{M_{1},M_{2},M_{3},M_{4}}=\iint_{x,t}u_{1,M_{1}}u_{2,M_{2}}u_{3,M_{3}}g_{M_{4}}dxdt
\end{align*}
with $u_{j,M_{j}}=P_{M_{j}}u_{j}$ and $g_{M_{4}}=P_{M_{4}}g$.

It suffices to consider the most difficult case A. $M_{1}\sim M_{2}\geq M_{3}\geq M_{4}$ while other cases can be dealt with in a similar way. Decompose the $M_{1}$ and $M_{2}$ dyadic into $M_{3}$ size cubes.
%Due to the frequency constraint $\xi_{2}\sim -(\xi_{1}+\xi_{3}+\xi)$, for each choice $Q$ of an $M_{3}$ size cubes within the $\xi_{1}$ space, the variable $\xi_{2}$
%is constrained to at most $3^{d}$ of $M_{3}$-cubes. For
%convenience, we denote these cubes by a single cube $Q_{c}$
%corresponding to $Q$.
\begin{align*}
I_{A}\lesssim& \sum_{\substack{M_{1},M_{2},M_{3},M_{4}\\ M_{1}\sim M_{2}\geq
M_{3}\geq M_{4}}} \sum_{Q}\n{P_{Q}u_{1,M_{1}}P_{Q_{c}}u_{2,M_{2}}
u_{3,M_{3}}g_{M_{4}}}_{L_{t,x}^{1}}\\
\lesssim& \sum_{\substack{M_{1},M_{2},M_{3},M_{4}\\ M_{1}\sim M_{2}\geq
M_{3}\geq M_{4}}}\sum_{Q}T^{\frac{1}{2}}\n{P_{Q}u_{1,M_{1}}}_{L_{t}^{\wq}L_{x}^{2}}\n{P_{Q_{c}}u_{2,M_{2}}}_{L_{t,x}^{5}}
\n{u_{3,M_{3}}}_{L_{t,x}^{5}}\n{g_{M_{4}}}_{L_{t,x}^{10}}\\
\end{align*}
By $(\ref{lemma:Strichartz estimate})$and $(\ref{lemma:strichartz estimate with noncertered frequency localization})$,
\begin{align*}
\lesssim& \sum_{\substack{M_{1},M_{2},M_{3},M_{4}\\ M_{1}\sim M_{2}\geq
M_{3}\geq M_{4}}}\sum_{Q}T^{\frac{1}{2}}\n{P_{Q}u_{1,M_{1}}}_{Y^{0}}M_{3}^{\frac{1}{5}}\n{P_{Q_{c}}u_{2,M_{2}}}_{Y^{0}}
M_{3}^{\frac{1}{5}}\n{u_{3,M_{3}}}_{Y^{0}}M_{4}^{\frac{3}{5}}\n{g_{M_{4}}}_{Y^{0}}
\end{align*}
Applying Cauchy-Schwarz to sum in $Q$,
\begin{align*}
\lesssim& T^{\frac{1}{2}}\sum_{\substack{M_{1},M_{2}\\M_{1}\sim M_{2}}}M_{1}^{\frac{1}{2}}M_{2}^{-\frac{1}{2}}\n{u_{1,M_{1}}}_{Y^{-\frac{1}{2}}}
\n{u_{2,M_{2}}}_{Y^{\frac{1}{2}}}
\sum_{\substack{M_{3},M_{4}\\M_{1}\sim M_{2}\geq M_{3}\geq M_{4}}}M_{3}^{-\frac{1}{10}}M_{4}^{\frac{1}{10}}\n{u_{3,M_{3}}}_{Y^{\frac{1}{2}}}
\n{g_{M_{4}}}_{Y^{\frac{1}{2}}}
\end{align*}
Applying Cauchy-Schwarz,
\begin{align*}
\lesssim& T^{\frac{1}{2}}\lrs{\sum_{M_{1}}\n{u_{1,M_{1}}}_{Y^{-\frac{1}{2}}}^{2}}^{\frac{1}{2}}
\lrs{\sum_{M_{2}}\n{u_{2,M_{2}}}_{Y^{\frac{1}{2}}}^{2}}^{\frac{1}{2}}\\
&\lrs{\sum_{\substack{M_{3},M_{4}\\ M_{3}\geq M_{4}}}\lrs{\frac{M_{4}}{M_{3}}}^{\frac{1}{10}}\n{u_{3,M_{3}}}_{Y^{\frac{1}{2}}}^{2}}^{\frac{1}{2}}
\lrs{\sum_{\substack{M_{3},M_{4}\\ M_{3}\geq M_{4}}}\lrs{\frac{M_{4}}{M_{3}}}^{\frac{1}{10}}\n{g_{M_{4}}}_{Y^{\frac{1}{2}}}^{2}}^{\frac{1}{2}}\\
\lesssim& T^{\frac{1}{2}}\n{u_{1}}_{Y^{-\frac{1}{2}}}
\n{u_{2}}_{Y^{\frac{1}{2}}}\n{u_{3}}_{Y^{\frac{1}{2}}}\n{g}_{Y^{\frac{1}{2}}}.
\end{align*}

For $d=3$, we have that
\begin{align*}
I_{A}\lesssim& \sum_{\substack{M_{1},M_{2},M_{3},M_{4}\\ M_{1}\sim M_{2}\geq
M_{3}\geq M_{4}}} \sum_{Q}\n{P_{Q}u_{1,M_{1}}P_{Q_{c}}u_{2,M_{2}}
u_{3,M_{3}}g_{M_{4}}}_{L_{t,x}^{1}}\\
\lesssim& \sum_{\substack{M_{1},M_{2},M_{3},M_{4}\\ M_{1}\sim M_{2}\geq
M_{3}\geq M_{4}}}\sum_{Q}\n{P_{Q}u_{1,M_{1}}}_{L_{t,x}^{4}}\n{P_{Q_{c}}u_{2,M_{2}}}_{L_{t,x}^{4}}
\n{u_{3,M_{3}}}_{L_{t}^{\wq}L_{x}^{2}}\n{g_{M_{4}}}_{L_{t}^{2}L_{x}^{\wq}}\\
\end{align*}
By $(\ref{lemma:strichartz estimate with noncertered frequency localization})$ and Bernstein,
\begin{align*}
\lesssim& \sum_{\substack{M_{1},M_{2},M_{3},M_{4}\\ M_{1}\sim M_{2}\geq
M_{3}\geq M_{4}}}\sum_{Q}M_{3}^{\frac{1}{4}}\n{P_{Q}u_{1,M_{1}}}_{Y^{0}}M_{3}^{\frac{1}{4}}\n{P_{Q_{c}}u_{2,M_{2}}}_{Y^{0}}
\n{u_{3,M_{3}}}_{Y^{0}}T^{\frac{1}{4}}M_{4}\n{g_{M_{4}}}_{Y^{0}}
\end{align*}
Applying Cauchy-Schwarz to sum in $Q$,
\begin{align*}
\lesssim& T^{\frac{1}{4}}\sum_{\substack{M_{1},M_{2}\\M_{1}\sim M_{2}}}M_{1}^{\frac{3}{4}}M_{2}^{-\frac{3}{4}}\n{u_{1,M_{1}}}_{Y^{-\frac{3}{4}}}
\n{u_{2,M_{2}}}_{Y^{\frac{3}{4}}}
\sum_{\substack{M_{3},M_{4}\\M_{1}\sim M_{2}\geq M_{3}\geq M_{4}}}M_{3}^{-\frac{1}{4}}M_{4}^{\frac{1}{4}}\n{u_{3,M_{3}}}_{Y^{\frac{3}{4}}}
\n{g_{M_{4}}}_{Y^{\frac{3}{4}}}\\
\lesssim&T^{\frac{1}{4}}\n{u_{1}}_{Y^{-\frac{3}{4}}}
\n{u_{2}}_{Y^{\frac{3}{4}}}\n{u_{3}}_{Y^{\frac{3}{4}}}\n{g}_{Y^{\frac{3}{4}}}.
\end{align*}

\end{proof}

\begin{lemma}
On $\T^{2}$,
\begin{align}
&\iint_{x,t} \wt{u}_{1}(t,x)\wt{u}_{2}(t,x)\wt{u}_{3}(t,x)\wt{u}_{4}(t,x)\wt{u}_{5}(t,x)\wt{g}(t,x)dtdx
\label{equ:quintilinear estimate,d=2}\\
&\lesssim T^{\frac{1}{3}}\n{u_{1}}_{Y^{-\frac{2}{3}}}
\n{u_{2}}_{Y^{\frac{2}{3}}}\n{u_{3}}_{Y^{\frac{2}{3}}}\n{u_{4}}_{Y^{\frac{2}{3}}}
\n{u_{5}}_{Y^{\frac{2}{3}}}\n{g}_{Y^{\frac{2}{3}}}.\notag
\end{align}
\begin{align}
&\iint_{x,t} \wt{u}_{1}(t,x)\wt{u}_{2}(t,x)\wt{u}_{3}(t,x)\wt{u}_{4}(t,x)\wt{u}_{5}(t,x)\wt{g}(t,x)dtdx\\
&\lesssim T^{\frac{1}{3}}\n{u_{1}}_{Y^{\frac{2}{3}}}
\n{u_{2}}_{Y^{\frac{2}{3}}}\n{u_{3}}_{Y^{\frac{2}{3}}}\n{u_{4}}_{Y^{\frac{2}{3}}}
\n{u_{5}}_{Y^{\frac{2}{3}}}\n{g}_{Y^{-\frac{2}{3}}}.\notag
\end{align}
\end{lemma}
\begin{proof}
Decompose the 6 factors into Littlewood-Paley pieces so that
\begin{align*}
I=\sum_{M_{1},M_{2},M_{3},M_{4},M_{5},M_{6}}I_{M_{1},M_{2},M_{3},M_{4},M_{5},M_{6}}
\end{align*}
where
\begin{align*}
I_{M_{1},M_{2},M_{3},M_{4},M_{5},M_{6}}=\iint_{x,t}u_{1,M_{1}}u_{2,M_{2}}u_{3,M_{3}}
u_{4,M_{4}}u_{5,M_{5}}g_{M_{6}}dxdt
\end{align*}
with $u_{j,M_{j}}=P_{M_{j}}u_{j}$ and $g_{M_{6}}=P_{M_{6}}g$.

In the same way as trilinear estimates in Lemma $\ref{lemma:trilinear estimate d=2,3}$, it suffices to take care of the most difficult case A. $M_{1}\sim M_{6}\geq M_{2}\geq M_{3}\geq M_{4}\geq M_{5}$. Decompose the $M_{1}$ and $M_{6}$ dyadic spaces into $M_{2}$ size cubes, then
\begin{align*}
I_{A}\lesssim& \sum_{\substack{M_{1},M_{2},M_{3},M_{4},M_{5},M_{6}\\M_{1}\sim M_{6}\geq  M_{2}\geq M_{3}\geq M_{4}\geq M_{5}}} \sum_{Q}\n{P_{Q}u_{1,M_{1}}P_{Q_{c}}u_{2,M_{2}}
u_{3,M_{3}}u_{4,M_{4}}u_{5,M_{5}}g_{M_{6}}}_{L_{t,x}^{1}}\\
\lesssim& \sum_{\substack{M_{1},M_{2},M_{3},M_{4},M_{5},M_{6}\\M_{1}\sim M_{6}\geq  M_{2}\geq M_{3}\geq M_{4}\geq M_{5}}}\sum_{Q}
\left(\n{P_{Q}u_{1,M_{1}}}_{L_{t,x}^{\frac{9}{2}}}
\n{u_{2,M_{2}}}_{L_{t,x}^{\frac{9}{2}}}
\n{u_{3,M_{3}}}_{L_{t,x}^{9}}\right.\\
&\left.\n{u_{4,M_{4}}}_{L_{t,x}^{9}}
\n{u_{5,M_{5}}}_{L_{t,x}^{9}}
\n{P_{Q_{c}}g_{M_{6}}}_{L_{t,x}^{\frac{9}{2}}}\right)
\end{align*}
By $(\ref{lemma:Strichartz estimate})$and $(\ref{lemma:strichartz estimate with noncertered frequency localization})$,
\begin{align*}
\lesssim &\sum_{\substack{M_{1},M_{2},M_{3},M_{4},M_{5},M_{6}\\M_{1}\sim M_{6}\geq  M_{2}\geq M_{3}\geq M_{4}\geq M_{5}}}\sum_{Q}
\left(
M_{2}^{\frac{1}{9}}\n{P_{Q}u_{1,M_{1}}}_{Y^{0}}
M_{2}^{\frac{1}{9}}\n{P_{Q_{c}}u_{2,M_{2}}}_{Y^{0}}\right.\\
&\left.T^{\frac{1}{9}}M_{3}^{\frac{7}{9}}\n{u_{3,M_{3}}}_{Y^{0}}
T^{\frac{1}{9}}M_{4}^{\frac{7}{9}}\n{u_{4,M_{4}}}_{Y^{0}}
T^{\frac{1}{9}}M_{5}^{\frac{7}{9}}\n{u_{5,M_{5}}}_{Y^{0}}
M_{2}^{\frac{1}{9}}\n{g_{M_{6}}}_{Y^{0}}\right)
\end{align*}
Applying Cauchy-Schwarz to sum in $Q$,
\begin{align*}
\lesssim &T^{\frac{1}{3}}\sum_{\substack{M_{1},M_{6}\\ M_{1}\sim M_{6}}}
M_{1}^{\frac{2}{3}}M_{6}^{-\frac{2}{3}}
\n{u_{1,M_{1}}}_{Y^{-\frac{2}{3}}}\n{g_{M_{6}}}_{Y^{\frac{2}{3}}}\\
&\sum_{\substack{M_{2},M_{3},M_{4},M_{5}\\ M_{2}\geq M_{3}\geq M_{4}\geq M_{5}}}
M_{2}^{-\frac{1}{3}}M_{3}^{\frac{1}{9}}M_{4}^{\frac{1}{9}}M_{5}^{\frac{1}{9}}
\n{u_{2,M_{2}}}_{Y^{\frac{2}{3}}}
\n{u_{3,M_{3}}}_{Y^{\frac{2}{3}}}
\n{u_{4,M_{4}}}_{Y^{\frac{2}{3}}}
\n{u_{5,M_{5}}}_{Y^{\frac{2}{3}}}\\
\lesssim &
T^{\frac{1}{3}}\n{u_{1}}_{Y^{-\frac{2}{3}}}
\n{u_{2}}_{Y^{\frac{2}{3}}}\n{u_{3}}_{Y^{\frac{2}{3}}}\n{u_{4}}_{Y^{\frac{2}{3}}}
\n{u_{5}}_{Y^{\frac{2}{3}}}\n{g}_{Y^{\frac{2}{3}}}.
\end{align*}

\end{proof}

\section{A More Usual Proof for the $\R^{d}$ Case}
With the dual Strichartz estimate and the existence of a better solution, we could give a more usual proof of the unconditional uniqueness under the energy-supercritical setting for $\R^{d}$ case. Such an argument has been used by many authors and we summarize it below, but it does not work for the $\T^{d}$ case. For simplicity, we prove it for the cubic case, as it works the same for the quintic case. At first, we need the following lemmas.
\begin{lemma}[Strichartz Estimate]\label{lemma:strichartz estimate}
Let $I$ be a compact time interval, and let $u:I\times \R^{3}\mapsto \mathbb{C}$ be a Schwartz solution to the forced Schr\"{o}dinger equation
\begin{align*}
i\pa_{t}u+\Delta u=\sum_{m=1}^{M}F_{m}
\end{align*}
for some Schwartz functions $F_{1}$,..., $F_{m}$. Then
\begin{align*}
\n{|\nabla|^{s}u}_{L_{t}^{q}L_{x}^{r}}\lesssim  \n{|\nabla|^{s}u_{0}}_{L_{x}^{2}}+\sum_{m=1}^{M}
\n{|\nabla|^{s}F_{m}}_{L_{t}^{q_{m}'}L_{x}^{r_{m}'}}
\end{align*}
for $s\geq 0$ and any admissible exponents $(q_{i},r_{i})$ for $i=1,2,...,m$, where $p'$ denotes the dual exponent to $p$.
\end{lemma}
\begin{lemma}[Leibniz Rule \cite{gulisashvili1996exact}]\label{lemma:Leibniz rule}
Let $s\geq 0$ and $1<r,r_{1},r_{2},q_{1},q_{2}<\wq$ such that
$\frac{1}{r}=\frac{1}{r_{i}}+\frac{1}{q_{i}}$ for $i=1,2$. Then,
\begin{align*}
\n{|\nabla|^{s}(fg)}_{L^{r}}\lesssim \n{f}_{L^{r_{1}}}\n{|\nabla|^{s}g}_{L^{q_{1}}}+
\n{|\nabla|^{s}f}_{L^{r_{2}}}\n{g}_{L^{q_{2}}}.
\end{align*}
\end{lemma}

%\begin{lemma}[Fractional Chain Rule]\label{lemma:fractional chain rule}
%Suppose $G\in C^{1}(\mathbb{C})$, $s\geq 0$, and $1<q<q_{1}<q_{2}<\wq$ are such that
%$\frac{1}{q}=\frac{1}{q_{1}}+\frac{1}{q_{2}}$. Then,
%\begin{align*}
%\n{|\nabla|^{s}G(u)}_{L^{q}}\lesssim \n{G'(u)}_{L^{q_{1}}}\n{|\nabla|^{s}u}_{L^{q_{2}}}.
%\end{align*}
%\end{lemma}

 Let $u$ be a maximal-lifespan solution constructed in \cite{killip2010energy} and $v$ be a $C([0,T);\dot{H}^{s_{c}})$ solution to NLS with the same initial datum. We write $w=v-u$ and observe that
$w$ obeys a difference equation, which we write in integral form as
\begin{align}
w(t)=&-i\int_{0}^{t}e^{i(t-s)\Delta}(|u+w|^{2}(u+w)(s)-|u|^{2}u(s))ds\\
=&\int_{0}^{t}e^{i(t-s)\Delta}\sum_{j=0}^{2}{\O}(u^{j}(s)w^{3-j}(s))ds,\notag
\end{align}
where ${\O(u^{j}w^{3-j})}$ is a finite linear combination of expressions which
could be possibly replaced by their complex conjugates.

By Sobolev inequality, we have
\begin{align*}
\n{|\nabla|^{s_{c}-1}w}_{L_{x}^{\frac{2d}{d-2}}}\lesssim \n{|\nabla|^{s_{c}}w}_{L^{2}},
\end{align*}
and hence $\n{|\nabla|^{s_{c}-1}w}_{L_{t}^{2}L_{x}^{\frac{2d}{d-2}}}$ is finite. Then
by Strichartz estimate in Lemma $\ref{lemma:strichartz estimate}$, we have that
\begin{align*}
\n{|\nabla|^{s_{c}-1}w}_{L_{t}^{2}L_{x}^{\frac{2d}{d-2}}}\lesssim&
\sum_{j=0}^{1}\n{|\nabla|^{s_{c}-1}(u^{j}w^{3-j})}_{L_{t}^{2}L_{x}^{\frac{2d}{d+2}}}
+\n{|\nabla|^{s_{c}-1}(u^{2}w)}_{L_{t}^{1}L_{x}^{2}}\\
:=&I_{0}+I_{1}+I_{2}.
\end{align*}

For $I_{0}$, by Lemma $\ref{lemma:Leibniz rule}$, we have that
\begin{align*}
I_{0}=\n{|\nabla|^{s_{c}-1}(|w|^{2}w)}_{L_{t}^{2}L_{x}^{\frac{2d}{d+2}}}
\lesssim & \n{|\nabla|^{s_{c}-1}w}_{L_{t}^{2}L_{x}^{\frac{2d}{d-2}}}\n{w}_{L_{t}^{\wq}L_{x}^{d}}^{2}
\end{align*}
Then by Sobolev inequality,
\begin{align*}
\lesssim\n{|\nabla|^{s_{c}-1}w}_{L_{t}^{2}L_{x}^{\frac{2d}{d-2}}}
\n{|\nabla|^{s_{c}}w}_{L_{t}^{\wq}L_{x}^{2}}^{2}.
\end{align*}

For $I_{1}$, by Lemma $\ref{lemma:Leibniz rule}$, we have that
\begin{align*}
I_{1}=&\n{|\nabla|^{s_{c}-1}(uw^{2})}_{L_{t}^{2}L_{x}^{\frac{2d}{d+2}}}\\
\lesssim&
\n{|\nabla|^{s_{c}-1}u}_{L_{t}^{\wq}L_{x}^{\frac{2d}{d-2}}}\n{w^{2}}_{L_{t}^{2}L_{x}^{\frac{d}{2}}}
+\n{u}_{L_{t}^{\wq}L_{x}^{d}}\n{|\nabla|^{s_{c}-1}(w^{2})}_{L_{t}^{2}L_{x}^{2}}\\
\lesssim& \n{|\nabla|^{s_{c}-1}u}_{L_{t}^{\wq}L_{x}^{\frac{2d}{d-2}}}\n{w^{2}}_{L_{t}^{2}L_{x}^{\frac{d}{2}}}
+\n{u}_{L_{t}^{\wq}L_{x}^{d}}\n{|\nabla|^{s_{c}-1}w}_{L_{t}^{2}L_{x}^{\frac{2d}{d-2}}}
\n{w}_{L_{t}^{\wq}L_{x}^{d}}
\end{align*}
By H\"{o}lder and Sobolev inequalities,
\begin{align*}
\lesssim
\n{|\nabla|^{s_{c}}u}_{L_{t}^{\wq}L_{x}^{2}}\n{|\nabla|^{s_{c}-1}w}_{L_{t}^{2}L_{x}^{\frac{2d}{d-2}}}
\n{|\nabla|^{s_{c}}w}_{L_{t}^{\wq}L_{x}^{2}}.
\end{align*}

For $I_{2}$, by Lemma \ref{lemma:Leibniz rule}, we obtain
\begin{align*}
I_{2}=&\n{|\nabla|^{s_{c}-1}(u^{2}w)}_{L_{t}^{1}L_{x}^{2}}\\
\lesssim& \n{|\nabla|^{s_{c-1}}(u^{2})}_{L_{t}^{2}L_{x}^{\frac{2d}{d-2}}}\n{w}_{L_{t}^{2}L_{x}^{d}}
+\n{u}_{L_{t}^{4}L_{x}^{2d}}^{2}\n{|\nabla|^{s_{c}-1}w}_{L_{t}^{2}L_{x}^{\frac{2d}{d-2}}}\\
\lesssim& \n{|\nabla|^{s_{c-1}}u}_{L_{t}^{4}L_{x}^{\frac{2d}{d-3}}}\n{u}_{L_{t}^{4}L_{x}^{2d}}
\n{w}_{L_{t}^{2}L_{x}^{d}}
+\n{u}_{L_{t}^{4}L_{x}^{2d}}^{2}\n{|\nabla|^{s_{c}-1}w}_{L_{t}^{2}L_{x}^{\frac{2d}{d-2}}}
\end{align*}
By Sobolev inequality,
\begin{align*}
\lesssim&\n{|\nabla|^{s_{c}}u}_{L_{t}^{4}L_{x}^{\frac{2d}{d-1}}}^{2}
\n{|\nabla|^{s_{c}-1}w}_{L_{t}^{2}L_{x}^{\frac{2d}{d-2}}}.
\end{align*}
where $(4,\frac{2d}{d-1})$ is a Strichartz pair.

Together with the above estimates, we get
\begin{align*}
&\n{|\nabla|^{s_{c}-1}w}_{L_{t}^{2}L_{x}^{\frac{2d}{d-2}}}\\
\lesssim&
\n{|\nabla|^{s_{c}-1}w}_{L_{t}^{2}L_{x}^{\frac{2d}{d-2}}}
\left(\n{|\nabla|^{s_{c}}w}_{L_{t}^{\wq}L_{x}^{2}}^{2}+\n{|\nabla|^{s_{c}}u}_{L_{t}^{\wq}L_{x}^{2}}
\n{|\nabla|^{s_{c}}w}_{L_{t}^{\wq}L_{x}^{2}}
+\n{|\nabla|^{s_{c}}u}_{L_{t}^{4}L_{x}^{\frac{2d}{d-1}}}^{2}\right)
\end{align*}
Note that $w\in C_{t}^{0}\dot{H}^{s_{c}}$ and $w(0)=0$, so we can ensure $\n{|\nabla|^{s_{c}}w}_{L_{t}^{\wq}L_{x}^{2}(I\times \R^{d})}\leq \ve$ by choosing $I$ sufficiently small.
Also, from the Strichartz analysis in \cite[Theorem 3.1 and Remarks. 1.]{killip2010energy} , $|\nabla|^{s_{c}}u$ has finite $S^{0}$ norm, and in
particular it has finite $L_{t}^{4}L_{x}^{\frac{2d}{d-1}}$. Thus we can also ensure that
$\n{|\nabla|^{s_{c}}u}_{L_{t}^{4}L_{x}^{\frac{2d}{d-1}}}(I\times \R^{d})\leq
\ve$ by choosing $I$ sufficiently small. From our choice of $I$, we have
\begin{align*}
\n{|\nabla|^{s_{c}-1}w}_{L_{t}^{2}L_{x}^{\frac{2d}{d-2}}}\leq
C\ve\n{|\nabla|^{s_{c}-1}w}_{L_{t}^{2}L_{x}^{\frac{2d}{d-2}}},
\end{align*}
which implies that $w$ vanishes identically on $I\times \R^{d}$ provided that $\ve$ is sufficiently small.\\

\noindent
\textbf{Acknowledgements.}
The authors are deeply indebted to the referees for their thorough checking of the paper and their fine comments and suggestions. X. Chen was supported in part by NSF grant DMS-2005469 and a Simons Fellowship and Z. Zhang was
supported in part by NSF of China under Grant 12171010.

\bibliographystyle{abbrv}
%\bibliographystyle{plain}
%\nocite{*}
\bibliography{references}

\begin{thebibliography}{10}

\bibitem{babin2011on}
A.~V. Babin, A.~A. Ilyin, and E.~S. Titi.
\newblock On the regularization mechanism for the periodic {K}orteweg-de
  {V}ries equation.
\newblock {\em Comm. Pure Appl. Math.}, 64(5):591--648, 2011.

\bibitem{bourgain1993fourier}
J.~Bourgain.
\newblock Fourier transform restriction phenomena for certain lattice subsets
  and applications to nonlinear evolution equations. {I}. {S}chr\"{o}dinger
  equations.
\newblock {\em Geom. Funct. Anal.}, 3(2):107--156, 1993.

\bibitem{bourgain1999global}
J.~Bourgain.
\newblock Global wellposedness of defocusing critical nonlinear
  {S}chr\"{o}dinger equation in the radial case.
\newblock {\em J. Amer. Math. Soc.}, 12(1):145--171, 1999.

\bibitem{bourgain2015the}
J.~Bourgain and C.~Demeter.
\newblock The proof of the {$l^2$} decoupling conjecture.
\newblock {\em Ann. of Math. (2)}, 182(1):351--389, 2015.

\bibitem{cazenave1990the}
T.~Cazenave and F.~B. Weissler.
\newblock The {C}auchy problem for the critical nonlinear {S}chr\"{o}dinger
  equation in {$H^s$}.
\newblock {\em Nonlinear Anal.}, 14(10):807--836, 1990.

\bibitem{chen2015unconditional}
T.~Chen, C.~Hainzl, N.~Pavlovi\'{c}, and R.~Seiringer.
\newblock Unconditional uniqueness for the cubic {G}ross-{P}itaevskii hierarchy
  via quantum de {F}inetti.
\newblock {\em Comm. Pure Appl. Math.}, 68(10):1845--1884, 2015.

\bibitem{chen2010on}
T.~Chen and N.~Pavlovi\'{c}.
\newblock On the {C}auchy problem for focusing and defocusing
  {G}ross-{P}itaevskii hierarchies.
\newblock {\em Discrete Contin. Dyn. Syst.}, 27(2):715--739, 2010.

\bibitem{chen2011the}
T.~Chen and N.~Pavlovi\'{c}.
\newblock The quintic {NLS} as the mean field limit of a boson gas with
  three-body interactions.
\newblock {\em J. Funct. Anal.}, 260(4):959--997, 2011.

\bibitem{chen2013a}
T.~Chen and N.~Pavlovi\'{c}.
\newblock A new proof of existence of solutions for focusing and defocusing
  {G}ross-{P}itaevskii hierarchies.
\newblock {\em Proc. Amer. Math. Soc.}, 141(1):279--293, 2013.

\bibitem{chen2014derivation}
T.~Chen and N.~Pavlovi\'{c}.
\newblock Derivation of the cubic {NLS} and {G}ross-{P}itaevskii hierarchy from
  manybody dynamics in {$d=3$} based on spacetime norms.
\newblock {\em Ann. Henri Poincar\'{e}}, 15(3):543--588, 2014.

\bibitem{chen2014higher}
T.~Chen and N.~Pavlovi\'{c}.
\newblock Higher order energy conservation and global well-posedness of
  solutions for {G}ross-{P}itaevskii hierarchies.
\newblock {\em Comm. Partial Differential Equations}, 39(9):1597--1634, 2014.

\bibitem{chen2010energy}
T.~Chen, N.~Pavlovi\'{c}, and N.~Tzirakis.
\newblock Energy conservation and blowup of solutions for focusing
  {G}ross-{P}itaevskii hierarchies.
\newblock {\em Ann. Inst. H. Poincar\'{e} Anal. Non Lin\'{e}aire},
  27(5):1271--1290, 2010.

\bibitem{chen2016collapsing}
X.~Chen.
\newblock Collapsing estimates and the rigorous derivation of the 2d cubic
  nonlinear {S}chr\"{o}dinger equation with anisotropic switchable quadratic
  traps.
\newblock {\em J. Math. Pures Appl. (9)}, 98(4):450--478, 2012.

\bibitem{chen2013on}
X.~Chen.
\newblock On the rigorous derivation of the 3{D} cubic nonlinear
  {S}chr\"{o}dinger equation with a quadratic trap.
\newblock {\em Arch. Ration. Mech. Anal.}, 210(2):365--408, 2013.

\bibitem{chen2016correlation}
X.~Chen and J.~Holmer.
\newblock Correlation structures, many-body scattering processes, and the
  derivation of the {G}ross-{P}itaevskii hierarchy.
\newblock {\em Int. Math. Res. Not. IMRN}, 2016(10):3051--3110, 2016.

\bibitem{chen2016focusing}
X.~Chen and J.~Holmer.
\newblock Focusing quantum many-body dynamics: the rigorous derivation of the
  1{D} focusing cubic nonlinear {S}chr\"{o}dinger equation.
\newblock {\em Arch. Ration. Mech. Anal.}, 221(2):631--676, 2016.

\bibitem{chen2016on}
X.~Chen and J.~Holmer.
\newblock On the {K}lainerman-{M}achedon conjecture for the quantum {BBGKY}
  hierarchy with self-interaction.
\newblock {\em J. Eur. Math. Soc. (JEMS)}, 18(6):1161--1200, 2016.

\bibitem{chen2017focusing}
X.~Chen and J.~Holmer.
\newblock Focusing quantum many-body dynamics, {II}: {T}he rigorous derivation
  of the 1{D} focusing cubic nonlinear {S}chr\"{o}dinger equation from 3{D}.
\newblock {\em Anal. PDE}, 10(3):589--633, 2017.

\bibitem{chen2019the}
X.~Chen and J.~Holmer.
\newblock The derivation of the {$\Bbb T^3$} energy-critical {NLS} from quantum
  many-body dynamics.
\newblock {\em Invent. Math.}, 217(2):433--547, 2019.

\bibitem{chen2020unconditional}
X.~Chen and J.~Holmer.
\newblock Unconditional uniqueness for the energy-critical nonlinear
  {S}chr\"{o}dinger equation on {$\Bbb T^4$}.
\newblock {\em Forum Math. Pi}, 10:Paper No. e3, 49, 2022.

\bibitem{chen2014on}
X.~Chen and P.~Smith.
\newblock On the unconditional uniqueness of solutions to the infinite radial
  {C}hern-{S}imons-{S}chr\"{o}dinger hierarchy.
\newblock {\em Anal. PDE}, 7(7):1683--1712, 2014.

\bibitem{colliander2008global}
J.~Colliander, M.~Keel, G.~Staffilani, H.~Takaoka, and T.~Tao.
\newblock Global well-posedness and scattering for the energy-critical
  nonlinear {S}chr\"{o}dinger equation in {$\Bbb R^3$}.
\newblock {\em Ann. of Math. (2)}, 167(3):767--865, 2008.

\bibitem{dodson2019defocusing}
B.~Dodson.
\newblock {\em Defocusing nonlinear {S}chr\"{o}dinger equations}, volume 217 of
  {\em Cambridge Tracts in Mathematics}.
\newblock Cambridge University Press, Cambridge, 2019.

\bibitem{erdos2006derivation}
L.~Erd\H{o}s, B.~Schlein, and H.-T. Yau.
\newblock Derivation of the {G}ross-{P}itaevskii hierarchy for the dynamics of
  {B}ose-{E}instein condensate.
\newblock {\em Comm. Pure Appl. Math.}, 59(12):1659--1741, 2006.

\bibitem{erdos2007derivation}
L.~Erd\H{o}s, B.~Schlein, and H.-T. Yau.
\newblock Derivation of the cubic non-linear {S}chr\"{o}dinger equation from
  quantum dynamics of many-body systems.
\newblock {\em Invent. Math.}, 167(3):515--614, 2007.

\bibitem{erdos2007rigorous}
L.~Erd\H{o}s, B.~Schlein, and H.-T. Yau.
\newblock {Rigorous Derivation of the Gross-Pitaevskii Equation}.
\newblock {\em Phys. Rev. Lett.}, 98:040404, Jan 2007.

\bibitem{erdos2009rigorous}
L.~Erd\H{o}s, B.~Schlein, and H.-T. Yau.
\newblock Rigorous derivation of the {G}ross-{P}itaevskii equation with a large
  interaction potential.
\newblock {\em J. Amer. Math. Soc.}, 22(4):1099--1156, 2009.

\bibitem{erdos2010derivation}
L.~Erd\H{o}s, B.~Schlein, and H.-T. Yau.
\newblock Derivation of the {G}ross-{P}itaevskii equation for the dynamics of
  {B}ose-{E}instein condensate.
\newblock {\em Ann. of Math. (2)}, 172(1):291--370, 2010.

\bibitem{furioli2003unconditional}
G.~Furioli, F.~Planchon, and E.~Terraneo.
\newblock Unconditional well-posedness for semilinear {S}chr\"{o}dinger and
  wave equations in {$H^s$}.
\newblock In {\em Harmonic analysis at {M}ount {H}olyoke ({S}outh {H}adley,
  {MA}, 2001)}, volume 320 of {\em Contemp. Math.}, pages 147--156. Amer. Math.
  Soc., Providence, RI, 2003.

\bibitem{glassey1977on}
R.~T. Glassey.
\newblock On the blowing up of solutions to the {C}auchy problem for nonlinear
  {S}chr\"{o}dinger equations.
\newblock {\em J. Math. Phys.}, 18(9):1794--1797, 1977.

\bibitem{gressman2014on}
P.~Gressman, V.~Sohinger, and G.~Staffilani.
\newblock On the uniqueness of solutions to the periodic 3{D}
  {G}ross-{P}itaevskii hierarchy.
\newblock {\em J. Funct. Anal.}, 266(7):4705--4764, 2014.

\bibitem{grillakis2000on}
M.~G. Grillakis.
\newblock On nonlinear {S}chr\"{o}dinger equations.
\newblock {\em Comm. Partial Differential Equations}, 25(9-10):1827--1844,
  2000.

\bibitem{gulisashvili1996exact}
A.~Gulisashvili and M.~A. Kon.
\newblock Exact smoothing properties of {S}chr\"{o}dinger semigroups.
\newblock {\em Amer. J. Math.}, 118(6):1215--1248, 1996.

\bibitem{guo2013normal}
Z.~Guo, S.~Kwon, and T.~Oh.
\newblock Poincar\'{e}-{D}ulac normal form reduction for unconditional
  well-posedness of the periodic cubic {NLS}.
\newblock {\em Comm. Math. Phys.}, 322(1):19--48, 2013.

\bibitem{herr2016the}
S.~Herr and V.~Sohinger.
\newblock The {G}ross-{P}itaevskii hierarchy on general rectangular tori.
\newblock {\em Arch. Ration. Mech. Anal.}, 220(3):1119--1158, 2016.

\bibitem{herr2019unconditional}
S.~Herr and V.~Sohinger.
\newblock Unconditional uniqueness results for the nonlinear {S}chr\"{o}dinger
  equation.
\newblock {\em Commun. Contemp. Math.}, 21(7):1850058, 33, 2019.

\bibitem{herr2011global}
S.~Herr, D.~Tataru, and N.~Tzvetkov.
\newblock Global well-posedness of the energy-critical nonlinear
  {S}chr\"{o}dinger equation with small initial data in {$H^1(\Bbb T^3)$}.
\newblock {\em Duke Math. J.}, 159(2):329--349, 2011.

\bibitem{herr2014strichartz}
S.~Herr, D.~Tataru, and N.~Tzvetkov.
\newblock Strichartz estimates for partially periodic solutions to
  {S}chr\"{o}dinger equations in {$4d$} and applications.
\newblock {\em J. Reine Angew. Math.}, 690:65--78, 2014.

\bibitem{hilton1991catalan}
P.~Hilton and J.~Pedersen.
\newblock Catalan numbers, their generalization, and their uses.
\newblock {\em Math. Intelligencer}, 13(2):64--75, 1991.

\bibitem{hong2015unconditional}
Y.~Hong, K.~Taliaferro, and Z.~Xie.
\newblock Unconditional uniqueness of the cubic {G}ross-{P}itaevskii hierarchy
  with low regularity.
\newblock {\em SIAM J. Math. Anal.}, 47(5):3314--3341, 2015.

\bibitem{hong2016uniqueness}
Y.~Hong, K.~Taliaferro, and Z.~Xie.
\newblock Uniqueness of solutions to the 3{D} quintic {G}ross-{P}itaevskii
  hierarchy.
\newblock {\em J. Funct. Anal.}, 270(1):34--67, 2016.

\bibitem{ionescu2012the}
A.~D. Ionescu and B.~Pausader.
\newblock The energy-critical defocusing {NLS} on {$\Bbb T^3$}.
\newblock {\em Duke Math. J.}, 161(8):1581--1612, 2012.

\bibitem{kato1995on}
T.~Kato.
\newblock On nonlinear {S}chr\"{o}dinger equations. {II}. {$H^s$}-solutions and
  unconditional well-posedness.
\newblock {\em J. Anal. Math.}, 67:281--306, 1995.

\bibitem{kato1996correction}
T.~Kato.
\newblock Correction to: ``{O}n nonlinear {S}chr\"{o}dinger equations. {II}.
  {$H^s$}-solutions and unconditional well-posedness''.
\newblock {\em J. Anal. Math.}, 68:305, 1996.

\bibitem{kenig2006global}
C.~E. Kenig and F.~Merle.
\newblock Global well-posedness, scattering and blow-up for the
  energy-critical, focusing, non-linear {S}chr\"{o}dinger equation in the
  radial case.
\newblock {\em Invent. Math.}, 166(3):645--675, 2006.

\bibitem{killip2016scale}
R.~Killip and M.~Vi\c{s}an.
\newblock Scale invariant {S}trichartz estimates on tori and applications.
\newblock {\em Math. Res. Lett.}, 23(2):445--472, 2016.

\bibitem{killip2010energy}
R.~Killip and M.~Visan.
\newblock Energy-supercritical {NLS}: critical {$\dot H^s$}-bounds imply
  scattering.
\newblock {\em Comm. Partial Differential Equations}, 35(6):945--987, 2010.

\bibitem{kirkpatrick2011derivation}
K.~Kirkpatrick, B.~Schlein, and G.~Staffilani.
\newblock Derivation of the two-dimensional nonlinear {S}chr\"{o}dinger
  equation from many body quantum dynamics.
\newblock {\em Amer. J. Math.}, 133(1):91--130, 2011.

\bibitem{kishimoto2021unconditional}
N.~Kishimoto.
\newblock Unconditional local well-posedness for periodic {NLS}.
\newblock {\em J. Differential Equations}, 274:766--787, 2021.

\bibitem{kishimoto2019unconditional}
N.~Kishimoto.
\newblock Unconditional uniqueness for the periodic modified benjaminono
  equation by normal form approach.
\newblock {\em Int. Math. Res. Not. IMRN (to appear)}, 2021.

\bibitem{klainerman1993space}
S.~Klainerman and M.~Machedon.
\newblock Space-time estimates for null forms and the local existence theorem.
\newblock {\em Comm. Pure Appl. Math.}, 46(9):1221--1268, 1993.

\bibitem{klainerman2008on}
S.~Klainerman and M.~Machedon.
\newblock On the uniqueness of solutions to the {G}ross-{P}itaevskii hierarchy.
\newblock {\em Comm. Math. Phys.}, 279(1):169--185, 2008.

\bibitem{koch2014dispersive}
H.~Koch, D.~Tataru, and M.~Vi\c{s}an.
\newblock {\em Dispersive equations and nonlinear waves}, volume~45 of {\em
  Oberwolfach Seminars}.
\newblock Birkh\"{a}user/Springer, Basel, 2014.
\newblock Generalized Korteweg-de Vries, nonlinear Schr\"{o}dinger, wave and
  Schr\"{o}dinger maps.

\bibitem{kwon2020normal}
S.~Kwon, T.~Oh, and H.~Yoon.
\newblock Normal form approach to unconditional well-posedness of nonlinear
  dispersive {PDE}s on the real line.
\newblock {\em Ann. Fac. Sci. Toulouse Math. (6)}, 29(3):649--720, 2020.

\bibitem{lewin2014derivation}
M.~Lewin, P.~T. Nam, and N.~Rougerie.
\newblock Derivation of {H}artree's theory for generic mean-field {B}ose
  systems.
\newblock {\em Adv. Math.}, 254:570--621, 2014.

\bibitem{mendelson2019poisson}
D.~Mendelson, A.~R. Nahmod, N.~Pavlovi{\'c}, M.~Rosenzweig, and G.~Staffilani.
\newblock Poisson commuting energies for a system of infinitely many bosons.
\newblock {\em arXiv preprint arXiv:1910.06959}, 2019.

\bibitem{mendelson2020a}
D.~Mendelson, A.~R. Nahmod, N.~Pavlovi\'{c}, M.~Rosenzweig, and G.~Staffilani.
\newblock A rigorous derivation of the {H}amiltonian structure for the
  nonlinear {S}chr\"{o}dinger equation.
\newblock {\em Adv. Math.}, 365:107054, 115, 2020.

\bibitem{mendelson2019an}
D.~Mendelson, A.~R. Nahmod, N.~Pavlovi\'{c}, and G.~Staffilani.
\newblock An infinite sequence of conserved quantities for the cubic
  {G}ross-{P}itaevskii hierarchy on {$\Bbb{R}$}.
\newblock {\em Trans. Amer. Math. Soc.}, 371(7):5179--5202, 2019.

\bibitem{merle2019blow}
F.~Merle, P.~Rapha\"{e}l, I.~Rodnianski, and J.~Szeftel.
\newblock On blow up for the energy super critical defocusing nonlinear
  {S}chr\"{o}dinger equations.
\newblock {\em Invent. Math.}, 227(1):247--413, 2022.

\bibitem{molinet2018unconditional}
L.~Molinet, D.~Pilod, and S.~Vento.
\newblock Unconditional uniqueness for the modified {K}orteweg--de {V}ries
  equation on the line.
\newblock {\em Rev. Mat. Iberoam.}, 34(4):1563--1608, 2018.

\bibitem{mosincat2020unconditional}
R.~Mosincat and H.~Yoon.
\newblock Unconditional uniqueness for the derivative nonlinear
  {S}chr\"{o}dinger equation on the real line.
\newblock {\em Discrete Contin. Dyn. Syst.}, 40(1):47--80, 2020.

\bibitem{ryckman2007global}
E.~Ryckman and M.~Visan.
\newblock Global well-posedness and scattering for the defocusing
  energy-critical nonlinear {S}chr\"{o}dinger equation in {$\Bbb R^{1+4}$}.
\newblock {\em Amer. J. Math.}, 129(1):1--60, 2007.

\bibitem{shen2021the}
S.~Shen.
\newblock The rigorous derivation of the {$\Bbb{T}^2$} focusing cubic {NLS}
  from 3{D}.
\newblock {\em J. Funct. Anal.}, 280(8):108934, 72, 2021.

\bibitem{sohinger2015a}
V.~Sohinger.
\newblock A rigorous derivation of the defocusing cubic nonlinear
  {S}chr\"{o}dinger equation on {$\Bbb{T}^3$} from the dynamics of many-body
  quantum systems.
\newblock {\em Ann. Inst. H. Poincar\'{e} Anal. Non Lin\'{e}aire},
  32(6):1337--1365, 2015.

\bibitem{sohinger2016local}
V.~Sohinger.
\newblock Local existence of solutions to randomized {G}ross-{P}itaevskii
  hierarchies.
\newblock {\em Trans. Amer. Math. Soc.}, 368(3):1759--1835, 2016.

\bibitem{sohinger2015randomization}
V.~Sohinger and G.~Staffilani.
\newblock Randomization and the {G}ross-{P}itaevskii hierarchy.
\newblock {\em Arch. Ration. Mech. Anal.}, 218(1):417--485, 2015.

\bibitem{spohn1980kinetic}
H.~Spohn.
\newblock Kinetic equations from {H}amiltonian dynamics: {M}arkovian limits.
\newblock {\em Rev. Modern Phys.}, 52(3):569--615, 1980.

\bibitem{xie2015derivation}
Z.~Xie.
\newblock Derivation of a nonlinear {S}chr\"{o}dinger equation with a general
  power-type nonlinearity in {$d=1,2$}.
\newblock {\em Differential Integral Equations}, 28(5-6):455--504, 2015.

\bibitem{zhou1997uniqueness}
Y.~Zhou.
\newblock Uniqueness of weak solution of the {K}d{V} equation.
\newblock {\em Int. Math. Res. Not. IMRN}, (6):271--283, 1997.

\end{thebibliography}

\end{document}